\documentclass[11pt, oneside]{amsart}
\usepackage[lmargin=1in,rmargin=1in, bmargin=1in, tmargin=1in]{geometry}
\usepackage{amsmath,amssymb,amsthm}
\usepackage{tikz-cd}
\usepackage{tikz}
\usepackage{todonotes}
\usepackage{mathrsfs}
\usepackage{extarrows}
\usepackage{graphicx}
\usepackage{colonequals}
\usepackage{stmaryrd}

\usepackage[breaklinks, pagebackref]{hyperref}
\usepackage{IEEEtrantools}
\usepackage{mathrsfs}
\usepackage[utf8]{inputenc}
\usepackage[english]{babel}
 
\usepackage{comment}
\usepackage{csquotes}

\usepackage[shortlabels]{enumitem}

\usepackage{bbm}

\newcommand{\defeq}{\colonequals}

\renewcommand{\tilde}{\widetilde}  

\DeclareMathOperator{\Gal}{Gal}

\numberwithin{equation}{subsection}

\newtheorem{theorem}[equation]{Theorem}
\newtheorem{proposition}[equation]{Proposition}
\newtheorem{lemma}[equation]{Lemma}
\newtheorem{corollary}[equation]{Corollary}
\newtheorem*{corollary*}{Corollary}

\newtheorem{assumption}[equation]{Assumption}

\newcounter{alphalabels}

\newtheorem{theoremx}[alphalabels]{Theorem}

\newtheorem{conjecturex}[alphalabels]{Conjecture}

\newtheorem{corollaryx}[alphalabels]{Corollary}

\theoremstyle{definition}

\newtheorem{definition}[equation]{Definition}
\newtheorem{notation}[equation]{Notation}
\newtheorem{convention}[equation]{Convention}

\theoremstyle{remark}
\newtheorem{remark}[equation]{Remark}
\newtheorem{example}[equation]{Example}

\DeclareFontFamily{U}{wncy}{}
\DeclareFontShape{U}{wncy}{m}{n}{<->wncyr10}{}
\DeclareSymbolFont{mcy}{U}{wncy}{m}{n}
\DeclareMathSymbol{\sha}{\mathord}{mcy}{"58}

\newcommand{\ide}[1]{\mathfrak{#1}}
\newcommand{\mbb}[1]{\mathbb{#1}}
\newcommand{\ordd}{\mathcal{O}}
\newcommand{\invs}[1]{\mathcal{#1}}
\newcommand{\opn}[1]{\operatorname{#1}}
\newcommand{\hatot}{\hat{\otimes}}
\newcommand{\mbf}[1]{\mathbf{#1}}

\newcommand{\tbyt}[4]{\left( \begin{array}{cc} #1 & #2 \\ #3 & #4 \end{array} \right)}

\newcommand{\Addresses}{{% additional braces for segregating \footnotesize
  \bigskip
  \footnotesize

  \textsc{Institut de Math\'{e}matique d'Orsay, Universit\'{e} Paris-Saclay, F-91405 Orsay Cedex, France}\par\nopagebreak
  \textit{E-mail address}: \texttt{andrew.graham@universite-paris-saclay.fr}

}}

\newcommand{\et}{\mathrm{\acute{e}t}}
\newcommand{\Qpb}{\overline{\mathbb{Q}}_p}

\title[p-adic interpolation of uFJ periods]{On the p-adic interpolation of unitary Friedberg--Jacquet periods}
\author{Andrew Graham}
\date{}

\begin{document}
\begin{abstract}
    We establish functoriality of higher Coleman theory for certain unitary Shimura varieties and use this to construct a $p$-adic analytic function interpolating unitary Friedberg--Jacquet periods. 
\end{abstract}

\maketitle

\tableofcontents

%%%%%%%%%%%%%%%%%%%%%%%%%%%%%%%%%%%%%%%%%%%%%%%%%%%%%%%%%%%%%%%%%%%%%%%%%%%%%%%%%%%%%%%%%%%%%%%%
%%%%%%        INTRODUCTION     %%%%%%%%%%%%%%%%%%%%%%%%%%%%%%%%%%%%%%%%%%%%%%%%%%%%%%%%%%%%%%%%%
%%%%%%%%%%%%%%%%%%%%%%%%%%%%%%%%%%%%%%%%%%%%%%%%%%%%%%%%%%%%%%%%%%%%%%%%%%%%%%%%%%%%%%%%%%%%%%%%

\section{Introduction}

The conjecture of Bloch--Kato describes a precise relation between special values of $L$-functions attached to geometric Galois representations and the dimension of the associated Bloch--Kato Selmer group (which can be seen as a generalisation of the free part of the Mordell--Weil group for an abelian variety). One of the key tools in establishing cases of this conjecture is an Euler system -- a collection of group cohomology classes for the Galois representation which, under a ``non-vanishing criterion'', impose constraints on the size of the Bloch--Kato Selmer group (for example, see \cite{rubin}, \cite{MazurRubin}). The application to the Bloch--Kato conjecture then arises from a relation between this ``non-vanishing criterion'' and special values of the $L$-function; such a relation is commonly referred to as an \emph{explicit reciprocity law}.

In the setting where the Galois representation is automorphic, it is often the case that these special $L$-values can be expressed as an automorphic period for a pair of reductive groups $(\mbf{G}, \mbf{H})$. If $(\mbf{G}, \mbf{H})$ can be enhanced to a pair of Shimura data, then one can often describe this automorphic period as a pairing in coherent cohomology for the pair of associated Shimura varieties. This provides an arithmetic interpretation of the $L$-values, which can be related to the Euler system classes via a $p$-adic $L$-function (interpolating these automorphic periods and hence the $L$-values).

This present article describes the construction of a $p$-adic analytic function which should play the role of this $p$-adic $L$-function in an explicit reciprocity law for the anticyclotomic Euler system constructed in \cite{ACES} (or more precisely, its generalisation to CM fields, which will appear in forthcoming work of the author, D. Barrera and C. Williams\footnote{In fact, we also show that these Euler system classes vary in Coleman families.}). The construction crucially uses the recently developed higher Coleman theory of Boxer--Pilloni \cite{BoxerPilloni} and the strategy is similar to the work of Loeffler--Zerbes \cite{LZBK21} and Loeffler--Pilloni--Skinner--Zerbes \cite{LPSZ}. Furthermore, as a key ingredient, we $p$-adically interpolate the branching laws for representations of $\opn{GL}_{2n}$ and $\opn{GL}_{2n-1}$ restricted to $\opn{GL}_n \times \opn{GL}_n$ and $\opn{GL}_{n-1} \times \opn{GL}_n$ respectively (see Appendix \ref{AppendixBranchingLaws}), using the fact that these pairs give rise to \emph{spherical varieties}.

Unfortunately, our result is not optimal -- there is a missing variable in this $p$-adic analytic function, which would therefore lead to a sub-optimal version of an explicit reciprocity law (similar to the restriction in \cite{LZBK21}). To account for the missing variable, one would need to incorporate the $p$-adic variation of certain theta operators into the picture. This incorporation will be pursued in future work.

\subsubsection{Unitary Friedberg--Jacquet periods}

The $p$-adic analytic function we construct interpolates so-called \emph{unitary Friedberg--Jacquet periods} for certain cuspidal automorphic representations of unitary groups, which is a variant of the automorphic periods for general linear groups studied by Friedberg--Jacquet \cite{FJ93}. Although expected, it is not yet known (in general) whether these unitary Friedberg--Jacquet periods calculate $L$-values, but there has been a lot of recent work towards showing this; in particular
\begin{itemize}
    \item The ``relative trace formula approach'' in forthcoming work of Jingwei Xiao and Wei Zhang, and the work of Spencer Leslie \cite{LeslieEndoscopy}, \cite{LeslieFL}.
    \item Applications of the residue method in the work of Pollack--Wan--Zydor \cite{PWZ19}.
    \item An approach via theta correspondences in the work of Chen--Gan \cite{ChenGan}.
\end{itemize}
As a consequence of these works, we at least know that if certain values of this $p$-adic analytic function are non-vanishing then the corresponding (complex) $L$-values are also non-vanishing (see Corollary \ref{NonVanishingCorollary} below). We expect that there is an analogous version of Waldspurger's formula in this setting which will express (the square of) these values in terms of the complex $L$-values, but we do not attempt to establish such an identity in this article. Nevertheless, with these considerations in mind, we will henceforth refer to this $p$-adic analytic function as a $p$-adic $L$-function.

\subsection{Statement of the results}

Let $F$ be a CM field with maximal totally real subfield $F^+$, and fix an odd rational prime $p$ which splits completely in $F/\mbb{Q}$. We impose the following assumptions:

\begin{assumption} \label{IntroAssumption}
We assume that:
\begin{enumerate}
    \item $F^+ \neq \mbb{Q}$ and $F$ contains an imaginary quadratic number field $E/\mbb{Q}$. 
    \item $p$ does not divide the class number of $F$.
\end{enumerate}
\end{assumption}

Fix an integer $n \geq 1$. Let $W$ be a $2n$-dimensional Hermitian space over $F$ with signature $(1, 2n-1)$ at one place, and signature $(0, 2n)$ at the remaining places. Fix a decomposition $W = W_1 \oplus W_2$ of Hermitian spaces where each factor has dimension $n$, the signature of $W_1$ is $(1, n-1)$ at one place and $(0, n)$ at all remaining places, and the signature of $W_2$ is $(0, n)$ at all places. Let $\mbf{G}$ be the reductive group over $\mbb{Q}$ of unitary similitudes of $W$ with similitude in $\mbb{G}_m$. We let $\mbf{H} \subset \mbf{G}$ denote the subgroup preserving the decomposition $W = W_1 \oplus W_2$. 

Let $\pi$ be a discrete series cuspidal automorphic representation of $\mbf{G}(\mbb{A})$, and let $\chi \colon \mbb{A}_F^{\times} \to \mbb{C}^{\times}$ be an algebraic Hecke character which is \emph{anticyclotomic} (i.e. its restriction to $\mbb{A}_{F^+}^{\times}$ is trivial). Then for any $\varphi \in \pi$, we can consider the following automorphic period 
\[
\mathscr{P}_{\pi, \chi}(\varphi) \defeq \int_{[\mbf{H}]} \varphi(h) \cdot \chi \left( \frac{\opn{det}h_2}{\opn{det}h_1} \right) dh .
\]
Here $h_i$ denotes the component of $h$ corresponding to the factor $W_i$, and $[\mbf{H}] = \mbf{H}(\mbb{Q}) A_{\mbf{G}}(\mbb{A}) \backslash \mbf{H}(\mbb{A})$ with $A_{\mbf{G}}$ denoting the maximal split subtorus of the centre of $\mbf{G}$ (which can be shown to lie in $\mbf{H}$). For this to make sense, we also need to assume that the central character of $\pi$ is trivial on $A_{\mbf{G}}(\mbb{A})$. 

Let $\psi \boxtimes \Pi_0$ denote the (weak) automorphic base-change of $\pi$ to $\opn{GL}_1(\mbb{A}_E) \times \opn{GL}_{2n}(\mbb{A}_F)$, as constructed in \cite{ShinAppendix}. We have the following conjecture of Xiao--Zhang (see \cite[Conjecture 7.4]{ChenGan}):

\begin{conjecturex}
With set-up as above, assume that $\pi$ is tempered. Then there exists $\varphi \in \pi$ such that $\mathscr{P}_{\pi, \chi}(\varphi) \neq 0$ if and only if the following three conditions are satisfied:
\begin{enumerate}
    \item The standard $L$-function $L(\Pi_0 \otimes \chi, s)$ is non-vanishing at $s=1/2$.
    \item The exterior square $L$-function $L(\Pi_0, \bigwedge^2, s)$ has a pole at $s=1$.
    \item There exists an irreducible constituent $\pi_0 \subset \pi|_{\mbf{H}_0}$ such that, for every (finite) rational prime $\ell$, the Hom-space satisfies
    \[
    \opn{Hom}_{\mbf{H}_0(\mbb{Q}_\ell)}\left( \pi_{0,\ell}, \chi^{-1} \circ \nu \right) \neq 0
    \] 
    where $\mbf{H}_0 \subset \mbf{H}$ is the kernel of the similitude character and $\nu$ is the character on $\mbf{H}_0$ given by $\nu(h) = \opn{det}h_2/\opn{det}h_1$.
\end{enumerate}
\end{conjecturex}

\begin{remark}
Because we are working with unitary similitudes, this conjecture is presented in a slightly different way to \cite[Conjecture 7.4]{ChenGan}. However the two statements are equivalent by Remark \ref{UFJRemark}. 
\end{remark}

Suppose that $\pi$ is ramified\footnote{In the weakest possible sense, namely there does not exist a maximal special subgroup with non-trivial fixed points on the corresponding local component of $\pi$.} only at primes which split in $E/\mbb{Q}$, the base-change $\Pi_0$ is cuspidal and $\pi$ satisfies a ``small slope condition'' at the prime $p$ (see Assumption \ref{SSSassumption}). Then, following \cite[\S 6.9]{BoxerPilloni} and \cite{LZBK21}, we show that there exists a unique family $\underline{\pi}$ of automorphic representations, passing though $\pi$ and defined over a certain open affinoid subspace $U$ of $n[F^+ : \mbb{Q}]$-dimensional weight space $\invs{W}_G$. Here, by family we mean an $\ordd(U)$-valued system of eigenvalues for a certain collection of Hecke operators (see Definition \ref{DefinitionOfFamilyThroughPi}) -- for a classical point $x \in U$, the specialisation of the family at $x$ corresponds to a cohomological cuspidal automorphic representation $\underline{\pi}_x$ of $\mbf{G}(\mbb{A})$ (see Remark \ref{TheSpecialisationRemark}).

On the other hand, by Assumption \ref{IntroAssumption}(2), we can construct a family $\underline{\chi}$ of anticyclotomic characters defined over the ($[F^+ : \mbb{Q}]-1$)-dimensional weight space $\invs{W}_H$ parameterising characters of $\left( \mbb{Z}_p^{\times} \right)^{[F^+ : \mbb{Q}]-1}$, which passes through the character $\chi$. As above, for a point $x \in \invs{W}_H$, we let $\underline{\chi}_x$ denote the specialisation of the family at $x$. The main result of the article is the following:

\begin{theoremx}[{Corollary \ref{RelationToAutoPeriods}}]
There exists a Zariski dense subset of classical weights $\Sigma^{\mathrm{int}} \subset U \times \invs{W}_H$ and a $p$-adic analytic function $\mathscr{L}_p = \mathscr{L}_p(\underline{\eta}. \underline{\chi}) \in \ordd(U \times \invs{W}_H)$ which interpolates the periods $\mathscr{P}_{\underline{\pi}_x, \underline{\chi}_x}(\varphi_x)$ for $x \in \Sigma^{\mathrm{int}}$ (where $\varphi_x \in \underline{\pi}_x$ is a certain non-zero choice of automorphic form).
\end{theoremx}

Combining this with \cite[Corollary 7.6]{ChenGan} (and the fact that regular algebraic conjugate self-dual cuspidal automorphic representations of $\opn{GL}_{2n}(\mbb{A}_F)$ are tempered \cite{Caraiani12}), we see that

\begin{corollaryx} \label{NonVanishingCorollary}
Let $x \in \Sigma^{\mathrm{int}}$ and let $\psi_x \boxtimes \opn{BC}(\underline{\pi}_x)$ denote the automorphic base-change of $\underline{\pi}_x$ to a representation of $\opn{GL}_1(\mbb{A}_E) \times \opn{GL}_{2n}(\mbb{A}_F)$. Suppose that $\opn{BC}(\underline{\pi}_x)$ is cuspidal. Then
\[
\mathscr{L}_p(x) \neq 0 \quad \Rightarrow \quad L(\opn{BC}(\underline{\pi}_x) \otimes \underline{\chi}_x, 1/2) \neq 0 .
\]
\end{corollaryx}

The strategy we will use for constructing $\mathscr{L}_p$ consists of three key steps:
\begin{enumerate}
    \item Express the automorphic periods $\mathscr{P}_{\underline{\pi}_x, \underline{\chi}_x}(\varphi_x)$ as a cup product in the coherent cohomology of a Shimura variety associated with $\mbf{H}$ involving (the restriction to $\mbf{H}$) of a coherent cohomology class $\underline{\eta}_x$ corresponding to $\varphi_x$. 
    \item Using higher Coleman theory one can reinterpret (1) in terms of a pairing in coherent cohomology over certain strata in the adic Shimura varieties for $\mbf{G}$ and $\mbf{H}$. In particular, this interpretation is amenable to $p$-adic interpolation provided that there exist families of cohomology classes $\underline{\eta}$ and $\underline{\chi}$ passing through $\underline{\eta}_x$ and $\underline{\chi}_x$ respectively.
    \item Under the above assumptions, we construct these families $\underline{\eta}$ and $\underline{\chi}$. The $p$-adic $L$-function $\mathscr{L}_p$ is then defined as a pairing between the classes $\underline{\eta}$ and $\underline{\chi}$. 
\end{enumerate}

\begin{remark}
Assumption \ref{IntroAssumption}(1) is imposed throughout the whole article, however assumption (2) is only imposed when showing the existence of certain anticyclotomic algebraic Hecke characters for $F$ (which we expect can be removed by passing to a finite cover of weight space). In fact, it is likely that assumption (1) is not needed until \S \ref{FamiliesOfCohomologyClassesSection} when applying the automorphic base-change results in \cite{ShinAppendix}.
\end{remark}

\begin{remark}[{Example \ref{BorelImpliesSSExample}}]
The ``small slope condition'' at the prime $p$ is implied by (but more general than) a Borel-ordinarity condition on $\pi$ (i.e. there exists an eigenvalue for the action of a suitably normalised Borel $U_p$-Hecke operator on $\pi_p$ which is a $p$-adic unit). 
\end{remark}

\begin{remark}
To show the existence of the family $\underline{\pi}$ we need to implicitly use the results in \cite{Mok} and \cite{KMSW14} on the endoscopic classification for unitary groups. As far as the author is aware, this work is still conditional on the stabilisation of the twisted trace formula for $\mbf{G}_0$ and $\opn{Res}_{F/\mbb{Q}}\opn{GL}_{2n}$ and their endoscopy groups.
\end{remark}

\subsection{Notation} \label{NotationSection}

Throughout this article, we fix a totally real number field $F^+ \neq \mbb{Q}$ with a fixed embedding $\tau_0 \colon F^+ \hookrightarrow \mbb{R}$. We fix a totally imaginary quadratic extension $F/F^+$ and a CM type $\Psi$ for $F$, i.e. $\Psi$ is a set of embeddings $F \hookrightarrow \mbb{C}$ of size $[F^+ : \mbb{Q}]$, with no two embeddings being equivalent to one another. We denote by $\tau_0$ the element of $\Psi$ which extends the embedding $\tau_0 \colon F^+ \hookrightarrow \mbb{R}$. Let $F^{\mathrm{cl}}$ denote the Galois closure of $F$. We assume that $F$ contains an imaginary quadratic number field $E$.

We fix an odd prime $p$ which splits completely in $F/\mbb{Q}$, and we fix an isomorphism $\iota_p \colon \mbb{C} \xrightarrow{\sim} \Qpb$. Under this isomorphism every embedding $\tau \in \Psi$ gives rise to a prime ideal $\ide{p}_{\tau}$ of $F$, lying above $p$. 

We also fix the following notation and conventions throughout:
\begin{itemize}
    \item For any split reductive group $G$, we let $w_G^{\mathrm{max}}$ denote the element of its Weyl group of maximal length.
    \item The group law on characters will be written additively, unless specified otherwise.
    \item Let $G$ be a split reductive group with a fixed parabolic $P \subset G$ and Levi $M$, and let $T \subset P$ be a maximal torus. Then for any algebraic character $\kappa$ of $T$ which is $M$-dominant, we will write
    \[
    \kappa^{\vee} = -w_{M}^{\mathrm{max}} \kappa - 2\rho_{nc} 
    \]
    for the \emph{Serre dual} of $\kappa$, where $\rho_{nc}$ is the half-sum of positive roots not lying in $M$ (with respect to a fixed Borel containing $T$ and contained in $P$). We will also use the notation $(-)^{\vee}$ to refer to the Serre dual of a vector bundle on a scheme.
    \item We will use the terminology \emph{neat} or \emph{sufficiently small} to refer to a compact open subgroup of the finite adelic points of a reductive group satisfying \cite[Definition B.6]{ACES}.
    \item All torsors are right torsors unless specified otherwise.
\end{itemize}

\subsection{Acknowledgements}

The author would like to thank Daniel Barrera Salazar, George Boxer, David Loeffler, Vincent Pilloni, Juan Esteban Rodr\'{i}guez Camargo, Chris Williams and Sarah Zerbes for helpful comments and discussions. The author is also grateful to the anonymous referee for their feedback and pointing out a significant error in an earlier version of this article. This work was supported in part by ERC-2018-COG-818856-HiCoShiVa.

%%%%%%%%%%%%%%%%%%%%%%%%%%%%%%%%%%%%%%%%%%%%%%%%%%%%%%%%%%%%%%%%%%%%%%%%%%%%%%%%%%%%%%%%%%%%%%%
%%%%%       PRELIMINARIES     %%%%%%%%%%%%%%%%%%%%%%%%%%%%%%%%%%%%%%%%%%%%%%%%%%%%%%%%%%%%%%%%%
%%%%%%%%%%%%%%%%%%%%%%%%%%%%%%%%%%%%%%%%%%%%%%%%%%%%%%%%%%%%%%%%%%%%%%%%%%%%%%%%%%%%%%%%%%%%%%%

\section{Preliminaries} \label{PreliminarySection}

Let $n \geq 1$ be a positive integer. Let $W$ denote a $2n$-dimensional Hermitian space over $F$ which has signature $(1, 2n-1)$ with respect to the embedding $\tau_0$, and signature $(0, 2n)$ at $\tau \in \Psi - \{\tau_0 \}$. Fix a decomposition $W = W_1 \oplus W_2$ of Hermitian spaces, where $W_i$ is a Hermitian space over $F$ of dimension $n$ with signatures
\[
\opn{signature}(W_i \otimes_{F, \tau} \mbb{C}) = \left\{ \begin{array}{cc} (1, n-1) & \text{ if } i=1 \text{ and } \tau = \tau_0 \\ (0, n) & \text{ otherwise } \end{array} \right.
\]
Denote the Hermitian pairings on $W$ and $W_i$ by $\langle , \rangle_W$ and $\langle , \rangle_{W_i}$ respectively.

\begin{definition}
Let $\mbf{G}$ and $\mbf{H}$ denote the reductive groups over $\mbb{Q}$ whose values on $R$-points, for a $\mbb{Q}$-algebra $R$, are
\begin{align*}
    \mbf{G}(R) &= \left\{ g \in \opn{GL}(W \otimes_{\mbb{Q}} R) : \begin{array}{c}\langle g \cdot x, g \cdot y \rangle_W = c(g)\langle x, y \rangle_W \text{ for all } x, y \in W \otimes_{\mbb{Q}} R \\ \text{ and some } c(g) \in R^{\times} \end{array}\right\} \\
    \mbf{H}(R) &= \left\{ g=(g_1, g_2) \in \opn{GL}(W_1 \otimes_{\mbb{Q}} R) \times \opn{GL}(W_2 \otimes_{\mbb{Q}} R) : \begin{array}{c} \langle g_i \cdot x_i, g_i \cdot y_i \rangle_{W_i} = c(g)\langle x_i, y_i \rangle_{W_i} \\ \text{ for all } x_i, y_i \in W_i \otimes_{\mbb{Q}} R \text{ and } i=1, 2 \\ \text{ and some } c(g) \in R^{\times} \end{array} \right\}
\end{align*}
We also let $\mbf{G}_0$ (resp. $\mbf{H}_0$) denote the kernel of the similitude character $c \colon \mbf{G} \to \mbb{G}_m$ (resp. $c \colon \mbf{H} \to \mbb{G}_m$). Note that we have natural embeddings 
\[
\mbf{H}_0 \hookrightarrow \mbf{G}_0, \quad \quad \mbf{H} \hookrightarrow \mbf{G}
\]
both of which we will denote by $\iota$. 
\end{definition}

\begin{remark} \label{RemarkGLIdentifications}
If $R$ is an $F^{\mathrm{cl}}$-algebra (with fixed embedding $F^{\mathrm{cl}} \hookrightarrow R$), then we have an identification
\[
W \otimes_{\mbb{Q}} R = \bigoplus_{\tau \in \Psi} \left(W \otimes_{F, \tau} R  \oplus W \otimes_{F, \overline{\tau}} R \right)
\]
where $\tau \colon F \hookrightarrow R$ denotes the embedding obtained from precomposing the fixed embedding $F^{\mathrm{cl}} \hookrightarrow R$ with $\tau \colon F \hookrightarrow F^{\mathrm{cl}}$, and $\overline{\tau} \colon F \hookrightarrow R$ denotes its complex conjugate. Under this identification, one has
\[
\mbf{G}_{0, F^{\opn{cl}}} = \prod_{\tau \in \Psi} \opn{GL}_{2n, F^{\opn{cl}}}, \quad \quad \mbf{G}_{F^{\opn{cl}}} = \opn{GL}_{1, F^{\opn{cl}}} \times \prod_{\tau \in \Psi} \opn{GL}_{2n, F^{\opn{cl}}}
\]
where the latter is described by sending an element $g \in \mbf{G}_{F^{\opn{cl}}}(R)$ to $(c(g), g|_{W \otimes_{F, \tau} R} )_{\tau \in \Psi}$.

In particular, if $p$ is a prime which splits completely in $F/\mbb{Q}$ and we have an isomorphism $\mbb{C} \cong \Qpb$, then we obtain a distinguished embedding $F \hookrightarrow \mbb{Q}_p$ arising from $\tau_0$ (and factoring through $F^{\mathrm{cl}}$) and $\mbf{G}_{\mbb{Q}_p}$ is identified with $\opn{GL}_{1, \mbb{Q}_p} \times \prod_{\tau \in \Psi} \opn{GL}_{2n, \mbb{Q}_p}$.

Similarly, we have identifications
\[
\mbf{H}_{0,F^{\opn{cl}}} = \prod_{\tau \in \Psi} \left( \opn{GL}_{n, F^{\opn{cl}}} \times \opn{GL}_{n, F^{\opn{cl}}} \right), \quad \quad \mbf{H}_{F^{\opn{cl}}} = \opn{GL}_{1, F^{\opn{cl}}} \times \prod_{\tau \in \Psi} \left( \opn{GL}_{n, F^{\opn{cl}}} \times \opn{GL}_{n, F^{\opn{cl}}} \right)  
\]
and the embeddings $\mbf{H}_{0, F^{\opn{cl}}} \xrightarrow{\iota} \mbf{G}_{0, F^{\opn{cl}}}$ and $\mbf{H}_{F^{\opn{cl}}} \xrightarrow{\iota} \mbf{G}_{F^{\opn{cl}}}$ map the $\opn{GL}_{1, F^{\opn{cl}}}$-factor to itself, and for each $\tau \in \Psi$, map $\opn{GL}_{n, F^{\opn{cl}}} \times \opn{GL}_{n, F^{\opn{cl}}}$ into $\opn{GL}_{2n, F^{\opn{cl}}}$ block diagonally. 
\end{remark}

Using the identifications in Remark \ref{RemarkGLIdentifications}, we define the following parabolic subgroups:
\begin{itemize}
    \item Let $B_{\mbf{G}}$ (resp. $B_{\mbf{H}}$) denote the upper-triangular Borel subgroup of $\mbf{G}_{F^{\mathrm{cl}}}$ (resp. $\mbf{H}_{F^{\mathrm{cl}}}$). We let $T$ denote the standard maximal torus inside $B_{\mbf{G}}$ (which also coincides with the standard maximal torus inside $B_{\mbf{H}}$). In particular, elements of $T$ can be described as tuples 
    \[
    (x; y_{1, \tau}, \dots, y_{2n, \tau})_{\tau \in \Psi} 
    \]
    corresponding to the diagonal matrix
    \[
    x \times \prod_{\tau \in \Psi} \opn{diag}(y_{1, \tau}, \dots, y_{2n, \tau}) \in \opn{GL}_{1} \times \prod_{\tau \in \Psi} \opn{GL}_{2n} .
    \]
    \item Let $P_{\mbf{G}}$ denote the parabolic subgroup of $\mbf{G}_{F^{\mathrm{cl}}}$ containing $B_{\mbf{G}}$ with Levi given by
    \[
    M_{\mbf{G}} = \opn{GL}_{1, F^{\mathrm{cl}}} \times \left(\opn{GL}_{1, F^{\mathrm{cl}}} \times \opn{GL}_{2n-1, F^{\mathrm{cl}}} \right) \times \prod_{\tau \in \Psi - \{\tau_0\}} \opn{GL}_{2n, F^{\mathrm{cl}}} .
    \]
    Similarly, we let $P_{\mbf{H}}$ denote the parabolic of $\mbf{H}_{F^{\mathrm{cl}}}$ containing $B_{\mbf{H}}$ with Levi given by
    \[
    M_{\mbf{H}} = \opn{GL}_{1, F^{\mathrm{cl}}} \times \left(\opn{GL}_{1, F^{\mathrm{cl}}} \times \opn{GL}_{n-1, F^{\mathrm{cl}}} \times \opn{GL}_{n, F^{\mathrm{cl}}} \right) \times \prod_{\tau \in \Psi - \{\tau_0\}} \left(\opn{GL}_{n, F^{\mathrm{cl}}} \times \opn{GL}_{n, F^{\mathrm{cl}}} \right) 
    \]
    so that $P_{\mbf{H}} = P_{\mbf{G}} \cap \mbf{H}_{F^{\mathrm{cl}}}$ and $M_{\mbf{H}} = M_{\mbf{G}} \cap \mbf{H}_{F^{\mathrm{cl}}}$.
    \item Let $T_0 \subset T$ denote the sub-torus given by elements of the form
    \[
    (x; y_{1, \tau}, \dots, y_{n, \tau}, y_{n, \tau}, \dots, y_{1, \tau})_{\tau \in \Psi} .
    \]
\end{itemize}

We now describe the relevant Weyl groups that will be used throughout this article. 

\begin{definition}
For $? \in \{\mbf{G}_{F^{\mathrm{cl}}}, \mbf{H}_{F^{\mathrm{cl}}}, M_{\mbf{G}}, M_{\mbf{H}}\}$, let $W_{?}$ denote the associated Weyl group. Let ${^MW_{\mbf{G}}}$ denote the set of Kostant representatives for the quotient $W_{M_{\mbf{G}}} \backslash W_{\mbf{G}_{F^{\opn{cl}}}}$. This set comprises of $2n$ elements
\[
{^MW_{\mbf{G}}} = \{ w_0, \dots, w_{2n-1} \}
\]
where the length of $w_i$ is $i$. Similarly, ${^MW_{\mbf{H}}}$ (the set of Kostant representatives for $W_{M_{\mbf{H}}} \backslash W_{\mbf{H}_{F^{\opn{cl}}}}$) is a set $\{w_0, \dots, w_{n-1} \}$ where the length of $w_i$ is $i$. We can (and do) choose representatives for the Weyl elements $w_i$ in $\mbf{G}$ such that the embedding $\mbf{H} \hookrightarrow \mbf{G}$ identifies ${^MW_{\mbf{H}}}$ with the subset of ${^MW_{\mbf{G}}}$ of elements of lengths $0, \dots, n-1$ (which justifies the notation). In \S \ref{SomeImportantElementsSubsec}, we will make a specific choice of representative for $w_n$.
\end{definition}

We let $X^*(T) = \opn{Hom}\left(T, \mbb{G}_m \right)$ denote the abelian group of algebraic characters of $T$. We also define $X^*(T/T_0) = \opn{Hom}\left(T/T_0, \mbb{G}_m \right)$ which can naturally be viewed as a subgroup of $X^*(T)$ (by precomposing with the quotient $T \twoheadrightarrow T/T_0$). We identify elements of $X^*(T)$ with tuples of integers
\[
(c_0; c_{1, \tau}, \dots, c_{2n, \tau})_{\tau \in \Psi} 
\]
which correspond to the character mapping an element $(x; y_{1, \tau}, \dots, y_{2n, \tau})_{\tau \in \Psi} \in T$ to the quantity 
\[
x^{c_0}\prod_{\substack{\tau \in \Psi \\ i = 1, \dots, 2n}} y_{i, \tau}^{c_{i, \tau}}.
\]
With this description, elements of $X^*(T/T_0)$ are identified with tuples as above, satisfying $c_0 = 0$ and $c_{i, \tau} + c_{2n+1-i, \tau} = 0$ for all $\tau \in \Psi$ and $i=1, \dots, 2n$. We let $X^*(T)^+ \subset X^*(T)$ denote the cone of dominant characters, i.e. tuples of integers as above which satisfy $c_{1, \tau} \geq \cdots \geq c_{2n, \tau}$ for all $\tau \in \Psi$, and we set $X^*(T/T_0)^+ = X^*(T/T_0) \cap X^*(T)^+$.

The Weyl group $W_{\mbf{G}}$ naturally acts on $X^*(T)$ by the formula
\[
w \cdot \lambda(t) = \lambda(w^{-1} t w), \quad \quad w \in W_{\mbf{G}}, t \in T.
\]
In particular the set ${^MW_{\mbf{G}}}$ acts by shuffles, i.e. $w_0$ acts as the identity, and for $i=1, \dots, 2n-1$, one has the following description:
\[
w_i \cdot (c_0; c_{1, \tau}, \dots, c_{2n, \tau})_{\tau \in \Psi} = (c_0; c_{i+1, \tau_0}, c_{1, \tau_0}, \dots, c_{i, \tau_0}, c_{i+2, \tau_0}, \dots, c_{2n, \tau_0}; c_{1, \tau}, \dots, c_{2n, \tau})_{\tau \in \Psi - \{\tau_0\}}.
\]

\begin{definition} \label{DefOfHalfSum}
Let $\rho \in \frac{1}{2}X^*(T)$ denote the half-sum of the positive roots of $\mbf{G}_F$ with respect to the Borel $B_{\mbf{G}}$. Explicitly, this is given by
\[
\rho = \left( 0; \frac{2n-1}{2}, \frac{2n-3}{2}, \dots, \frac{3-2n}{2}, \frac{1-2n}{2} \right)_{\tau \in \Psi}.
\]
Let $\rho_c \in \frac{1}{2}X^*(T)$ (resp. $\rho_{nc} \in \frac{1}{2}X^*(T)$) denote the half-sum of positive roots which lie in $M_{\mbf{G}}$ (resp. do not lie in $M_{\mbf{G}}$). Explicitly, the components of $\rho_{c}$ (resp. $\rho_{nc}$) agree with $\rho$ on the $\opn{GL}_1$-factor and on the $\opn{GL}_{2n}$-factor for $\tau \neq \tau_0$ (resp. are zero on the $\tau \neq \tau_0$ factor), but the $\tau_0$-factors are given by 
\begin{align*}
    \left( 0, n-1, n-2, \dots, 2-n, 1-n \right) \\
    \left( \frac{2n-1}{2}, -\frac{1}{2}, -\frac{1}{2}, \dots, -\frac{1}{2}, -\frac{1}{2} \right)
\end{align*}
respectively.
\end{definition}

We conclude this section by introducing notation for the categories of algebraic representations of $M_{\mbf{G}}$ and $M_{\mbf{H}}$.

\begin{notation}
Let $\opn{Rep}(M_{\mbf{G}})$ (resp. $\opn{Rep}(M_{\mbf{H}})$) denote the category of finite-dimensional algebraic representations of $M_{\mbf{G}}$ (resp. $M_{\mbf{H}}$).
\end{notation}

\subsection{Shimura varieties} \label{ShimuraVarietiesSubSec}

We consider the following Shimura data for the groups $\mbf{G}$ and $\mbf{H}$. Let $\mbb{S} = \opn{Res}_{\mbb{C}/\mbb{R}}\mbb{G}_m$ denote the Deligne torus. Recall from Remark \ref{RemarkGLIdentifications} that we have an identification
\[
W \otimes_{\mbb{Q}} \mbb{C} = \bigoplus_{\tau \in \Psi} \left( W \otimes_{F, \tau} \mbb{C} \oplus W \otimes_{F, \overline{\tau}} \mbb{C} \right).
\]
For an embedding $\tau \colon F \hookrightarrow \mbb{C}$, each piece $W_{\tau} \defeq W \otimes_{F, \tau} \mbb{C}$ comes equipped with a Hermitian pairing by base-extension of $\langle, \rangle_W$. We fix a decomposition $W_{\tau} = W^+_{\tau} \oplus W^{-}_{\tau}$ into maximal subspaces where the induced pairing is positive (resp. negative) definite. We define the following Hodge structure (of type $\{(-1, 0), (0, -1)\}$)
\[
W \otimes_{\mbb{Q}} \mbb{C} = W^{(-1, 0)} \oplus W^{(0, -1)}
\]
by imposing that 
\[
W^{(-1, 0)} \defeq \bigoplus_{\tau \in \Psi} (W^+_{\tau} \oplus W^{-}_{\overline{\tau}}), \quad \quad W^{(0, -1)} \defeq \bigoplus_{\tau \in \Psi} (W^{-}_{\tau} \oplus W^+_{\overline{\tau}} ).
\]
This defines a homomorphism $h_{\mbf{G}} \colon \mbb{S} \to \mbf{G}_{\mbb{R}}$. We have a similar description for $h_{\mbf{H}}$ and we can arrange it in such a way that $h_{\mbf{G}} = \iota \circ h_{\mbf{H}}$. 

Let $\mu_{\mbf{G}}$ denote the restriction of $h_{\mbf{G}, \mbb{C}}$ to the first component in the identification $\mbb{S}_{\mbb{C}} \cong \mbb{G}_{m, \mbb{C}} \times \mbb{G}_{m, \mbb{C}}$. Then (after possibly conjugating $h_{\mbf{G}}$ by an element of $\mbf{G}(\mbb{R})$) under the identification in Remark \ref{RemarkGLIdentifications}, the cocharacter $\mu_{\mbf{G}}$ is given by
\begin{align*}
    \mu_{\mbf{G}} \colon \mbb{G}_{m, \mbb{C}} &\to \opn{GL}_{1, \mbb{C}} \times \prod_{\tau \in \Psi} \opn{GL}_{2n, \mbb{C}} \\
    z &\mapsto z \times \opn{diag}(z, 1, \dots, 1) \times \prod_{\tau \in \Psi - \{\tau_0 \}} \opn{diag}(1, \dots, 1) .
\end{align*}
In particular, $\mu_{\mbf{G}}$ is defined over $F^{\mathrm{cl}}$. Furthermore, the field of definition of the $\mbf{G}(\mbb{C})$-conjugacy class of $\mu_{\mbf{G}}$ is $F$, because of the conditions on the signatures and our assumption that $F$ contains an imaginary quadratic number field. Note that $\mu_{\mbf{G}}$ is of the form $\iota \circ \mu_{\mbf{H}}$ for a cocharacter $\mu_{\mbf{H}} \colon \mbb{G}_{m, \mbb{C}} \to \mbf{H}_{\mbb{C}}$, and this cocharacter coincides with the one obtained from $h_{\mbf{H}}$ similar to above. The field of definition of the $\mbf{H}(\mbb{C})$-conjugacy class of $\mu_{\mbf{H}}$ is also $F$, and the cocharacter $\mu_{\mbf{H}}$ is defined over $F^{\mathrm{cl}}$.

\begin{remark}
The centraliser of $\mu_{\mbf{G}}$ (resp. $\mu_{\mbf{H}}$) in $\mbf{G}_{F^{\mathrm{cl}}}$ (resp. $\mbf{H}_{F^{\mathrm{cl}}}$) is $M_{\mbf{G}}$ (resp. $M_{\mbf{H}}$).
\end{remark}

\begin{lemma}
The data $(\mbf{G}, h_{\mbf{G}})$ and $(\mbf{H}, h_{\mbf{H}})$ define Shimura--Deligne data in the sense of \cite[Appendix B]{ACES}, and additionally satisfy (SD5). The datum $(\mbf{G}, h_{\mbf{G}})$ is a Shimura datum in the usual sense. The reflex field for both of these data is $F$.
\end{lemma}

For a neat compact open subgroup $K \subset \mbf{G}(\mbb{A}_f)$, we let $S_{\mbf{G}, K}$ denote the associated Shimura variety over the reflex field $F$. Similarly, for a neat compact open subgroup $U \subset \mbf{H}(\mbb{A}_f)$, we let $S_{\mbf{H}, U}$ denote the associated Shimura--Deligne variety over the reflex field $F$ (a canonical model exists as the connected component of the PEL-type moduli problem associated with $\mbf{H}$ and $h_{\mbf{H}}$). If $\iota(U) \subset K$, then we have an induced finite unramified morphism
\[
\iota \colon S_{\mbf{H}, U} \to S_{\mbf{G}, K}.
\]
We note that $S_{\mbf{H}, U}$ and $S_{\mbf{G}, K}$ are smooth projective varieties, because we have assumed $F^+ \neq \mbb{Q}$ (for example, the conditions in \cite[\S 5.3.3]{LanArithmetic} are satisfied).

\begin{convention}
From now on, all of the Shimura--Deligne varieties we consider will be base-changed to $F^{\mathrm{cl}}$ (or a field extension of $F^{\mathrm{cl}}$) via the embedding $\tau_0 \colon F \hookrightarrow F^{\mathrm{cl}}$, but we will suppress this from the notation. 
\end{convention}

\subsection{Automorphic vector bundles} \label{AutomorphicVectorBundlesSection}

In this section, we recall the construction of automorphic vector bundles on $S_{\mbf{G}, K}$. 

Let $P_{\mbf{G}}^{\mathrm{std}}$ denote the opposite of $P_{\mbf{G}}$ with respect to the torus $T$, and consider the flag variety $\opn{FL}^{\mathrm{std}}_{\mbf{G}} \defeq \mbf{G}/P_{\mbf{G}}^{\mathrm{std}}$. Let $X_{\mbf{G}}$ denote the $\mbf{G}(\mbb{R})$-conjugacy class of homomorphisms $\mbb{S} \to \mbf{G}_{\mbb{R}}$ containing $h_{\mbf{G}}$, which is a Hermitian symmetric domain. Then we have a holomorphic embedding (the Borel embedding)
\[
\beta \colon X_{\mbf{G}} \hookrightarrow \opn{FL}_{\mbf{G}}^{\mathrm{std}}(\mbb{C}).
\]

\begin{definition} \label{DefOfAutoVectorBundle}
Let $K \subset \mbf{G}(\mbb{A}_f)$ be a sufficiently small compact open subgroup. For an algebraic representation $V$ of $P_{\mbf{G}}^{\mathrm{std}}$, let $[V]$ denote the vector bundle on $S_{\mbf{G}, K}(\mbb{C})$ defined as
\[
[V] \defeq \mbf{G}(\mbb{Q}) \backslash \beta^*(V) \times \mbf{G}(\mbb{A}_f)/K
\]
where we view $V$ as a $\mbf{G}(\mbb{C})$-homogeneous vector bundle on $\opn{FL}^{\mathrm{std}}_{\mbf{G}}(\mbb{C})$ in the usual way.
\end{definition}

\begin{remark}
One can show that $[V]$ descends to an algebraic vector bundle on $S_{\mbf{G},K}$ (see \cite[\S III]{MilneCanonicalMixed} for example).
\end{remark}

\begin{definition}
The association in Definition \ref{DefOfAutoVectorBundle} defines a functor 
\[
 {[-] = [-]_K} \colon \opn{Rep}(M_{\mbf{G}}) \to \opn{VB}(S_{\mbf{G}, K})
\]
by inflating a representation of $M_{\mbf{G}}$ to one of $P_{\mbf{G}}^{\mathrm{std}}$, where $\opn{VB}(-)$ denotes the category of vector bundles on a scheme. This functor is compatible with varying $K$, in the sense that if $g \in \mbf{G}(\mbb{A}_f)$ and $L \subset g^{-1}Kg$, then $g^*[-]_K = [-]_L$. Here $g^*$ denotes pullback under the map $S_{\mbf{G}, L} \to S_{\mbf{G}, K}$ induced from right-translation by $g$.
\end{definition}

We have a similar description of automorphic vector bundles over $S_{\mbf{H}, U}$ arising from algebraic representations of $M_{\mbf{H}}$, and one has the relation 
\[
\iota^*[V] = [V|_{M_{\mbf{H}}}]
\]
where $V$ is an algebraic representation of $M_{\mbf{G}}$ and $\iota \colon S_{\mbf{H}, U} \to S_{\mbf{G}, K}$ is the finite unramified morphism at the end of the previous section.

\begin{example}[{\cite[\S 4.2.8]{BoxerPilloni}}]
Let $V_{-2\rho_{nc}}$ denote the irreducible algebraic representation of $M_{\mbf{G}}$ with highest weight $-2\rho_{nc}$ (see Definition \ref{DefOfHalfSum}). Then $[V_{-2\rho_{nc}}] \cong \Omega^{2n-1}_{S_{\mbf{G}, K}}$.
\end{example}

\subsection{Discrete series representations} \label{DiscreteSeriesRepsSection}

Let $K_{\infty} \subset \mbf{G}(\mbb{R})$ denote the stabiliser of $h_{\mbf{G}}$ under the adjoint action. Explicitly, this has the following description. Upon base-change to $\mbb{R}$, one has the following identification
\[
W \otimes_{\mbb{Q}} \mbb{R} = \bigoplus_{\tau \in \Psi} (W \otimes_{F^+, \tau} \mbb{R})
\]
where each summand is a $2n$-dimensional Hermitian space over $\mbb{C}$. In particular, $W \otimes_{\mbb{Q}} \mbb{R}$ is a Hermitian space over $\mbb{C}$, and the fixed decomposition
\[
W \otimes_{\mbb{Q}} \mbb{C} = \left( \bigoplus_{\tau \in \Psi} W_{\tau}^+ \oplus W^+_{\bar{\tau}} \right) \oplus \left( \bigoplus_{\tau \in \Psi} W_{\tau}^{-} \oplus W^{-}_{\bar{\tau}} \right)
\]
descends to a decomposition $W \otimes_{\mbb{Q}} \mbb{R} = W^+ \oplus W^{-}$ into maximal subspaces where the Hermitian pairing is positive (resp. negative) definite. Then $K_{\infty}$ can be described as the subgroup of $\mbf{G}(\mbb{R})$ preserving the decomposition $W \otimes_{\mbb{Q}} \mbb{R} = W^+ \oplus W^{-}$. In particular, the complexification of $K_{\infty}$ is equal to $M_{\mbf{G}}(\mbb{C})$.

Let $H_{\infty}$ denote the compact (mod centre) Cartan subgroup of $K_{\infty}$ whose complexification is equal to $T(\mbb{C})$. Then algebraic characters of $H_{\infty}$ can be identified with tuples $(c_0; c_{1, \tau}, \dots, c_{2n, \tau}) \in X^*(T)$ satisfying the parity condition
\[
c_0 \equiv \sum_{\tau \in \Psi} \sum_{i=1}^{2n} c_{i, \tau} \quad \text{ modulo } 2 .
\]
For any dominant algebraic character $\lambda$ of $H_{\infty}$ and $i=0, \dots, 2n-1$, we set $\xi_i \defeq w_i \cdot (\lambda + \rho)$. Then $\xi_i$ is the Harish-Chandra parameter of a discrete series representation $\pi(\xi_i)$ of $\mbf{G}(\mbb{R})$ (see \cite[\S 3]{BHR94}) and the local $L$-packet containing this representation is of the form
\[
\{ \pi(\xi_0), \dots, \pi(\xi_{2n-1}) \} .
\]
Therefore, discrete series $L$-packets of $\mbf{G}(\mbb{R})$ are parameterised by dominant algebraic characters of $H_{\infty}$. One has a similar description for discrete series $L$-packets of $\mbf{G}_0(\mbb{R})$.

\begin{remark}
Note that discrete series $L$-packets of both $\mbf{G}_0(\mbb{R})$ and $\mbf{G}(\mbb{R})$ have size $2n$, because $K_{\infty}$ differs from the maximal compact subgroup of $\mbf{G}_{\mathrm{der}}(\mbb{R})$ by the centre of $\mbf{G}(\mbb{R})$. In particular, if $\pi(\xi_i)$ is a discrete series representation of $\mbf{G}(\mbb{R})$ as above, then
\[
\pi(\xi_i)|_{\mbf{G}_0(\mbb{R})} = \pi(\xi'_i)
\]
where $\xi'_j$ denotes the restriction of $\xi_j$ to $H_{\infty} \cap \mbf{G}_0(\mbb{R})$ and $\pi(\xi'_j)$ denotes the discrete series representation of $\mbf{G}_0(\mbb{R})$ with Harish-Chandra parameter $\xi_j'$. 
\end{remark}

For convenience, we introduce the following dictionary of weights and parameters. Let $\lambda$ be a dominant algebraic character of $H_{\infty}$. Then 
\begin{enumerate}
    \item (Harish-Chandra parameters) The Harish-Chandra parameters in the $L$-packet parameterised by $\lambda$ are given by
    \[
    \xi_i = w_i \cdot (\lambda + \rho)
    \]
    for $i=0, \dots, 2n-1$. 
    \item (Blattner parameters) The Blattner parameters associated with $\lambda$ are
    \[
    \nu_i = w_i \cdot (\lambda + 2\rho) - 2 \rho_c .
    \]
    In particular, the lowest $K_{\infty} \cap \mbf{G}_0(\mbb{R})$-type of $\pi(\xi'_i)$ has highest weight given by (the restriction to $H_{\infty} \cap \mbf{G}_0(\mbb{R})$ of) $\nu_i$. This implies that
    \[
    \opn{dim}\opn{Hom}_{K_{\infty}}(\nu_j, \pi(\xi_i)) = \left\{ \begin{array}{cc} 1 & \text{ if } j=i  \\ 0 & \text{ otherwise } \end{array} \right.
    \]
    \item (Vector bundle weights) If we let $\lambda^* = -w_G^{\mathrm{max}}\lambda$, then the vector bundle weights are
    \[
    \kappa_i = w_i \star \lambda^* \defeq w_i \cdot (\lambda^* + \rho) - \rho.
    \]
    In the notation of \cite{BoxerPilloni}, we have 
    \[
    C(\kappa_i)^{-} = \{ w \in W_{\mbf{G}} : w^{-1}(\kappa_i+\rho) \in X^*(T)_{\mbb{Q}}^+ \} = \{w_i\}
    \]
    so we expect the coherent cohomology of $[V_{\kappa_i}]$ to be concentrated in degree $\ell_{-}(w_i) = 2n-1-i$ (at least on small slope parts). Let $\ide{p} = \opn{Lie}P_{\mbf{G}}^{\mathrm{std}}$ and $\ide{m} = \opn{Lie}M_{\mbf{G}}$, then for $i=0, \dots, 2n-1$, $\bigwedge^i \ide{p}/\ide{m}$ is an irreducible algebraic representation of $M_{\mbf{G}}$ under the adjoint action. If we let $\alpha_i$ denote the highest weight of this representation, then the vector bundle weights and Blattner parameters are related by the formula:
    \[
    \nu_i = \alpha_i - w_{M_G}^{\mathrm{max}} \kappa_{2n-1-i}.
    \]
\end{enumerate}

\subsection{Some important elements} \label{SomeImportantElementsSubsec}

Recall that we have identifications 
\begin{align*} 
\mbf{G}_{F^{\mathrm{cl}}} &= \opn{GL}_{1, F^{\mathrm{cl}}} \times \prod_{\tau \in \Psi} \opn{GL}_{2n, F^{\mathrm{cl}}} \\
\mbf{H}_{F^{\mathrm{cl}}} &= \opn{GL}_{1, F^{\mathrm{cl}}} \times \prod_{\tau \in \Psi} \left(\opn{GL}_{n, F^{\mathrm{cl}}} \times \opn{GL}_{n, F^{\mathrm{cl}}} \right).
\end{align*} 
In particular $\mbf{G}_{F^{\mathrm{cl}}}$ and $\mbf{H}_{F^{\mathrm{cl}}}$ (and the algebraic subgroups considered throughout this section) have models over $\ordd \defeq \ordd_{F^{\mathrm{cl}}}$, which we will denote by the same letters.

Let $w_n \in {^MW_{\mbf{G}}}$ denote the Weyl element of length $n$. We will now make explicit a choice of representative (which we will also denote $w_n$) in $\mbf{G}(\ordd)$ which represents the element $w_n \in {^MW_{\mbf{G}}}$.

\begin{definition}
Let $w_n = 1 \times \prod_{\tau \in \Psi} (w_n)_{\tau} \in \mbf{G}(\ordd)$ be the element where $(w_n)_{\tau} = \opn{id}$ for $\tau \neq \tau_0$, and $(w_n)_{\tau_0}$ is the matrix
\[
[(w_n)_{\tau_0}]_{i, j} = \left\{ \begin{array}{cc} 1 & \text{ if } (i, j) = (1, n+1) \\ 1 & \text{ if } j=i-1, 2 \leq i \leq n+1 \\ 1 & \text{ if } i=j \geq n+2 \\ 0 & \text{ otherwise } \end{array} \right.
\]
\end{definition}

The following elements are key to the whole construction in this paper. 

\begin{definition}
Let $u_{\tau_0}' \in \opn{GL}_{2n-1}(\ordd)$ denote the matrix whose $(i, j)$-th element is
\[
(u_{\tau_0}')_{i, j} = \left\{ \begin{array}{cc} 1 & \text{ if } i=j \\ 1 & \text{ if } j=2n-i, i \leq n \\ 0 & \text{ otherwise } \end{array} \right.
\]
and we let $u_{\tau_0} = 1 \times u_{\tau_0}' \in \opn{GL}_1(\ordd) \times \opn{GL}_{2n-1}(\ordd)$. For $\tau \neq \tau_0$, we let $u_{\tau} \in \opn{GL}_{2n}(\ordd)$ denote the block matrix (with block size $(n \times n)$) given by
\[
u_{\tau} = \tbyt{1}{}{w_{\opn{GL}_n}^{\mathrm{max}}}{1}
\]
where $w_{\opn{GL}_n}^{\mathrm{max}}$ denotes the antidiagonal matrix with $1$s along the antidiagonal (which represents the longest Weyl element in $W_{\opn{GL}_n}$). We let $u \in M_{\mbf{G}}(\ordd)$ be the element $u = 1 \times \prod_{\tau \in \Psi} u_{\tau}$. 

Denote by $x_{\tau_0}$ the $(1 \times 2n-1)$-matrix whose first $n$ entries are $1$ and the rest are $0$. We let $\gamma_{\tau_0} \in \opn{GL}_{2n}(\ordd)$ denote the block matrix 
\[
\gamma_{\tau_0} = u_{\tau_0} \cdot \tbyt{1}{x_{\tau_0}}{}{1}
\]
and we set $\gamma_{\tau} = u_{\tau} \in \opn{GL}_{2n}(\ordd)$ for $\tau \neq \tau_0$. Define $\gamma \in P_{\mbf{G}}(\ordd)$ to be $\gamma \defeq 1 \times \prod_{\tau} \gamma_{\tau}$. 

Finally, we define $\widehat{\gamma} \defeq \gamma \cdot w_n \in \mbf{G}(\ordd)$ (with the specific choice of $w_n$ fixed above). 
\end{definition}

Here are some key properties of these elements. 

\begin{lemma} \label{OpenOrbitLemma}
\begin{enumerate}
    \item The orbit $M_{\mbf{H}} \cdot u \cdot B_{M_{\mbf{G}}}$ is Zariski open in $M_{\mbf{G}}$ (over $\opn{Spec}\ordd$), where $B_{M_{\mbf{G}}}$ denotes the standard Borel of $M_{\mbf{G}}$.
    \item The orbit $\mbf{H} \cdot \widehat{\gamma} \cdot B_{\mbf{G}}$ is Zariski open in $\mbf{G}$ (over $\opn{Spec}\ordd$).
\end{enumerate}
\end{lemma}
\begin{proof}
It is enough to check that the stabiliser $M_{\mbf{H}} \cap u B_{M_{\mbf{G}}} u^{-1}$ (resp. $\mbf{H} \cap \widehat{\gamma}B_{\mbf{G}} \widehat{\gamma}^{-1}$) for the action of $M_{\mbf{H}}$ (resp. $\mbf{H}$) on the flag variety $M_{\mbf{G}}/B_{M_\mbf{G}}$ (resp. $\mbf{G}/B_{\mbf{G}}$) has the required dimension. But an explicit calculation shows that
\begin{align*}
    M_{\mbf{H}} \cap u B_{M_{\mbf{G}}} u^{-1} &= \opn{GL}_1 \times \{ \opn{diag}(x_1, x_2, \dots, x_{n+1}, x_n, \dots, x_2) \in \opn{GL}_{1} \times \opn{GL}_{n-1} \times \opn{GL}_n \} \\ &\; \; \; \times \prod_{\tau \neq \tau_0}\{ \opn{diag}(y_1, \dots, y_n, y_n, \dots, y_1) \in \opn{GL}_n \times \opn{GL}_n \} 
\end{align*}
which proves part (1). For part (2), we separate the calculation into three separate cases depending on the decomposition of $\mbf{H}$ and $\mbf{G}$ into general linear groups, namely the $\opn{GL}_1$-component, the $\tau_0$-component and the $\tau$-component for $\tau \neq \tau_0$.

There is nothing to check for the $\opn{GL}_1$-component, and the $\tau \neq \tau_0$-component follows from the computation as in part (1). So we are left to prove the lemma for the $\tau_0$-component. One can find $X, Z \in \opn{GL}_n(\ordd)$, $Y$ an $n \times n$-matrix with entries in $\ordd$, such that:
\begin{itemize}
    \item $Z$ is upper triangular
    \item $X w_{\opn{GL}_n}^{\mathrm{max}} = U$ is block upper triangular and lies in the standard parabolic of $\opn{GL}_n$ with Levi $\opn{GL}_1 \times \opn{GL}_{n-1}$. Its projection to the Levi is $1 \times w_{\opn{GL}_{n-1}}^{\mathrm{max}}$. 
    \item One has the equality
    \[
    \widehat{\gamma}_{\tau_0} = \tbyt{X}{}{}{1} \tbyt{1}{}{w_{\opn{GL}_n}^{\mathrm{max}}}{1} \tbyt{1}{Y}{}{Z} .
    \]
\end{itemize}
We therefore find that, for $h = (A, B) \in \opn{GL}_n \times \opn{GL}_n$, $\widehat{\gamma}_{\tau_0}^{-1} h \widehat{\gamma}_{\tau_0}$ lies in the standard Borel of $\opn{GL}_{2n}$ if and only if $U^{-1}AU$ (resp. $B$) is lower (resp. upper triangular) and $B = U^{-1}AU$. This gives the required dimension for the stabiliser.
\end{proof}

\subsection{Level subgroups at \texorpdfstring{$p$}{p}} \label{LevelSubgroupsAtp}

Let $p$ be a prime which splits completely in $F/\mbb{Q}$, and fix an isomorphism $\mbb{C} \cong \Qpb$. Then, as in Remark \ref{RemarkGLIdentifications}, we have identifications
\begin{align*}
    G \defeq \mbf{G}_{\mbb{Q}_p} &= \opn{GL}_{1, \mbb{Q}_p} \times \prod_{\tau \in \Psi} \opn{GL}_{2n, \mbb{Q}_p} \\
    H \defeq \mbf{H}_{\mbb{Q}_p} &= \opn{GL}_{1, \mbb{Q}_p} \times \prod_{\tau \in \Psi} \left(\opn{GL}_{n, \mbb{Q}_p} \times \opn{GL}_{n, \mbb{Q}_p} \right).
\end{align*}

\begin{remark} \label{ChoiceOfIntModelsRemark}
Note that the choice of $\ordd$-models in the previous section give rise to $\mbb{Z}_p$-models of $G$, $H$, and the various subgroups under consideration. We will denote these models by the same letters. For various objects attached to $\mbf{G}$ and $\mbf{H}$, we will use non-bold letters to indicate their analogue for the groups $G$ and $H$. For example, will write $M_G$ for $M_{\mbf{G}, \mbb{Q}_p}$.
\end{remark}

We introduce the following level subgroups:

\begin{definition}
\begin{enumerate} 
\item For $t \geq 1$, let $K^G_{\mathrm{Iw}}(p^t) \subset G(\mbb{Z}_p)$ denote the depth $t$ upper triangular Iwahori of $G$, i.e. all elements in $G(\mbb{Z}_p)$ which land in $B_G$ modulo $p^t$. We also use the same definition for $H$.
\item For $t \geq 1$, we let $K^H_{\diamondsuit}(p^t) \subset H(\mbb{Q}_p)$ denote the subgroup $H(\mbb{Q}_p) \cap \widehat{\gamma} K^G_{\opn{Iw}}(p^t) \widehat{\gamma}^{-1}$, where $\widehat{\gamma}$ is treated as an element of $G(\mbb{Z}_p)$. 
\end{enumerate}
\end{definition}

We have the following:

\begin{lemma} \label{IndexLevelSubgroup}
The subgroup $K^H_{\diamondsuit}(p^t)$ is contained in $K^H_{\opn{Iw}}(p^t)$. Furthermore, one has
\[
[K^H_{\diamondsuit}(p^t) : K^H_{\diamondsuit}(p^{t+1})] = [K^G_{\opn{Iw}}(p^t) : K^G_{\opn{Iw}}(p^{t+1})] = p^{dn(2n-1)}
\]
where $d = [F^+ : \mbb{Q}]$.
\end{lemma}
\begin{proof}
For the first part, the computation for the $\opn{GL}_1$-component and $\tau \neq \tau_0$-component follows from the stabiliser computations in Lemma \ref{OpenOrbitLemma}. For the $\tau_0$-component, with notation as in Lemma \ref{OpenOrbitLemma}, we note that if $U^{-1}AU$ lies in the standard maximal torus modulo $p^t$, then $A$ lies in the depth $p^t$ Iwahori for $\opn{GL}_n$, because the Levi component of $U$ is $1 \times w_{\opn{GL}_{n-1}}^{\mathrm{max}}$ which normalises the maximal torus. The index calculation follows from a direct computation using the stabiliser descriptions in Lemma \ref{OpenOrbitLemma}.
\end{proof}

We will choose the level-at-$p$ of our Shimura varieties to be one of these subgroups; therefore we introduce the following notation.

\begin{notation}
For a fixed neat compact open subgroup $K^p \subset \mbf{G}(\mbb{A}_f^p)$, we set $S_{\mbf{G}, \mathrm{Iw}}(p^t)$ to be the Shimura variety of level $K^p K^G_{\mathrm{Iw}}(p^t)$. Similarly, for a fixed neat compact open subgroup $U^p \subset \mbf{H}(\mbb{A}_f^p)$, we let $S_{\mbf{H}, \diamondsuit}(p^t)$ and $S_{\mbf{H}, \mathrm{Iw}}(p^t)$ denote the Shimura varieties of levels $U^pK^H_{\diamondsuit}(p^t)$ and $U^pK^H_{\mathrm{Iw}}(p^t)$ respectively. If $U^p \subset K^p$, then we have a morphism
\[
\widehat{\iota} \colon S_{\mbf{H}, \diamondsuit}(p^t) \to S_{\mbf{G}, \mathrm{Iw}}(p^t)
\]
defined as the composition $\widehat{\gamma} \circ \iota$.
\end{notation}

\subsection{Branching laws} \label{PrelimBranchingLawSubSec}

To be able to construct the relevant pairing in coherent cohomology, we need to understand how representations of $M_{\mbf{G}}$ decompose after restricting them to $M_{\mbf{H}}$. For convenience, we recall that a general element of $M_{\mbf{H}}$ is of the form $(x; y_1, y_2, y_3; z_{1, \tau}, z_{2, \tau})$ where $\tau$ runs over $\Psi - \{\tau_0\}$ and
\begin{itemize}
    \item $x \in \opn{GL}_1$
    \item $y_1 \in \opn{GL}_1$, $y_2 \in \opn{GL}_{n-1}$ and $y_3 \in \opn{GL}_n$
    \item $z_{i, \tau} \in \opn{GL}_n$ for $i=1, 2$.
\end{itemize}
This description will be useful for describing characters of $M_{\mbf{H}}$. 

\begin{proposition} \label{ClassicalBranchingProp}
Let $\lambda = (c_0; c_{1, \tau}, \dots, c_{2n, \tau}) \in X^*(T/T_0)^+$ and $\kappa_n = w_n \star (-w_G^{\mathrm{max}}\lambda) = w_n \star \lambda$ as in \S \ref{DiscreteSeriesRepsSection}. Set $\kappa_n^* = -w_{M_G}^{\mathrm{max}} \kappa_n$ and let $V_{\kappa_n^*}$ denote the irreducible algebraic representation of $M_{\mbf{G}}$ with highest weight $\kappa_n^*$. Let $j = (j_{\tau})_{\tau \in \Psi - \{\tau_0\}}$ be a tuple of integers satisfying $|j_{\tau}| \leq c_{n, \tau}$. Then there exists a unique up to scaling vector $v_{\kappa_n}^{[j]} \in V_{\kappa_n^*}$ such that $M_{\mbf{H}}$ acts on $v_{\kappa_n}^{[j]}$ through the character
\begin{eqnarray}
    M_{\mbf{H}} &\to& \mbb{G}_m \nonumber \\
    (x; y_1, y_2, y_3; z_{1, \tau}, z_{2, \tau}) &\mapsto& y_1^{n+c_{n, \tau_0}} \opn{det}y_2^{c_{n, \tau_0}} \opn{det}y_3^{-(c_{n, \tau_0}+1)} \prod_{\tau \neq \tau_0} \opn{det}z_{1, \tau}^{j_{\tau}} \opn{det}z_{2, \tau}^{-j_{\tau}} . \label{InverseOfSigma}
\end{eqnarray}
\end{proposition}
\begin{proof}
This follows from \cite[Theorem 2.1]{Knapp} (see also Appendix \ref{AppendixBranchingLaws}).
\end{proof}

\begin{remark}
We fix a specific model of $V_{\kappa_n^*}$ namely the space of algebraic functions $f \colon M_\mbf{G} \to \mbb{A}^1$ which transform as 
\[
f(gb) = \kappa_n(b)f(g)
\]
for all $g \in M_\mbf{G}$ and $b \in B_{M_\mbf{G}}$. The action of $m \in M_{\mbf{G}}$ is then given by $(m \cdot f)(g) = f(m^{-1} g)$. Since $M_{\mbf{H}} \cdot u \cdot B_{M_{\mbf{G}}}$ is Zariski dense in $M_{\mbf{G}}$ (Lemma \ref{OpenOrbitLemma}), we can (and do) normalise $v_{\kappa_n}^{[j]}$ so that its value on $u$ is $1$.
\end{remark}

Let $\sigma^{[j]}_n$ denote the inverse of the character in (\ref{InverseOfSigma}). Then after fixing an isomorphism $V_{\kappa_n^*} \cong V_{\kappa_n}^*$, we obtain a $M_{\mbf{H}}$-equivariant linear map
\[
V_{\kappa_n} \twoheadrightarrow \sigma^{[j]}_n .
\]
We can therefore consider the following $F^{\mathrm{cl}}$-bilinear pairing 
\[
\boxed{%
\langle , \rangle_{\mathrm{alg}} \colon \opn{H}^{n-1}\left( S_{\mbf{G}, \mathrm{Iw}}(p), [V_{\kappa_n}] \right) \times \opn{H}^0\left(S_{\mbf{H}, \diamondsuit}(p), [\sigma_n^{[j]}]^\vee \right) \to F^{\mathrm{cl}}
}
\]
defined as $\langle \eta, \chi \rangle_{\mathrm{alg}} = \opn{tr}( \widehat{\iota}^*\eta \cup \chi )$, where $\opn{tr}$ denotes the residue morphism 
\[
\opn{H}^{n-1}\left( S_{\mbf{H}, \diamondsuit}(p), \Omega^{n-1} \right) \to F^{\mathrm{cl}}.
\]
In \S \ref{ConstructionOfPadicLSection}, we will show that this recovers twisted unitary Friedberg--Jacquet periods when $\eta$ (resp. $\chi$) is associated with an automorphic representation of $\mbf{G}(\mbb{A})$ (resp. automorphic character of $\mbf{H}(\mbb{A})$). The goal of this paper is to $p$-adically interpolate this pairing.

%%%%%%%%%%%%%%%%%%%%%%%%%%%%%%%%%%%%%%%%%%%%%%%%%%%%%%%%%%%%%%%%%%%%%%%%%%%%%%%%%%%%%%%%%%
%%%%%%%%%%%%%%    FUNCTORIALITY ON THE FLAG VARIETY     %%%%%%%%%%%%%%%%%%%%%%%%%%%%%%%%%%
%%%%%%%%%%%%%%%%%%%%%%%%%%%%%%%%%%%%%%%%%%%%%%%%%%%%%%%%%%%%%%%%%%%%%%%%%%%%%%%%%%%%%%%%%%

\section{Functoriality on the flag variety}

In this section we consider the functoriality of higher Coleman theory on the level of flag varieties (over $\mbb{Z}_p$). This section is entirely local; in particular, we use notation and conventions as in \S \ref{LevelSubgroupsAtp} (so $G$ and $H$ denote the integral models in Remark \ref{ChoiceOfIntModelsRemark} for $\mbf{G}_{\mbb{Q}_p}$ and $\mbf{H}_{\mbb{Q}_p}$ respectively, etc.).

\begin{definition}
Let $\opn{FL}_G$ (resp. $\opn{FL}_H$) denote the flag variety $P_G \backslash G$ (resp. $P_H \backslash H$) over $\mbb{Z}_p$. This can be described as the space of row vectors in $\mbb{P}^{2n-1}$ (resp. $\mbb{P}^{n-1}$) with the action of $g \in G$ (resp. $h \in H$) given by 
\begin{align*}
    [x_0 : \cdots : x_{2n-1}] \star g &= [x_0 : \cdots : x_{2n-1}] \cdot {^t g^{-1}} \\ 
    [y_0 : \cdots : y_{n-1}] \star h &= [y_0 : \cdots : y_{n-1}] \cdot {^t h^{-1}} .
\end{align*}
The embedding $\opn{FL}_H \xrightarrow{\iota} \opn{FL}_G$ induced from $H \hookrightarrow G$ is described in coordinates as
\[
\iota( [y_0 : \cdots : y_{n-1}] ) = [y_0 : \cdots : y_{n-1} : 0 : \cdots : 0 ].
\]
\end{definition}

We will consider certain stratifications on these flag varieties, and relations between them. Recall that ${^MW_G}$ denotes the set of Kostant representatives for the quotient $W_{M_G} \backslash W_G$, where $W_{?}$ denotes the Weyl group of $?$. This can be described as 
\[
{^MW_G} = \{ w_0, \dots, w_{2n-1} \}
\]
where $l(w_i) = i$, and each $w_i$ corresponds to a shuffle and acts on the flag variety $\opn{FL}_G$ as 
\[
[x_0 : \cdots : x_{2n-1}] \star w_i = [x_1 : \cdots : x_{i} : x_0 : x_{i+1} : \cdots : x_{2n-1}]
\]
(the element $w_0$ acts as the identity). We have a similar description for $H$ and, as mentioned in \S \ref{PreliminarySection}, we have a map ${^MW_H} \hookrightarrow {^MW_G}$ induced from $H \hookrightarrow G$, preserving the lengths of the Weyl elements. 

\subsection{The Bruhat stratification}

For either $? = G, H$, we have the following stratification of $\opn{FL}_{?, \mbb{F}_p}$ given by the cells
\[
C^?_{w} = P_? \backslash P_? \cdot w \cdot B_?
\]
for $w \in {^MW_?}$. In coordinates, we have that $C^G_{w_i}$ is the orbit of $[0 : \cdots : 0 : 1 : 0 : \cdots : 0]$ (where the $1$ is in the $(i+1)$-th place) under the $\star$-action of $B_G$. Explicitly, this is described as the collection of tuples 
\[
[x_0 : \cdots : x_{i-1} : 1 : 0 : \cdots : 0], \quad \quad x_{j} \in \mbb{A}^1_{\mbb{F}_p} \text{ for } j=0, \dots, i-1 .
\]
Each cell $C^G_{w_i}$ has dimension $i$, and they are ordered as $C^G_{w'} \subset \overline{C^G_{w}}$ if and only if $l(w') \leq l(w)$. We have a similar description for $H$. 

\begin{definition}
For $? = G, H$ and $w \in {^MW_?}$, we set
\[
Y^?_w = \bigcup_{l(w') \geq l(w)} C^?_{w'}, \quad \quad \quad X^?_{w} = \bigcup_{l(w') \leq l(w)} C^?_{w'}. 
\]
The former is open in $\opn{FL}_{?, \mbb{F}_p}$, the latter is closed, and one has the relation $C^?_w = Y^?_w \cap X^?_w$.
\end{definition}

Recall the definition of $\widehat{\gamma}$ in \S \ref{SomeImportantElementsSubsec}, which we view as an element of $G(\mbb{Z}_p)$. Let $\widehat{\iota} \colon \opn{FL}_H \to \opn{FL}_G$ denote the map given by $P_H \cdot h \mapsto P_G \cdot h \widehat{\gamma}$. This map satisfies the following properties:

\begin{lemma} \label{GeneralPositionOfCells}
One has
\begin{enumerate}
    \item $\widehat{\iota}^{-1}\left(C^G_{w_i} \right) = \varnothing$ if $i < n$.
    \item $\widehat{\iota}^{-1}\left(C^G_{w_n} \right) = C^H_{\mathrm{id}}$.
\end{enumerate}
\end{lemma}
\begin{proof}
In coordinates, the map $\widehat{\iota}$ is given by
\[
\widehat{\iota}([y_0 : \dots : y_{n-1}]) = [y_1 : y_2 : \cdots : y_{n-1} : 0 : y_0 - \sum_{i=1}^{n-1}y_i, -y_{n-1} : \cdots : -y_1].
\]
The result immediately follows from this and the description of $C^G_{w_i}$ in coordinates.
\end{proof}

\subsection{Tubes in the flag variety} \label{TubesInFlagVar}

We recall some notation from \cite[\S 3.3]{BoxerPilloni} and \cite[\S 5.4]{LZBK21}. Suppose that $X/\mbb{Z}_p$ is a finite-type scheme and let $\invs{X} = X \times_{\opn{Spec}\mbb{Z}_p} \opn{Spa}(\mbb{Q}_p, \mbb{Z}_p)$ denote the associated adic space over $\opn{Spa}(\mbb{Q}_p, \mbb{Z}_p)$. Let $X_0$ denote the special fibre of $X$ over $\mbb{F}_p$. Then one has a specialisation map $\opn{sp} \colon \invs{X} \to X_0$, and for any locally closed subscheme $U \subset X_0$, we define the tube $]U[ \subset \invs{X}$ to be the interior of $\opn{sp}^{-1}(U)$.

\begin{definition}
For $m \in \mbb{Q}$, let $\invs{B}^{\circ}_m \subset \overline{\invs{B}}^{\circ}_m \subset \invs{B}_m \subset \overline{\invs{B}}_m$ denote the four flavours of ``disc'' inside the adic affine line defined as follows:
\[
\invs{B}_m = \{|\cdot | : |z| \leq |p|^m\}, \quad \overline{\invs{B}}_m = \bigcap_{m' < m} \invs{B}_{m'}, \quad \invs{B}^{\circ}_m = \bigcup_{m' > m} \invs{B}_{m'}, \quad \overline{\invs{B}}^{\circ}_m = \{ |\cdot | : |z| < |p|^m \}.
\]
\end{definition}

We let $\mathtt{FL}^G$ and $\mathtt{FL}^H$ denote the adic flag varieties (over $\opn{Spa}(\mbb{Q}_p, \mbb{Z}_p)$) associated with $\opn{FL}_G$ and $\opn{FL}_H$. For $? = G, H$, we let $\Phi^{\pm}$ denote the set of $\pm$-roots with respect to $B_{?}$, and set $\Phi^{-, M}$ to be the set of negative roots which are not contained in $M_?$. Set $\delta_H = n-1$ and $\delta_G = 2n-1$. Then, for $w \in {^MW_?}$, we set $U_{w} = C^?_{w_{\delta_?}} \cdot w_{\delta_?}^{-1}w$ which is an open set containing $C^?_{w}$. Let $\invs{U}_{w}^{\mathrm{an}}$ denote its analytification. Then, following \cite[\S 3.3.6]{BoxerPilloni}, we have an Iwahori decomposition:
\begin{align}
    \prod_{\alpha \in w^{-1} \Phi^{-, M}} \mbb{A}^{1, \mathrm{an}} &\xrightarrow{\sim} \invs{U}_w^{\mathrm{an}} \nonumber \\
    (u_{\alpha}) &\mapsto w\prod_{\alpha} u_{\alpha} . \label{IWUw}
\end{align}

\begin{definition}
Let $m, k \in \mbb{Q}$ and $w \in {^MW_?}$. We define $]C^?_w[_{m, k}$, $]C^?_w[_{\overline{m}, k}$, $]C^?_w[_{m, \overline{k}}$ and $]C^?_w[_{\overline{m}, \overline{k}}$ to be the images of 
\begin{align*}
\prod_{\alpha \in (w^{-1}\Phi^{-, M}) \cap \Phi^{-}} \invs{B}^{\circ}_m \times \prod_{\alpha \in (w^{-1}\Phi^{-, M}) \cap \Phi^{+}} \invs{B}_k \\
\prod_{\alpha \in (w^{-1}\Phi^{-, M}) \cap \Phi^{-}} \overline{\invs{B}}^{\circ}_m \times \prod_{\alpha \in (w^{-1}\Phi^{-, M}) \cap \Phi^{+}} \invs{B}_k \\
\prod_{\alpha \in (w^{-1}\Phi^{-, M}) \cap \Phi^{-}} \invs{B}^{\circ}_m \times \prod_{\alpha \in (w^{-1}\Phi^{-, M}) \cap \Phi^{+}} \overline{\invs{B}}_k \\
\prod_{\alpha \in (w^{-1}\Phi^{-, M}) \cap \Phi^{-}} \overline{\invs{B}}^{\circ}_m \times \prod_{\alpha \in (w^{-1}\Phi^{-, M}) \cap \Phi^{+}} \overline{\invs{B}}_k
\end{align*}
respectively, under the map (\ref{IWUw}).
\end{definition}

\begin{remark}
If $m, k \in \mbb{Q}_{\geq 0}$ then $]C^?_w[_{m, k} \subset ]C^?_w[$ with equality if $m=k=0$. If $m \geq k \geq 0$, then $]C^?_{w_i}[_{m, k}$ is described in coordinates as the subset of tuples
\[
[y_0 : \cdots : y_{\delta_?}]
\]
satisfying 
\[
y_j \in \left\{ \begin{array}{cc} \invs{B}_k & \text{ if } j < i \\ 1 + \invs{B}^{\circ}_m & \text{ if } j=i \\ \invs{B}^{\circ}_m & \text{ if } j > i \end{array} \right.
\]
One has a similar description for $]C^?_{w_i}[_{\overline{m}, k}$ by replacing $\invs{B}^{\circ}_m$ with $\overline{\invs{B}}^{\circ}_m$, and a similar description for $]C^?_{w_i}[_{m, \overline{k}}$ when $k > 0$ by replacing $\invs{B}_k$ with $\overline{\invs{B}}_k$ (see \cite[\S 3.3.10]{BoxerPilloni}). In particular, if $i=0$ (so $w_0 = \opn{id}$) then these tubes do not depend on $k$, so we will drop it from the notation.
\end{remark}

We will now make specific choices of tubes which will be relevant for the construction of the $p$-adic $L$-function. Throughout, we let $m, k, t$ be integers satisfying 
\begin{equation} 
0 \leq k \leq m < t, \quad \text{ with } m > k \text{ if } k \neq 0. \label{mktTriple}
\end{equation}
We also introduce the following stronger condition:
\begin{equation}
    m, k, t \text{ as in (\ref{mktTriple}) with } m > (2n-1)(k+1) \text{ and } t > m+k. \label{mktstrong}
\end{equation}
We define some tubes in $\mathtt{FL}^G$ as follows. 

\begin{definition}
Let $m, k, t$ be as in (\ref{mktTriple}).
\begin{enumerate}
    \item Let $\mathtt{U}^G_0 = ]Y_{w_n}^G[$, $\mathtt{Z}^G_0 = \overline{]X^G_{w_n}[}$ and $\mathtt{I}^G_{0, 0} = \mathtt{U}^G_0 \cap \mathtt{Z}^G_0$.
    \item We define $\mathtt{I}^G_{m, k} = ]C_{w_n}^G[_{\overline{m}, k} \cdot K^G_{\opn{Iw}}(p^t)$, which is independent of $t$ by the description in \cite[\S 3.3.10]{BoxerPilloni}.
    \item For $k \geq 1$, we define $\mathtt{U}^G_{k} = ]C_{w_n}^G[_{k, k} \cdot K^G_{\opn{Iw}}(p^t)$, which is independent of $t$ by the description in \emph{loc.cit.}. Furthermore, we have $\mathtt{I}^G_{m, k} \subset \mathtt{U}^G_{k}$.
\end{enumerate}
\end{definition}

We now define some tubes for $H$.

\begin{definition}
\begin{enumerate}
    \item For $m \geq 0$ and $t \geq 1$, one defines 
    \[
    \mathtt{Z}^H_m = ]C^H_{\mathrm{id}}[_{\overline{m}} \cdot K^H_{\diamondsuit}(p^t)
    \]
    which is equal to $]C^H_{\mathrm{id}}[_{\overline{m}}$ for $t > m$.
    \item For $k \geq 1$ and $t \geq 1$, we define
    \[
    \mathtt{U}^H_k = ]C^H_{\mathrm{id}}[_k \cdot K^H_{\diamondsuit}(p^t)
    \]
    which is equal to $]C^H_{\mathrm{id}}[_k$ for $t > k$. For $k=0$, we define $\mathtt{U}^H_0 = \mathtt{FL}^H$.
\end{enumerate}
\end{definition}

We obtain the following lemma, essentially by construction:

\begin{lemma} \label{CartesianSquareStrataLemma}
For $m, t, k$ as in (\ref{mktTriple}), one has $\mathtt{U}^H_k = \widehat{\iota}^{-1}\left(\mathtt{U}^G_k \right)$ and $\mathtt{Z}^H_m = \widehat{\iota}^{-1}\left(\mathtt{I}^G_{m,k} \right)$. Furthermore, there is a Cartesian diagram
\[
\begin{tikzcd}
\mathtt{Z}^H_m \arrow[d] \arrow[r, hook] & \mathtt{U}^H_k \arrow[d] \\
{\mathtt{I}^G_{m, k}} \arrow[r, hook]    & \mathtt{U}^G_k          
\end{tikzcd}
\]
with each map a closed embedding.
\end{lemma}
\begin{proof}
The lemma is clear for $(m, k) = (0, 0)$ by Lemma \ref{GeneralPositionOfCells}; so assume that $(m, k) \neq (0, 0)$. Then we can express $\mathtt{I}^G_{m, k}$ as the intersection
\[
\mathtt{I}^G_{m, k} = ]C^G_{w_n}[_{k, k} \cdot K^G_{\mathrm{Iw}}(p^t) \cap ]C^G_{w_n}[_{\overline{m}, \overline{0}} \cdot K^G_{\mathrm{Iw}}(p^t) .
\]
Indeed, the group $K^G_{\mathrm{Iw}}(p^t)$ acts continuously and preserves $]C^G_{w_n}[_{\overline{m}, 0}$, so must also preserve $\overline{]C^G_{w_n}[_{\overline{m}, 0}} = ]C^G_{w_n}[_{\overline{m}, \overline{0}}$. One then follows the proof of \cite[Lemma 3.3.17]{BoxerPilloni}. 

The above description implies that $\mathtt{I}^G_{m, k}$ is closed in $\mathtt{U}^G_{k}$. Furthermore, the map $\widehat{\iota}$ is a closed embedding of flag varieties, therefore it is enough to check $\mathtt{U}^H_k = \widehat{\iota}^{-1}\left(\mathtt{U}^G_k \right)$ and $\mathtt{Z}^H_m = \widehat{\iota}^{-1}\left(\mathtt{I}^G_{m,k} \right)$. But this follows immediately from the explicit description involving coordinates, and the fact that $\widehat{\iota}(\mathtt{U}^H_k) \subset ]C^G_{w_n}[_{k, k}$ and $\widehat{\iota}(\mathtt{Z}^H_m) \subset ]C^G_{w_n}[_{\overline{m}, k}$ for $(m, k) \neq (0, 0)$.
\end{proof}

%%%%%%%%%%%%%%%%%%%%%%%%%%%%%%%%%%%%%%%%%%%%%%%%%%%%%%%%%%%%%%%%%%%%%%%%%%%%%%%%%%%%%%%%%%%%%
%%%%%%%%    PULLBACKS ON ADIC SHIMURA VARIETIES   %%%%%%%%%%%%%%%%%%%%%%%%%%%%%%%%%%%%%%%%%%%
%%%%%%%%%%%%%%%%%%%%%%%%%%%%%%%%%%%%%%%%%%%%%%%%%%%%%%%%%%%%%%%%%%%%%%%%%%%%%%%%%%%%%%%%%%%%%

\section{Pullbacks on adic Shimura varieties} \label{PullbacksOnAdicSVs}

We now transfer the functoriality of the last section to the setting of adic Shimura varieties, via the Hodge--Tate period map. We fix a neat compact open subgroup $K^p \subset \mbf{G}(\mbb{A}_f^p)$, and let $K = K^pK_p$ for a compact open subgroup $K_p \subset G(\mbb{Q}_p)$. Let $\invs{S}_{G, K} = \invs{S}_{G, K}^{\mathrm{an}}$ denote the adic Shimura variety over $\opn{Spa}(\mbb{Q}_p, \mbb{Z}_p) = \opn{Spa}(F_{\ide{p}_{\tau_0}}, \ordd_{F_{\ide{p}_{\tau_0}}})$ associated with $S_{\mbf{G}, K}$ (note our assumption $F^+ \neq \mbb{Q}$ implies that $S_{\mbf{G}, K}$ is proper). Similarly, we fix a neat compact open subgroup $U^p \subset \mbf{H}(\mbb{A}_f^p)$ contained in $K^p$, and we let $\invs{S}_{H, U}$ denote the corresponding adic Shimura variety of level $U = U^pU_p$. If we choose $K_p = K^G_{\mathrm{Iw}}(p^t)$ or $U_p = K^H_{\mathrm{Iw}}(p^t), K^H_{\diamondsuit}(p^t)$ then we will use the notation $\invs{S}_{G, \mathrm{Iw}}(p^t)$, $\invs{S}_{H, \mathrm{Iw}}(p^t)$ and $\invs{S}_{H, \diamondsuit}(p^t)$ respectively.

\subsection{The Hodge--Tate period map}

Since $(\mbf{G}, h_{\mbf{G}})$ defines a PEL-type (and hence Hodge-type) Shimura datum, there exists a perfectoid space $\invs{S}_{G, K^p}$ over $\mbb{Q}_p$ which represents the diamond $\varprojlim_{K_p} \invs{S}_{G, K}$. In fact, the existence of such a perfectoid space does not require axiom (SV3), i.e. $\mbf{G}^{\mathrm{ad}}(\mbb{R})$ has no $\mbb{Q}$-simple factors which are $\mbb{R}$-anisotropic, provided that one has embedding into a Siegel datum. This leads to the following proposition:

\begin{proposition}
There exists a perfectoid space $\invs{S}_{H, U^p}$ over $\mbb{Q}_p$ which represents the diamond $\varprojlim_{U_p} \invs{S}_{H, U}$.
\end{proposition}
\begin{proof}
Although the set-up is slightly different, this follows the proof of \cite[Theorem IV.1.1]{ScholzeTorsion} verbatim. Note that we do not need a description of the connected components of $S_{\mbf{H}, U}$ in terms of Shimura data for the group $\mbf{H}^{\mathrm{der}}$ (this would require (SV3)).
\end{proof}

Both of these perfectoid spaces come equipped with a Hodge--Tate period map into a flag variety associated with the ambient Siegel datum. It is shown in \cite{CS17} that one can refine this morphism so that its image is contained in a flag variety associated with $G$ or $H$. In particular, since the same Siegel datum can be chosen for $\mbf{G}$ and $\mbf{H}$ (compatible with the embedding $\iota \colon \mbf{H} \hookrightarrow \mbf{G}$), one has a commutative diagram
\[
\begin{tikzcd}
{\invs{S}_{G, K^p}} \arrow[r, "\pi_{\mathrm{HT}, G}"]           & \mathtt{FL}^G           \\
{\invs{S}_{H, U^p}} \arrow[u] \arrow[r, "\pi_{\mathrm{HT}, H}"] & \mathtt{FL}^H \arrow[u]
\end{tikzcd}
\]
where the vertical arrows are the natural ones (induced from $\iota$) and $\pi_{\mathrm{HT}}$ denotes the Hodge--Tate period map. We will often drop the subscripts for $\pi_{\mathrm{HT}}$ when the context is clear. Since $\pi_{\mathrm{HT}, G}$ is $G(\mbb{Q}_p)$-equivariant, the twisted embedding $\widehat{\iota} \colon \invs{S}_{H, \diamondsuit}(p^t) \to \invs{S}_{G, \mathrm{Iw}}(p^t)$ commutes with the twisted morphism
\begin{align*}
    \widehat{\iota} \colon \mathtt{FL}^H/K^H_{\diamondsuit}(p^t) &\to \mathtt{FL}^G/K^G_{\mathrm{Iw}}(p^t) \\
    x K^H_{\diamondsuit}(p^t) &\mapsto \widehat{\iota}(x) K^G_{\mathrm{Iw}}(p^t)
\end{align*}
via the Hodge--Tate period morphisms. This is of course well-defined because $\widehat{\gamma}^{-1}K^H_{\diamondsuit}(p^t) \widehat{\gamma} \subset K^G_{\mathrm{Iw}}(p^t)$.

\subsection{Twisting torsors} \label{TwistingTorsorsSubSec}

In this section, we describe a general procedure for Tate-twisting pro\'{e}tale torsors and record some properties of this construction. Our choice of convention for twisting below will be consistent with our convention for the torsors on Shimura varieties (namely that they are defined via frames of relative \emph{homology} groups). 

Let $L/\mbb{Q}_p$ be a finite extension and $\invs{X}/L$ a smooth adic space. Let $\invs{T}^{\times} \to \invs{X}$ denote the pro\'{e}tale $\mbb{Z}_p^{\times}$-torsor parameterising isomorphisms (of pro\'{e}tale sheaves) $\mbb{Z}_p \xrightarrow{\sim} \mbb{Z}_p(1)$. The action of $\mbb{Z}_p^{\times}$ is given by precomposition, i.e. for $\lambda \in \mbb{Z}_p^{\times}$ and $\phi \colon \mbb{Z}_p \xrightarrow{\sim} \mbb{Z}_p(1)$, we set
\[
\phi \cdot \lambda = \phi(\lambda \cdot - ) .
\]
Let $M$ be a smooth adic group scheme over $\opn{Spa}L$ and suppose that we have a homomorphism
\[
\mu \colon \mbb{Z}_p^{\times} \to M
\]
that is central (i.e. its image is contained in the centre of $M$).

\begin{definition} \label{DefOfTwistedTorsor}
Let $\invs{M} \to \invs{X}$ be a (right) pro\'{e}tale $M$-torsor. We define the twist of $\invs{M}$ along $\mu$ to be 
\[
{^\mu \invs{M}} \defeq \invs{M} \times^{[\mbb{Z}_p^{\times}, \mu]} \invs{T}^{\times} 
\]
where the right-hand side denotes the quotient of $\invs{M} \times_{\invs{X}} \invs{T}^{\times}$ by the equivalence relation:
\[
(m \cdot \mu(\lambda), \phi) \sim (m, \phi \cdot \lambda^{-1} ), \quad \quad \text{ for all } m \in \invs{M}, \; \phi \in \invs{T}^{\times}, \; \lambda \in \mbb{Z}_p^{\times} .
\]
This defines a pro\'{e}tale $M$-torsor ${^\mu \invs{M}} \to \invs{X}$ via the action $(m, \phi) \cdot n = (m \cdot n, \phi)$, for $m \in \invs{M}$, $\phi \in \invs{T}^{\times}$ and $n \in M$, because the homomorphism $\mu$ is central.
\end{definition}

\begin{example}
Suppose that $M = \mbb{G}_m^{\opn{an}}$ and $\mu \colon \mbb{Z}_p^{\times} \to M$ is the natural inclusion. Let $\mathscr{F}$ be a locally free sheaf of rank one on the pro\'{e}tale site of $\invs{X}$. Then $\invs{M} \defeq \underline{\opn{Isom}}(\hat{\ordd}_{\invs{X}}, \mathscr{F})$ is a pro\'{e}tale $M$-torsor, and we have a natural identification
\[
{^\mu \invs{M}} = \underline{\opn{Isom}}(\hat{\ordd}_{\invs{X}}, \mathscr{F}(-1)) .
\]
\end{example}

This twisting procedure enjoys the following properties:

\begin{lemma} \label{PropertiesOfTTorsorLemma}
\begin{enumerate}
    \item The construction ${^\mu \invs{M}}$ is functorial in the (right) pro\'{e}tale torsor $\invs{M}$.
    \item If $f \colon \invs{Y} \to \invs{X}$ is a morphism of smooth adic spaces over $\opn{Spa}L$, then
    \[
    f^*\left( {^\mu \invs{M}} \right) \cong {^\mu \left(f^* \invs{M} \right)}
    \]
    canonically (i.e. we have a natural isomorphism $f^* \circ {^\mu (-)} \xrightarrow{\sim} {^\mu (-)} \circ f^*$).
    \item If $N \subset M$ is a smooth subgroup and $\mu$ factors through $N$, then for any pro\'{e}tale $N$-torsor $\invs{N} \to \invs{X}$ one has
    \[
    {^\mu \left( \invs{N} \times^N M \right)} \cong {^\mu \invs{N}} \times^N M
    \]
    canonically (i.e. it is natural in $\invs{N}$).
\end{enumerate}
\end{lemma}
\begin{proof}
All of these properties follow immediately from tracing through the definitions.
\end{proof}

\subsection{Torsors on adic Shimura varieties}

We would like to recover the construction of the automorphic vector bundles in \S \ref{AutomorphicVectorBundlesSection} via the Hodge--Tate period morphism (which plays the role of the Borel embedding). This is accomplished in \cite[\S 2]{CS17}, and we give a brief review of the results. We will describe the construction for the group $\mbf{G}$ only, as the construction for $\mbf{H}$ follows the same argument. 

Let $\invs{M}_G$ denote the adic generic fibre associated with $M_G$ (the adic generic fibre of its completion along the special fibre) and $\invs{M}_G^{\mathrm{an}} = M_G^{\mathrm{an}}$. Let $\mu \colon \mbb{Z}_p^{\times} \to \invs{M}_G$ denote the (central) homomorphism induced from the Hodge cocharacter $\mu_{\mbf{G}}$ defined in \S \ref{ShimuraVarietiesSubSec}. By the results of \emph{loc.cit.}, there exists a pro\'{e}tale $\invs{M}_G^{\mathrm{an}}$-torsor $\invs{M}_{G,\mathrm{HT}}^{\mathrm{an}}$ over $\invs{S}_{G, K}$ such that its twist ${^\mu \invs{M}_{G,\mathrm{HT}}^{\mathrm{an}}}$ along $\mu$ is canonically isomorphic to $M_{G,\mathrm{dR}}$ under analytification.\footnote{See the paragraph preceding \cite[Lemma 2.3.5]{CS17} for the definition of this torsor (which in the notation of \emph{loc.cit.} would be $\invs{M}_{dR}$).} It is shown in \cite[\S 4.6]{BoxerPilloni} that $\invs{M}_{G,\mathrm{HT}}^{\mathrm{an}}$ has an integral structure, namely the pro\'{e}tale $\invs{M}_G$-torsor $\invs{M}_{G,\mathrm{HT}}$. By Lemma \ref{PropertiesOfTTorsorLemma}, this defines an integral structure ${^\mu \invs{M}_{G,\mathrm{HT}}}$ on ${^\mu \invs{M}_{G,\mathrm{HT}}^{\mathrm{an}}}$, which is an \'{e}tale $\invs{M}_G$-torsor because the morphism ${^\mu \invs{M}_{G,\mathrm{HT}}} \to \invs{S}_{G, K}$ is surjective on geometric points and smooth (as ${^\mu \invs{M}_{G,\mathrm{HT}}}$ is an open subset of ${^\mu \invs{M}_{G,\mathrm{HT}}^{\mathrm{an}}} = M_{G, \opn{dR}}^{\opn{an}}$).

On the other hand, if $N_G$ is the unipotent radical of $P_G$ with associated adic generic fibre $\invs{N}_G$, then one can consider the (right) $\invs{M}_G$-torsor
\[
\mathtt{M}^G \colon \invs{G}/\invs{N}_G \to \mathtt{FL}^G
\]
via the morphism $x \mapsto x^{-1}$. These torsors are related in the following way:

\begin{lemma} \label{PullbackOfTorsorDescription}
The pullback of $\invs{M}_{G,\mathrm{HT}}$ to the perfectoid space $\invs{S}_{G, K^p}$ is identified with $\pi_{\mathrm{HT}}^*\mathtt{M}^G$.
\end{lemma}
\begin{proof}
Immediate from the proof of \cite[Proposition 4.6.3]{BoxerPilloni}.
\end{proof}

Recall that we have a twisted morphism $\widehat{\iota} \colon \invs{S}_{H, \diamondsuit}(p^t) \to \invs{S}_{G, \mathrm{Iw}}(p^t)$. Also, recall that the choice of Hodge cocharacters $\mu_{\mbf{G}}$ and $\mu_{\mbf{H}}$ are compatible under the inclusion $\mbf{H} \hookrightarrow \mbf{G}$, therefore the homomorphism $\mu$ above factors through $\invs{M}_H$. The description in the above lemma gives the following reduction of structure.

\begin{proposition} \label{FirstReductionOfStructure}
One has a reduction of structure of pro\'{e}tale torsors over $\invs{S}_{H, \diamondsuit}(p^t)$
\[
\widehat{\iota}^* \invs{M}_{G,\mathrm{HT}} = \invs{M}_{H,\mathrm{HT}} \times^{[\invs{M}_{H}, u]} \invs{M}_G
\]
where the superscript means we view $\invs{M}_H$ as a subgroup of $\invs{M}_G$ via the embedding $u^{-1}\invs{M}_H u \subset \invs{M}_G$. In particular, one has a reduction of structure of \'{e}tale torsors
\[
\widehat{\iota}^* \left( {^\mu \invs{M}_{G,\mathrm{HT}}} \right) = {^\mu \invs{M}_{H,\mathrm{HT}}} \times^{[\invs{M}_{H}, u]} \invs{M}_G .
\]
\end{proposition}
\begin{proof}
For the first part and via the interpretation in Lemma \ref{PullbackOfTorsorDescription}, it is enough to show that $\widehat{\iota}^* \mathtt{M}^G = \mathtt{M}^G \times^{[\invs{M}_H, u]} \invs{M}_G$ on the level of flag varieties. This follows from the following commutative diagram:
\[
\begin{tikzcd}
\invs{H}/\invs{N}_H \arrow[d] \arrow[r]    & \invs{G}/\invs{N}_G \arrow[d] \\
\mathtt{FL}^H \arrow[r, "\widehat{\iota}"] & \mathtt{FL}^G                
\end{tikzcd}
\]
where the vertical arrows are the torsors $\mathtt{M}^H$ and $\mathtt{M}^G$ and the top horizontal map is given by
\[
h \invs{N}_H \mapsto \widehat{\gamma}^{-1} h \gamma \invs{N}_G  = \widehat{\gamma}^{-1} h u \invs{N}_G.
\]
where the last equality follows from the fact that $\gamma$ maps to $u$ under the projection $\invs{P}_G \twoheadrightarrow \invs{M}_G$. 

The last part of the proposition follows from the functoriality properties of twisted torsors in Lemma \ref{PropertiesOfTTorsorLemma} and the fact $\mu$ is central (so is unaffected by conjugation by $u$).
\end{proof}

\begin{remark} \label{AlternativeReductionOfStructure}
One has an alternative reduction of structure as follows. In this remark only, set $U = U^p K^H_{\diamondsuit}(p^t)$, $K = K^p K^G_{\mathrm{Iw}}(p^t)$ and $K_{\widehat{\gamma}} = \widehat{\gamma}K \widehat{\gamma}^{-1}$, and we will include the level in the notation for $\invs{M}_{\mathrm{HT}}$ and $M_{\mathrm{dR}}$. Then we obtain a twisted morphism $\widehat{\iota} \colon {^\mu \invs{M}_{H,\mathrm{HT}, U}^{\mathrm{an}} } \to {^\mu \invs{M}_{G,\mathrm{HT}, K}^{\mathrm{an}}}$ defined as the analytification of the composition
\[
M_{H,\mathrm{dR}, U} \xrightarrow{\iota} M_{G,\mathrm{dR}, K_{\widehat{\gamma}}} \xrightarrow{\widehat{\gamma}} M_{G,\mathrm{dR}, K}.
\]
This is simply the twist along $\mu$ of the morphism of torsors induced from the natural map on the level of flag varieties $\invs{H}/\invs{N}_H \to \invs{G}/\invs{N}_G$ sending $h \invs{N}_H$ to $\widehat{\gamma}^{-1}h \invs{N}_G$ (see Appendix \ref{AppendixComparisonsFamilies}), so in fact preserves the integral structure. This gives a reduction of structure 
\begin{equation} \label{Creduction}
    \widehat{\iota}^*\left( {^\mu \invs{M}_{G,\mathrm{HT}, K} } \right) = {^\mu \invs{M}_{H,\mathrm{HT}, U}} \times^{\invs{M}_H} \invs{M}_G
\end{equation}
and we have a commutative diagram
\[
\begin{tikzcd}
                                                                        &  & {{^\mu\invs{M}_{H,\mathrm{HT}, U}} \times^{\invs{M}_H} \invs{M}_G} \arrow[dd] \\
{\widehat{\iota}^*\left({^\mu \invs{M}_{G,\mathrm{HT}, K}} \right)} \arrow[rru] \arrow[rrd] &  &                                                                          \\
                                                                        &  & {{^\mu \invs{M}_{H,\mathrm{HT}, U}} \times^{[\invs{M}_H, u]} \invs{M}_G}      
\end{tikzcd}
\]
where every map is an isomorphism; the top diagonal map is the reduction of structure in (\ref{Creduction}), the bottom diagonal map is the reduction of structure in Proposition \ref{FirstReductionOfStructure}, and the vertical map is given by $[x, m] \mapsto [x, u^{-1}mu]$.

The reduction of structure in (\ref{Creduction}) will be useful for the comparison with the archimedean setting, whereas the reduction of structure in Proposition \ref{FirstReductionOfStructure} will be useful when we speak about sheaves of distributions in \S \ref{LocallyAnalyticCohomologySection}. 
\end{remark}

\subsection{Comparison with the archimedean pairing}

We can now reinterpret the pairing at the end of \S \ref{PreliminarySection} in the setting of adic Shimura varieties via rigid GAGA. For a representation $V \in \opn{Rep}(M_G)$ we let $[V]$ denote the associated bundle on $\invs{S}_{G, K}$ using the torsor ${^\mu \invs{M}_{G,\mathrm{HT}}^{\opn{an}}}= M_{G, \opn{dR}}^{\opn{an}}$; and similarly for $H$. We place ourselves in the setting of \S \ref{PrelimBranchingLawSubSec} -- in particular, we let $\lambda \in X^*(T/T_0)^+$. Then, after fixing an isomorphism $V_{\kappa_n^*} \cong V_{\kappa_n}^*$ we obtain a $M_H$-equivariant morphism
\begin{equation} \label{MapOfRepsTwistedEqn}
V_{\kappa_n} \twoheadrightarrow \sigma_n^{[j]}
\end{equation}
by pairing with the vector $u^{-1} \cdot v_{\kappa_n}^{[j]}$, where $M_H$ acts on $V_{\kappa_n}$ via the embedding $u^{-1}M_H u \subset M_G$. Via the reduction of structure in Proposition \ref{FirstReductionOfStructure}, this gives a morphism of sheaves
\[
\widehat{\iota}^*[V_{\kappa_n}] \to [\sigma_n^{[j]}]
\]
over $\invs{S}_{H, \diamondsuit}(p)$. Using this morphism, we therefore obtain a pairing 
\[
\boxed{ 
\langle , \rangle_{\mathrm{an}} \colon \opn{H}^{n-1}\left( \invs{S}_{G, \mathrm{Iw}}(p), [V_{\kappa_n}]\right) \times \opn{H}^0\left(\invs{S}_{H, \diamondsuit}(p), [\sigma_n^{[j]}]^\vee \right) \to \mbb{Q}_p
}
\]
defined as $\langle \eta, \chi \rangle_{\mathrm{an}} = \opn{tr}(\widehat{\iota}^*\eta \cup \chi)$. By the discussion in Remark \ref{AlternativeReductionOfStructure} and the fact that the analytification of $M_{\mathrm{dR}}$ is identified with ${^\mu \invs{M}_{\mathrm{HT}}^{\mathrm{an}}}$, we obtain the following proposition:

\begin{proposition} \label{GAGApairingProp}
The pairings $\langle , \rangle_{\mathrm{alg}}$ and $\langle , \rangle_{\mathrm{an}}$ correspond to each other under rigid GAGA, where we have base-changed the former to $\mbb{Q}_p$ via the embedding $F^{\mathrm{cl}} \hookrightarrow \mbb{Q}_p$ induced from the fixed isomorphism $\mbb{C} \cong \Qpb$.
\end{proposition}

\subsection{Hecke operators} \label{HeckeOperators}

We would like to restrict the pairing $\langle, \rangle_{\mathrm{an}}$ to one over certain strata in the adic Shimura varieties, without losing any information. To accomplish this, we need to pass to ``small-slope'' parts of cohomology with respect to the action of certain Hecke operators, which we will now describe.

Let $T^{-} \subset T(\mbb{Q}_p)$ denote the submonoid defined as 
\[
T^{-} = \{ x \in T(\mbb{Q}_p) : v(\alpha(x)) \leq 0 \text{ for all } \alpha \in \Phi^+ \}
\]
where $\Phi^+$ is the set of positive roots of $G$ (with respect to $B_G$) and $v \colon \mbb{Q}_p^{\times} \to \mbb{Z}$ is the $p$-adic valuation, normalised so that $v(p) = 1$. We let $T^{--} \subset T^-$ be the subset of elements satisfying $v(\alpha(x)) < 0$ for all $\alpha \in \Phi^+$. For $t \geq 1$, We let $\invs{H}^{-}_{p, t}$ denote the algebra $\mbb{Q}_p[K^G_{\mathrm{Iw}}(p^t) \backslash T^{-} / K^G_{\mathrm{Iw}}(p^t)]$ with multiplication given by the double coset description in \cite[\S 4.2]{BoxerPilloni}. This is isomorphic to the algebra $\mbb{Q}_p[T^{-}]$ (with the usual definition of multiplication), with an element $x \in T^-$ corresponding to $[K^G_{\mathrm{Iw}}(p^t) x K^G_{\mathrm{Iw}}(p^t)]$.

We fix a specific choice of Hecke operator.

\begin{definition} \label{BorelHeckeOperator}
Let $\lambda$ be an algebraic character of $H_{\infty}$ (see \S \ref{DiscreteSeriesRepsSection}) and set $\lambda^* = -w_G^{\mathrm{max}}\lambda$. We let $\invs{U}_B'(p^t) \in \invs{H}^{-}_{p, t}$ denote the Hecke operator $\lambda^*(x^{-1})[K^G_{\mathrm{Iw}}(p^t) x K^G_{\mathrm{Iw}}(p^t)]$ where $x \in T^{--}$ is given by
\[
x = (1; 1, p, p^2, \dots, p^{2n-1})_{\tau \in \Psi} .
\]
\end{definition}
\begin{remark}
It will turn out that the action of $\invs{U}'_B(p^t)$ on cohomology will be independent of the level, so we will often write $\invs{U}'_B$ instead.
\end{remark}

Note that a $\mbb{Q}_p$-algebra homomorphism $\invs{H}^{-}_{p, t} \to \Qpb$ is identified with a monoid homomorphism $\theta \colon T^{-} \to (\Qpb, \times)$ via the isomorphism above. We say that $\theta$ is \emph{finite-slope} if $\theta(x) \neq 0$ for some $x \in T^{--}$ (in fact, this implies $\theta(x) \neq 0$ for all $x \in T^{-}$).

\begin{definition} \label{DefOfFiniteSlopeForM}
Let $M$ be a Banach $\mbb{Q}_p$-module (or more generally, a bounded complex of projective Banach $\mbb{Q}_p$-modules) with an action of a potent compact operator $T$ (see \cite[Definition 2.4.13]{BoxerPilloni}). Then $M$ has a slope decomposition with respect to (some power of) $T$ and we set
\[
M^{\mathrm{fs}} \defeq \opn{colim}_{h} M^{\leq h}
\]
where the colimit is over $h \in \mbb{Q}_{\geq 0}$. This is called the finite-slope part of $M$.
\end{definition}

If $M$ carries an action of $\invs{H}^{-}_{p, t}$ such that $[K^G_{\mathrm{Iw}}(p^t) x K^G_{\mathrm{Iw}}(p^t)]$ acts as a potent compact operator for some $x \in T^{--}$, then we denote the finite-slope part by $M^{-, \mathrm{fs}}$ (which is independent of $x$ by \cite[Lemma 5.1.7]{BoxerPilloni}). Furthermore, $M^{\leq h}$ can be decomposed into generalised eigenspaces for the action of $T^{-}$, for any $h \in \mbb{Q}_{\geq 0}$ (since slope decompositions are unique and $M^{\leq h}$ is finite-dimensional). This will allow us to pass to the ``small slope part'' of $M$, in the following sense.

\begin{definition}
Let $\lambda \in X^*(T/T_0)^+$. We say that a (monoid) homomorphism $\theta \colon T^- \to \Qpb^{\times}$ is \emph{small slope} (with respect to $\kappa_n$) if, for every $w \in {^MW_G} - \{w_n \}$, there exists $x \in T^-$ such that
\begin{equation} \label{sssIneq}
v(\theta(x)) < v\left((w^{-1} \star \kappa_n)(x) \right) . 
\end{equation}
If $M$ is as in the paragraph following Definition \ref{DefOfFiniteSlopeForM}, then we let $M^{-, \mathrm{ss}(\kappa_n)}$ denote the sum of generalised eigenspaces in $M^{\leq h}$ for which $T^{-}$ acts through a small slope homomorphism $\theta \colon T^{-} \to \Qpb^{\times}$ (for any sufficiently large $h$ depending on $\kappa_n$). We will write $M^{-, \mathrm{ss}}$ when $\kappa_n$ is clear from the context. 
\end{definition}

\subsection{Restriction to smaller strata}

We transfer the strata in \S \ref{TubesInFlagVar} to adic Shimura varieties via the Hodge--Tate period map.

\begin{definition}
For $m, t, k$ as in (\ref{mktTriple}), we define
\begin{itemize}
    \item $\invs{U}^G_{k}(p^t) = \pi_{\mathrm{HT},G, t}^{-1}(\mathtt{U}^G_k)$
    \item $\invs{I}^G_{m, k}(p^t) = \pi_{\mathrm{HT},G, t}^{-1}(\mathtt{I}^G_{m, k})$
    \item $\invs{Z}^G_{0}(p^t) = \pi_{\mathrm{HT},G, t}^{-1}(\mathtt{Z}^G_0)$
    \item $\invs{Z}^G_m(p^t) = \pi_{\mathrm{HT}, G, t}^{-1}(]C^G_{w_n}[_{\overline{m}, \overline{0}}\cdot K^G_{\mathrm{Iw}}(p^t))$ for $m \geq 1$
\end{itemize}
where $\pi_{\mathrm{HT}, G, t} \colon \invs{S}_{G, \mathrm{Iw}}(p^t) \to \mathtt{FL}^G/K^G_{\mathrm{Iw}}(p^t)$ is the map (of topological spaces) induced from the Hodge--Tate period map. We will write $\invs{U}^G_k$, $\invs{I}^G_{m, k}$ and $\invs{Z}^G_{m}$ when $t$ is clear from the context.
\end{definition}

Note that, by the Iwahori decompositions in \S \ref{TubesInFlagVar}, $\invs{U}^G_{k}(p^t)$ is an open subset of $\invs{S}_{G, \mathrm{Iw}}(p^t)$ which is a finite union of quasi-Stein open subsets, and $\invs{Z}^G_m(p^t)$ is a closed subset of $\invs{S}_{G, \mathrm{Iw}}(p^t)$ whose complement is a finite union of quasi-Stein open subsets. Note that we have 
\[
\invs{I}^G_{m, k}(p^t) = \invs{U}^G_{k}(p^t) \cap \invs{Z}^G_m(p^t)
\]
so by \cite[Lemma 2.5.21]{BoxerPilloni}, the cohomology complex $R\Gamma_{\invs{I}^G_{m, k}}(\invs{U}^G_{k}, [V_{\kappa_n}])$ is represented by a complex in $\opn{Pro}_{\mbb{N}}(\invs{K}^{\mathrm{proj}}(\mbf{Ban}(\mbb{Q}_p)))$. Furthermore, $R\Gamma_{\invs{I}^G_{0, 0}}(\invs{U}^G_0, [V_{\kappa_n}])$ carries an action of $\invs{H}^{-}_{p, t}$ for which $\invs{U}_B'(p^t)$ acts as a potent compact operator (see Theorem 5.4.3 in \emph{op.cit.}).

\begin{proposition} \label{ChangeOfSupportProp}
For $m, k, t$ in (\ref{mktstrong}), the complex $R\Gamma_{\invs{I}^G_{m, k}}(\invs{U}^G_k, [V_{\kappa_n}])$ carries an action of $\invs{U}_B'(p^t)^{m}$ as a potent compact operator, and the natural maps
\[
R\Gamma_{\invs{I}_{m, k}^G(p^t)}\left(\invs{U}^G_k(p^t), [V_{\kappa_n}] \right) \xleftarrow{\mathrm{res}} R\Gamma_{\invs{I}_{m, 0}^G(p^t)}\left(\invs{U}^G_0(p^t), [V_{\kappa_n}] \right) \xrightarrow{\mathrm{cores}} R\Gamma_{\invs{I}_{0, 0}^G(p^t)}\left(\invs{U}^G_0(p^t), [V_{\kappa_n}] \right)
\]
are equivariant for $\invs{U}_B'(p^t)^m$ and become quasi-isomorphisms after passing to finite-slope parts.
\end{proposition}
\begin{proof}
For this proof only, let $K = K^pK_{\mathrm{Iw}}^G(p^t)$ and $K_x = K \cap xKx^{-1}$, where $x$ is the element in Definition \ref{BorelHeckeOperator}. Let $T$ denote the correspondence 
\[
\invs{S}_{G, K} \xleftarrow{p_2} \invs{S}_{G, K_x} \xrightarrow{p_1} \invs{S}_{G, K}
\]
where $p_1$ is the forgetful map associated with the inclusion $K_x \subset K$, and $p_2$ is the composition of right-translation by $x$ and the forgetful map associated with the inclusion $x^{-1}K_x x \subset K$. For a subset $\invs{W} \subset \invs{S}_{G, K}$, we let $T(\invs{W}) = p_2p_1^{-1}(\invs{W})$ and $(T^t)(\invs{W}) = p_1p_2^{-1}(\invs{W})$. For a non-negative integer $s$, we let
\[
T^{s+1}(\invs{W}) = T(T^s(\invs{W})), \quad \quad \quad (T^t)^{s+1}(\invs{W}) = (T^t)((T^t)^s(\invs{W}))
\]
with the convention that $T^0(\invs{W}) = (T^t)^0(\invs{W}) = \invs{W}$.

By \cite[Lemma 3.3.17 and Lemma 3.5.10]{BoxerPilloni}, one has the following inclusions
\begin{eqnarray}
    (T^t)^{k+1 + m}(\invs{Z}^G_0) \cap \invs{U}^G_0  \subset & \invs{Z}^G_m & \subset  (T^t)^{k+1}(\invs{Z}^G_0) \nonumber \\
    T^m(\invs{U}^G_0) \cap (T^t)^{k+1}(\invs{Z}^G_0)  \subset & \invs{U}^G_{k} & \subset \invs{U}^G_0 \nonumber
\end{eqnarray}
so the result follows from Corollary 5.3.8 in \emph{op.cit.} (note that the action of $x$ factors through its projection to the $\tau_0$-component on the flag variety, so we can apply the cited lemmas with $\opn{min}(x) = 1$ and $\opn{max}(x) = 2n-1$).
\end{proof}

\begin{remark}
It does not seem possible to apply \cite[Corollary 5.3.8]{BoxerPilloni} for general $m, k, t$ satisfying (\ref{mktTriple}), and we do not know if there is an alternative way to show that $R\Gamma_{\invs{I}_{m, k}^G}(\invs{U}^G_k, [V_{\kappa_n}])$ carries an action of a power of $\invs{U}_B'$ as a potent compact operator such that the conclusion of Proposition \ref{ChangeOfSupportProp} holds.
\end{remark}

We also define strata for $\invs{S}_{H, \diamondsuit}(p^t)$.

\begin{definition}
Let $\pi_{\mathrm{HT}, H, t} \colon \invs{S}_{H, \diamondsuit}(p^t) \to \mathtt{FL}^H/K^H_{\diamondsuit}(p^t)$ denote the map induced from the Hodge--Tate period map.
\begin{itemize}
    \item For $m \geq 0$ and $t \geq 1$, we define 
    \[
    \invs{Z}^H_m(p^t) = \pi_{\mathrm{HT}, H, t}^{-1}(\mathtt{Z}^H_m) .
    \]
    \item For $k \geq 0$ and $t \geq 1$, we define
    \[
    \invs{U}^H_k(p^t) = \pi_{\mathrm{HT}, H, t}^{-1}(\mathtt{U}^H_k) .
    \]
\end{itemize}
We will write $\invs{Z}^H_m$ and $\invs{U}^H_k$ when $t$ is clear from the context.
\end{definition}

We now define the relevant cohomology complexes with partial compact support conditions, following \cite[\S 5.4]{BoxerPilloni}. 

\begin{definition}
Let $\lambda \in X^*(T/T_0)^+$. Then we define
\[
R\Gamma_{w_n}^G(\kappa_n)^{-, \mathrm{fs}} \defeq R\Gamma_{\invs{I}_{0, 0}^G(p)}\left(\invs{U}^G_0(p), [V_{\kappa_n}] \right)^{-,\mathrm{fs}}
\]
where $(-)^{-,\mathrm{fs}}$ denotes the finite-slope part with respect to the action of $\invs{H}^{-}_{p, 1}$ as in \S \ref{HeckeOperators}. We denote the cohomology of this complex by $\opn{H}^i_{w_n}(\kappa_n)^{-, \mathrm{fs}}$.
\end{definition}

We record some important properties.

\begin{theorem} \label{ImportantPropertiesThm}
Let $\lambda \in X^*(T/T_0)^+$. 
\begin{enumerate}
    \item (Change of level) Let $m, k, t$ be as in (\ref{mktstrong}) (resp. $m=0, k=0$ and $t \geq 1$). The trace map
    \[
    R\Gamma_{\invs{I}_{m, k}^G(p^{t+1})}\left(\invs{U}^G_k(p^{t+1}), [V_{\kappa_n}] \right) \to R\Gamma_{\invs{I}_{m, k}^G(p^t)}\left(\invs{U}^G_k(p^t), [V_{\kappa_n}] \right)
    \]
    is $(\invs{U}_B')^m$-equivariant (resp. $T^-$-equivariant) and induces a quasi-isomorphism on finite-slope parts. 
    \item (Classicality for small slope) The natural maps 
    \[
    R\Gamma_{\invs{I}_{0, 0}^G(p)}\left(\invs{U}^G_0(p), [V_{\kappa_n}]\right) \xrightarrow{\mathrm{cores}} R\Gamma\left(\invs{U}^G_0(p), [V_{\kappa_n}]\right) \xleftarrow{\mathrm{res}} R\Gamma\left(\invs{S}_{G, \mathrm{Iw}}(p), [V_{\kappa_n}]\right)
    \]
    are $\invs{H}^{-}_{p, 1}$-equivariant and induce quasi-isomorphisms on small slope parts.
    \item (Vanishing for small slope) The complex $R\Gamma\left(\invs{S}_{G, \mathrm{Iw}}(p), [V_{\kappa_n}]\right)^{-, ss}$ is concentrated in degree $n-1$.
\end{enumerate}
\end{theorem}
\begin{proof}
Part (1) is an application of \cite[Corollary 4.2.16 and Theorem 5.4.14]{BoxerPilloni}. Because the Shimura variety is compact, Theorem 6.10.1 implies Conjecture 5.9.2 in \emph{op.cit.} (i.e. the expected slope bounds hold). Parts (2) and (3) then follow immediately from the small slope versions of Theorem 5.12.3 and Theorem 5.12.5 in \emph{op.cit.}.
\end{proof}

We define similar complexes for $\invs{S}_{H, \diamondsuit}(p^t)$, however we do not consider the finite-slope part of these complexes. 

\begin{definition}
We set
\[
R\Gamma^H_{\mathrm{id}}\left(\invs{S}_{H, \diamondsuit}(p^t), \sigma^{[j]}_n \right)^{(-, \dagger)} \defeq \varprojlim_m R\Gamma_{\invs{Z}^H_m(p^t)}\left(\invs{S}_{H, \diamondsuit}(p^t), [\sigma_n^{[j]}]\right) 
\]
where the transition maps are given by corestriction. If $t = 1$, we simply write $R\Gamma^H_{\mathrm{id}}\left(\sigma^{[j]}_n \right)^{(-, \dagger)}$ and denote the cohomology of this complex by $\opn{H}^i_{\mathrm{id}}(\sigma_n^{[j]})^{(-, \dagger)}$.
\end{definition}

\subsection{Functoriality} \label{PullbackAdicFunctoriality}

The goal of this section is to construct a map
\[
R\Gamma^G_{w_n}(\kappa_n)^{-, \mathrm{fs}} \to R\Gamma^H_{\mathrm{id}}\left(\sigma^{[j]}_n \right)^{(-, \dagger)}
\]
which is compatible with pull-back by $\widehat{\iota}$ on the usual cohomology.

\begin{definition} \label{VarThetaMap}
Let $m, k, t$ be as in (\ref{mktTriple}). Then we define a morphism
\[
\vartheta_{m, k, t} \colon R\Gamma_{\invs{I}_{m, k}^G(p^t)}\left( \invs{U}_k^G(p^t), [V_{\kappa_n}] \right) \to R\Gamma_{\invs{Z}_{m}^H(p^t)}\left( \invs{S}_{H, \diamondsuit}(p^t), [\sigma_n^{[j]}] \right) 
\]
as the composition of the following maps:
\begin{itemize}
    \item $\widehat{\iota}^* \colon R\Gamma_{\invs{I}_{m, k}^G(p^t)}\left( \invs{U}_k^G(p^t), [V_{\kappa_n}] \right) \to R\Gamma_{\invs{Z}_{m}^H(p^t)}\left( \invs{U}_k^H(p^t), \widehat{\iota}^*[V_{\kappa_n}] \right)$
    \item (Excision) $R\Gamma_{\invs{Z}_{m}^H(p^t)}\left( \invs{U}_k^H(p^t), \widehat{\iota}^*[V_{\kappa_n}] \right) \xrightarrow{\sim} R\Gamma_{\invs{Z}_{m}^H(p^t)}\left( \invs{S}_{H, \diamondsuit}(p^t), \widehat{\iota}^*[V_{\kappa_n}] \right)$
    \item $R\Gamma_{\invs{Z}_{m}^H(p^t)}\left( \invs{S}_{H, \diamondsuit}(p^t), \widehat{\iota}^*[V_{\kappa_n}] \right) \to R\Gamma_{\invs{Z}_{m}^H(p^t)}\left( \invs{S}_{H, \diamondsuit}(p^t), [\sigma_n^{[j]}] \right)$
\end{itemize}
where the last map is induced from $V_{\kappa_n} \to \sigma_n^{[j]}$ (as in (\ref{MapOfRepsTwistedEqn})). Note that $\widehat{\iota}^*$ is well-defined by the Cartesian square in Lemma \ref{CartesianSquareStrataLemma} and the fact that the strata on the level of flag varieties are independent of $t$ for $t > m$ (c.f. property (2) in \cite[\S 2.1]{BoxerPilloni}). The excision step is well-defined because $\invs{Z}_m^H(p^t)$ is closed $\invs{S}_{H, \diamondsuit}(p^t)$ (c.f. property (3) in \emph{loc.cit.}).
\end{definition}

Let $m, k, t$ and $m', k', t'$ be triples satisfying (\ref{mktTriple}), such that $m' \geq m$, $k' \geq k$ and $t' \geq t$. Then the maps in Definition \ref{VarThetaMap} fit into the following commutative diagram: 
\[
\begin{tikzcd}
{R\Gamma_{\invs{I}_{m', k'}^G(p^{t'})}\left( \invs{U}_{k'}^G(p^{t'}), [V_{\kappa_n}] \right)} \arrow[r, "{\vartheta_{m', k', t'}}"]                                                        & {R\Gamma_{\invs{Z}_{m'}^H(p^{t'})}\left( \invs{S}_{H, \diamondsuit}(p^{t'}), [\sigma_n^{[j]}] \right)} \arrow[d, equal] \\
{R\Gamma_{\invs{I}_{m', k}^G(p^{t'})}\left( \invs{U}_k^G(p^{t'}), [V_{\kappa_n}] \right)} \arrow[r, "{\vartheta_{m', k, t'}}"] \arrow[u, "\mathrm{res}"] \arrow[d, "\mathrm{cores}"'] & {R\Gamma_{\invs{Z}_{m'}^H(p^{t'})}\left( \invs{S}_{H, \diamondsuit}(p^{t'}), [\sigma_n^{[j]}] \right)} \arrow[d, "\mathrm{cores}"]    \\
{R\Gamma_{\invs{I}_{m, k}^G(p^{t'})}\left( \invs{U}_k^G(p^{t'}), [V_{\kappa_n}] \right)} \arrow[r, "{\vartheta_{m, k, t'}}"] \arrow[d, "\mathrm{tr}"']                              & {R\Gamma_{\invs{Z}_{m}^H(p^{t'})}\left( \invs{S}_{H, \diamondsuit}(p^{t'}), [\sigma_n^{[j]}] \right)} \arrow[d, "\mathrm{tr}"]       \\
{R\Gamma_{\invs{I}_{m, k}^G(p^t)}\left( \invs{U}_k^G(p^t), [V_{\kappa_n}] \right)} \arrow[r, "{\vartheta_{m, k, t}}"]                                                            & {R\Gamma_{\invs{Z}_{m}^H(p^t)}\left( \invs{S}_{H, \diamondsuit}(p^t), [\sigma_n^{[j]}] \right)}                                   
\end{tikzcd}
\]
where $\opn{tr}$ denotes the trace map (see \cite[Lemma 2.1.2]{BoxerPilloni}). The bottom square is commutative because, by Lemma \ref{IndexLevelSubgroup}, we have a Cartesian diagram of Shimura varieties:
\[
\begin{tikzcd}
{\invs{S}_{H, \diamondsuit}(p^{t+1})} \arrow[d] \arrow[r, "\widehat{\iota}"] & {\invs{S}_{G, \mathrm{Iw}}(p^{t+1})} \arrow[d] \\
{\invs{S}_{H, \diamondsuit}(p^t)} \arrow[r, "\widehat{\iota}"]               & {\invs{S}_{G, \mathrm{Iw}}(p^t)}              
\end{tikzcd}
\]
for any $t \geq 1$.

\begin{proposition} \label{MinusPullbackProp}
One has a well-defined map
\[
R\Gamma^G_{w_n}(\kappa_n)^{-, \mathrm{fs}} \to R\Gamma^H_{\mathrm{id}}\left(\sigma^{[j]}_n \right)^{(-, \dagger)}
\]
defined as the (inverse limit over $m$ of the) composition of
\begin{itemize}
    \item The inverse of the trace map followed by the inverse of corestriction 
    \[
    R\Gamma_{\invs{I}_{0, 0}^G(p)}(\invs{U}^G_0(p), [V_{\kappa_n}])^{-, \mathrm{fs}} \xrightarrow{\sim} R\Gamma_{\invs{I}^G_{m, 0}(p^t)}(\invs{U}^G_0(p^t), [V_{\kappa_n}])^{-, \mathrm{fs}}
    \]
    which makes sense by Proposition \ref{ChangeOfSupportProp} and Theorem \ref{ImportantPropertiesThm}.
    \item The morphism $\vartheta_{m, 0, t}$
    \item The trace map 
    \[
    R\Gamma_{\invs{Z}^H_m(p^t)}(\invs{S}_{H, \diamondsuit}(p^t), [\sigma_n^{[j]}]) \to R\Gamma_{\invs{Z}^H_m(p)}(\invs{S}_{H, \diamondsuit}(p), [\sigma_n^{[j]}])
    \]
\end{itemize}
for any $m \geq 0$, $t \geq 1$ satisfying $m > 2n-1$ and $t > m+1$ (i.e. the tuple $(m, 0, t)$ satisfies (\ref{mktstrong})).
\end{proposition}
\begin{proof}
This map is well-defined by the above commutative diagram and the fact that trace and corestriction commute with each other (c.f. the construction in \cite[Lemma 2.1.2]{BoxerPilloni}).
\end{proof}

Set $R\Gamma_{\mathrm{id}}^H(\sigma_n^{[j], \vee})^{(+, \dagger)} \defeq \varinjlim_m R\Gamma(\invs{Z}_m^H(p), [\sigma_n^{[j]}]^\vee )$ with transition maps given by restriction, and denote the cohomology of this complex by $\opn{H}^i_{\mathrm{id}}(\sigma_n^{[j], \vee})^{(+, \dagger)}$. By \cite[Theorem 2.7.1]{BoxerPilloni} (using the fact that $\invs{Z}^H_m$ is the closure of $\invs{U}^H_m$), one has a natural pairing between $R\Gamma_{\mathrm{id}}^H(\sigma_n^{[j]})^{(-, \dagger)}$ and $R\Gamma_{\mathrm{id}}^H(\sigma_n^{[j], \vee})^{(+, \dagger)}$ built from the Serre duality pairings, which commutes with the Serre duality pairing between $R\Gamma(\invs{S}_{H, \diamondsuit}(p), [\sigma_n^{[j]}])$ and $R\Gamma(\invs{S}_{H, \diamondsuit}(p), [\sigma_n^{[j]}]^{\vee})$ via corestriction and restriction on the former and latter complex respectively.\footnote{To be more precise, one cannot directly apply \cite[Theorem 2.7.1]{BoxerPilloni} because $\invs{U}^H_m$ is not quasi-compact. However one can find quasi-compact open subsets $\invs{U}'_m$ satisfying $\invs{U}^H_{m+1} \subset \invs{U}'_m \subset \invs{U}^H_m$ and apply the theorem with these strata instead, as this does not affect the cohomology groups in the limit.} Proposition \ref{MinusPullbackProp} therefore allows us to define a pairing 
\[
\boxed{ 
\langle , \rangle_{\mathrm{an}}^{-} \colon \opn{H}^{n-1}_{w_n}( \kappa_n )^{-, \mathrm{fs}} \times \opn{H}^0_{\mathrm{id}}(\sigma^{[j], \vee}_n)^{(+, \dagger)} \to \mbb{Q}_p
}
\]
by composing the map in Proposition \ref{MinusPullbackProp} with the duality pairing between the $(-, \dagger)$ and $(+, \dagger)$ cohomologies above. Considering classes in the small slope part, we obtain the following result:

\begin{theorem} \label{RestrictionToStrataThm}
Let 
\begin{itemize} 
\item $\chi \in \opn{H}^0\left(\invs{S}_{H, \diamondsuit}(p), [\sigma_n^{[j]}]^\vee \right)$ 
\item $\eta \in \opn{H}^{n-1}_{w_n}( \kappa_n )^{-, \mathrm{ss}} \cong \opn{H}^{n-1}\left(\invs{S}_{G, \mathrm{Iw}}(p), [V_{\kappa_n}] \right)^{-, ss}$
\end{itemize}
and denote by $\opn{res}\chi$ the image of $\chi$ under the restiction map
\[
\opn{H}^0\left(\invs{S}_{H, \diamondsuit}(p), [\sigma_n^{[j]}]^\vee \right) \to \opn{H}^0_{\mathrm{id}}(\sigma^{[j], \vee}_n)^{(+, \dagger)} .
\]
Then $\langle \eta, \opn{res}\chi \rangle^-_{\mathrm{an}} = \langle \eta, \chi \rangle_{\mathrm{an}}$.
\end{theorem}
\begin{proof}
Since the embedding $\widehat{\iota} \colon \invs{S}_{H, \diamondsuit}(p^t) \to \invs{S}_{G, \mathrm{Iw}}(p^t)$ factors through $\invs{U}^G_0(p^t)$, we obtain the commutative diagram
\[
\begin{tikzcd}
{R\Gamma^G_{w_n}(\kappa_n)^{-, \mathrm{fs}}} \arrow[r] \arrow[d, "\mathrm{cores}"']         & {R\Gamma^H_{\mathrm{id}}(\sigma_n^{[j]})^{(-, \dagger)}} \arrow[d, "\mathrm{cores}"] \\
{R\Gamma(\invs{U}^G_0(p), [V_{\kappa_n}])} \arrow[r]                                        & {R\Gamma(\invs{S}_{H, \diamondsuit}(p), [\sigma_n^{[j]}])} \arrow[d, equal]                 \\
{R\Gamma(\invs{S}_{G, \mathrm{Iw}}(p), [V_{\kappa_n}])} \arrow[r] \arrow[u, "\mathrm{res}"] & {R\Gamma(\invs{S}_{H, \diamondsuit}(p), [\sigma_n^{[j]}])}                          
\end{tikzcd}
\]
where the top horizontal arrow is as in Proposition \ref{MinusPullbackProp}, and the bottom two are obtained from composing $\widehat{\iota}^*$ with the map of sheaves $\widehat{\iota}^*[V_{\kappa_n}] \to [\sigma_n^{[j]}]$. Passing to small slope parts and cohomology gives the result.
\end{proof}

%%%%%%%%%%%%%%%%%%%%%%%%%%%%%%%%%%%%%%%%%%%%%%%%%%%%%%%%%%%%%%%%%%%%%%%%%%%%%%%%%%%%%%%%%%%%%%%%
%%%%%%%%%%%%%%       LOCALLY ANALYTIC COHOMOLOGY   %%%%%%%%%%%%%%%%%%%%%%%%%%%%%%%%%%%%%%%%%%%%%
%%%%%%%%%%%%%%%%%%%%%%%%%%%%%%%%%%%%%%%%%%%%%%%%%%%%%%%%%%%%%%%%%%%%%%%%%%%%%%%%%%%%%%%%%%%%%%%%

\section{Locally analytic cohomology} \label{LocallyAnalyticCohomologySection}

\subsection{Further reduction of structure} \label{FurtherReductionOfStructureSection}

We first consider the reduction of structure for $\invs{M}_{G,\mathrm{HT}}$. Let $U_G$, $\overline{U}_G$, $U_{M_G}$ and $\overline{U}_{M_G}$ denote the unipotent radicals of $B_G$, $\overline{B}_G$, $B_{M_G}$ and $\overline{B}_{M_G}$ respectively. For $k > 0$, let $\invs{G}^1_{k, k}$ (resp. $\invs{M}_{G, k, k}^1$) denote the subgroup of $\invs{G}$ (resp. $\invs{M}_G$) of elements which reduce to $\invs{U}_G$ (resp. $\invs{U}_{M_G}$) modulo $p^{k+\varepsilon}$ for all $\varepsilon > 0$, and to $\overline{\invs{U}}_G$ (resp. $\overline{\invs{U}}_{M_G}$) modulo $p^k$. We have similar definitions for $\invs{M}^1_{H, k, k}$ and $\invs{H}^1_{k, k}$.

We introduce the following group:

\begin{definition}
Let $\invs{M}^{\square}_{G, k, k} = \invs{M}^{1}_{G, k, k} \cdot B_{M_G}(\mbb{Z}_p)$, which is a subgroup of $\invs{M}_G$ containing the Iwahori subgroup of $M_G(\mbb{Z}_p)$ of depth $p^t$ for any $t>k$.
\end{definition}

\begin{remark}
The homomorphism $\mu \colon \mbb{Z}_p^{\times} \to \invs{M}_G$ factors through the subgroup $\invs{M}^{\square}_{G, k, k}$.
\end{remark}

\begin{remark}
Let $t>k>0$. If we let $K_{p, w_n, M_G}$ equal the projection of $w_n K^G_{\mathrm{Iw}}(p^t) w_n^{-1} \cap \invs{P}_G$ to $\invs{M}_G$, then the proof of \cite[Proposition 4.6.9]{BoxerPilloni} shows that $K_{p, w_n, M_G}$ equals the Iwahori subgroup of $M_G(\mbb{Z}_p)$ of depth $p^t$ (the proposition only treats the case $t=1$, but the proof easily generalises to arbitrary $t$). Therefore $\invs{M}^{\square}_{G, k, k} = \invs{M}^1_{G, k, k} \cdot K_{p, w_n, M_G} = K_{p, w_n, M_G} \cdot \invs{M}^1_{G, k, k}$.
\end{remark}

For $t > k > 0$, let $\mathtt{M}^G_{k,k, t}$ denote the space
\begin{align*}
    K^G_{\mathrm{Iw}}(p^t) \invs{G}^1_{k,k} / \left( K^G_{\mathrm{Iw}}(p^t) \invs{G}^1_{k,k} \cap w_n^{-1}\invs{N}_G w_n \right) &\to \invs{P}_G \backslash \invs{P}_G w_n K^G_{\mathrm{Iw}}(p^t) \invs{G}^1_{k, k} = \mathtt{U}^G_k \\
    x &\mapsto \invs{P}_G w_n x^{-1}
\end{align*}
which is a (right) torsor for the group $\invs{M}^{\square}_{G, k, k}$ via the embedding $w_n^{-1}\invs{M}^{\square}_{G, k, k}w_n \subset K^G_{\mathrm{Iw}}(p^t) \invs{G}^1_{k, k}$.

\begin{proposition} \label{MGTorsorReductionStructure}
Let $t>k>0$. The torsor $\invs{M}_{G,\mathrm{HT}}$ has a reduction of structure to a pro\'{e}tale $\invs{M}^{\square}_{G, k, k}$-torsor $\invs{M}_{G,\mathrm{HT}, k,k, t}$ over $\invs{U}^G_k(p^t)$. Furthermore, the pullback of $\invs{M}_{G,\mathrm{HT}, k,k, t}$ to the perfectoid space $\invs{S}_{G, K^p}$ is canonically isomorphic to the torsor $\pi_{\mathrm{HT}}^* \mathtt{M}^G_{k,k, t}$.

Moreover, the twisted torsor ${^\mu \invs{M}_{G,\mathrm{HT}, k,k, t}}$ defines a reduction of structure of the torsor ${^\mu \invs{M}_{G, \opn{HT}}}$ to an \'{e}tale $\invs{M}_{G, k, k}^{\square}$-torsor.
\end{proposition}
\begin{proof}
This is essentially \cite[Proposition 4.6.12]{BoxerPilloni}, but we have conjugated our groups by $w_n$. Note that we have a commutative diagram:

\[
\begin{tikzcd}
{K^G_{\mathrm{Iw}}(p^t) \invs{G}^1_{k,k} / \left( K^G_{\mathrm{Iw}}(p^t) \invs{G}^1_{k,k} \cap w_n^{-1}\invs{N}_G w_n \right)} \arrow[d] \arrow[r] & \invs{G}/\invs{N}_G \arrow[d] \\
\mathtt{U}^G_k \arrow[r]                                                                                                                           & \mathtt{FL}^G                
\end{tikzcd}
\]
where the vertical maps are the torsors $\mathtt{M}^G_{k,k, t}$ and $\mathtt{M}^G$, the bottom map is the natural inclusion and the top map is given by 
\[
g \left( K^G_{\mathrm{Iw}}(p^t) \invs{G}^1_k \cap w_n^{-1}\invs{N}_G w_n \right) \mapsto gw_n^{-1} \invs{N}_G.
\]
Therefore $\mathtt{M}^G_{k,k, t}$ gives a reduction of structure for $\mathtt{M}^G$, and the pro\'{e}tale torsor $\pi^*_{\opn{HT}}\mathtt{M}^G_{k, k, t}$ descends to a pro\'{e}tale torsor $\invs{M}_{G, \opn{HT}, k, k, t}$ over $\invs{U}^G_k(p^t)$ because it is a $K^G_{\opn{Iw}}(p^t)$-invariant open subset of the pro\'{e}tale torsor $\pi^*_{\opn{HT}} \mathtt{M}^G$ (which we already know descends).  The last part follows from the fact that $\mu$ factors through $\invs{M}^{\square}_{G, k, k}$, Lemma \ref{PropertiesOfTTorsorLemma}(3), and because ${^\mu \invs{M}_{G,\mathrm{HT}, k,k, t}} \to \invs{U}^G_k(p^t)$ is surjective on geometric points and smooth (as ${^\mu \invs{M}_{G,\mathrm{HT}, k,k, t}}$ is an open subset of the \'{e}tale torsor ${^\mu \invs{M}_{G, \opn{HT}}}$).
\end{proof}

We now discuss the reduction of structure for $\invs{M}_{H,\mathrm{HT}}$. Consider the following subtori of $T$ consisting of elements $(x; y_{1, \tau}, \dots, y_{2n, \tau})$ satisfying the following relations:
\begin{itemize}
    \item $T^{\clubsuit} \subset T$ is the sub-torus given by the relations $y_{i, \tau_0} = y_{2n+2-i, \tau_0}$ for $i=2, \dots, 2n$, and $y_{i, \tau} = y_{2n+1-i, \tau}$ for all $i=1, \dots, 2n$ and $\tau \neq \tau_0$
    \item $T^{\diamondsuit} \subset T^{\clubsuit}$ is the sub-torus with the additional relation that $y_{1, \tau_0} = y_{n+1, \tau_0}$.
\end{itemize}

We begin with the following lemma:

\begin{lemma}
Let $\opn{Iw}_{M_G}(p^t) \subset M_G(\mbb{Z}_p)$ denote the Iwahori subgroup of depth $p^t$, and let $M^H_{\diamondsuit}(p^t)$ denote the projection of $K^H_{\diamondsuit}(p^t) \cap \invs{P}_H$ to $\invs{M}_H$. Then
\begin{enumerate}
    \item $M^H_{\diamondsuit}(p^t)$ is the subgroup of $M_H(\mbb{Z}_p)$ of all elements which land in $T^{\diamondsuit}$ modulo $p^t$
    \item $M_{\clubsuit}^H(p^t) \defeq u \opn{Iw}_{M_G}(p^t) u^{-1} \cap \invs{M}_H$ is the subgroup of $M_H(\mbb{Z}_p)$ of all elements which land in $T^{\clubsuit}$ modulo $p^t$. It is contained in the projection of $K^H_{\opn{Iw}}(p^t) \cap \invs{P}_H$ to $\invs{M}_H$.
\end{enumerate}
In particular, one has $M^H_{\diamondsuit}(p^t) \subset M_{\clubsuit}^H(p^t)$.
\end{lemma}
\begin{proof}
By the proof of Lemma \ref{OpenOrbitLemma}, we see that 
\[
h = x \times \tbyt{y_{1, \tau}}{}{}{y_{2, \tau}} \in H(\mbb{Z}_p)
\]
lies in $K^H_{\diamondsuit}(p^t)$ if and only if:
\begin{itemize}
    \item For all $\tau \neq \tau_0$, the block diagonal matrix $(y_{1, \tau}, y_{2, \tau})$ lies in the $\tau$-component of $T^{\diamondsuit}$ modulo $p^t$
    \item The elements $U^{-1}y_{1, \tau_0}U$ and $y_{2, \tau_0}$ are lower-triangular and upper-triangular modulo $p^t$ respectively, where $U$ is a $(n \times n)$ matrix lying the standard parabolic of $\opn{GL}_n$ with Levi $\opn{GL}_1 \times \opn{GL}_{n-1}$, whose projection to the Levi equals
    \[
    1 \times w_{\opn{GL}_{n-1}}^{\mathrm{max}} .
    \]
    \item The elements $U^{-1}y_{1, \tau_0} U$ and $y_2$ are congruent to each other modulo $p^t$.
\end{itemize}
From these properties, one then immediately obtains part (1). Part (2) follows from the stabiliser computations in Lemma \ref{OpenOrbitLemma}. It is contained in the projection of $K^H_{\opn{Iw}}(p^t) \cap \invs{P}_H$ to $\invs{M}_H$ because $T^{\clubsuit}$ is contained in $B_H$.
\end{proof}

For $t \geq 1$ and $k > 0$, we let $\invs{M}^{\diamondsuit}_{H, k,k, t} = M^H_{\diamondsuit}(p^t) \invs{M}^{1}_{H, k, k}$ and $\invs{M}^{\clubsuit}_{H, k,k, t} = M^H_{\clubsuit}(p^t) \invs{M}^{1}_{H, k, k}$. Both of these are groups by \cite[Lemma 3.3.15]{BoxerPilloni} (because $K^H_{\diamondsuit}(p^t) \subset K^H_{\mathrm{Iw}}(p^t)$). If $t > k$, then these groups don't depend on $t$; explicitly, we have
\[
\invs{M}^{\diamondsuit}_{H, k,k, t} = \invs{M}^{\diamondsuit}_{H, k, k} \defeq T^{\diamondsuit}(\mbb{Z}_p) \invs{M}^1_{H, k, k}, \quad \quad \invs{M}^{\clubsuit}_{H, k,k, t} = \invs{M}^{\clubsuit}_{H, k, k} \defeq T^{\clubsuit}(\mbb{Z}_p) \invs{M}^1_{H, k, k}. 
\]
Furthermore, we have $u^{-1}\invs{M}^{\diamondsuit}_{H, k, k} u \subset u^{-1}\invs{M}^{\clubsuit}_{H, k, k} u \subset \invs{M}^{\square}_{G, k, k}$.

\begin{remark}
The homomorphism $\mu \colon \mbb{Z}_p^{\times} \to \invs{M}_H$ induced from the Hodge cocharacter $\mu_{\mbf{H}}$ factors through $\invs{M}^{\clubsuit}_{H, k, k, t}$ for any $t \geq 1$ and $k > 0$. It doesn't factor through $\invs{M}^{\diamondsuit}_{H, k, k, t}$, although the latter group is useful for discussing the reduction of structure below.
\end{remark}

As above, we introduce the following space $\mathtt{M}^H_{k,k, t}$ (for $k>0$ and $t \geq 1$):
\begin{align*}
    K^H_{\diamondsuit}(p^t) \invs{H}^1_{k, k} /\left(  K^H_{\diamondsuit}(p^t) \invs{H}^1_{k, k} \cap \invs{N}_H \right) &\to \invs{P}_H \backslash \invs{P}_H  K^H_{\diamondsuit}(p^t) \invs{H}^1_{k, k} = \mathtt{U}^H_k \\
    x &\mapsto \invs{P}_H x^{-1}
\end{align*}
which is a (right) torsor for the group $\invs{M}^{\diamondsuit}_{H, k,k, t}$ via the embedding $\invs{M}^{\diamondsuit}_{H, k,k, t} \subset K^H_{\diamondsuit}(p^t) \invs{H}^1_{k,k}$.

\begin{proposition}
Let $t \geq 1$ and $k > 0$. 
\begin{enumerate}
\item Then the torsor $\invs{M}_{H,\mathrm{HT}}$ has a reduction of structure to a pro\'{e}tale $\invs{M}^{\diamondsuit}_{H, k,k, t}$-torsor $\invs{M}'_{H,\mathrm{HT}, k,k, t}$ over $\invs{U}^H_k(p^t)$. Furthermore, the pullback of $\invs{M}'_{H,\mathrm{HT}, k,k, t}$ to the perfectoid space $\invs{S}_{H, U^p}$ is canonically isomorphic to $\pi_{\mathrm{HT}}^*\mathtt{M}^H_{k,k, t}$.
\item If we define $\invs{M}_{H,\mathrm{HT}, k,k, t}$ as the pushout $\invs{M}'_{H,\mathrm{HT}, k,k, t} \times^{\invs{M}^{\diamondsuit}_{H, k,k, t}} \invs{M}^{\clubsuit}_{H, k,k, t}$, then the pro\'{e}tale $\invs{M}^{\clubsuit}_{H, k,k, t}$-torsor $\invs{M}_{H,\mathrm{HT}, k,k, t}$ (resp. \'{e}tale $\invs{M}^{\clubsuit}_{H, k,k, t}$-torsor ${^\mu \invs{M}_{H,\mathrm{HT}, k,k, t}}$) provides a reduction of structure of the torsor $\invs{M}_{H,\mathrm{HT}}$ (resp. ${^\mu \invs{M}_{H,\mathrm{HT}}}$).
\end{enumerate}
\end{proposition}
\begin{proof}
For the first part, this follows from a similar argument in Proposition \ref{MGTorsorReductionStructure}. Note that the proof of \cite[Proposition 4.6.12]{BoxerPilloni} also applies in this situation, even though $K^H_{\diamondsuit}(p^t)$ is not of the form in the statement of \emph{loc.cit.}.

The second part follows immediately from the inclusions 
\[
\invs{M}^{\diamondsuit}_{H, k,k, t} \subset \invs{M}^{\clubsuit}_{H, k,k, t} \subset \invs{M}_H,
\]
the functoriality properties in Lemma \ref{PropertiesOfTTorsorLemma}, and the fact that ${^\mu \invs{M}_{H,\mathrm{HT}, k,k, t}} \to \invs{U}^H_k(p^t)$ is smooth and surjective on geometric points.
\end{proof}

We have the following proposition which relates the torsors for $G$ and $H$.

\begin{proposition} \label{MGHTorsorComparison}
Let $t > k > 0$. One has a reduction of structure of \'{e}tale torsors
\[
\widehat{\iota}^*\left( {^\mu \invs{M}_{G,\mathrm{HT}, k,k, t}} \right) = {^\mu \invs{M}_{H,\mathrm{HT}, k,k, t}} \times^{[\invs{M}^{\clubsuit}_{H, k,k}, u]} \invs{M}^{\square}_{G, k,k}
\]
where $\widehat{\iota}$ denotes the embedding $\invs{U}^H_k(p^t) \hookrightarrow \invs{U}^G_k(p^t)$.
\end{proposition}
\begin{proof}
We first show that we have a reduction of structure:
\begin{equation} \label{FlagReduction}
\widehat{\iota}^* \invs{M}_{G,\mathrm{HT}, k,k, t} = \invs{M}'_{H,\mathrm{HT}, k,k, t} \times^{[\invs{M}^{\diamondsuit}_{H, k,k}, u]} \invs{M}^{\square}_{G, k,k} .
\end{equation}
It is enough to show the analogous statement for the torsors $\mathtt{M}^G_{k, k, t}$ and $\mathtt{M}^H_{k, k, t}$. In this case, we have a commutative diagram: 
\[
\begin{tikzcd}
{K^H_{\diamondsuit}(p^t) \invs{H}^1_{k, k} /\left(  K^H_{\diamondsuit}(p^t) \invs{H}^1_{k, k} \cap \invs{N}_H \right)} \arrow[d] \arrow[r] & {K^G_{\mathrm{Iw}}(p^t) \invs{G}^1_{k,k} / \left( K^G_{\mathrm{Iw}}(p^t) \invs{G}^1_{k,k} \cap w_n^{-1}\invs{N}_G w_n \right)} \arrow[d] \\
\mathtt{U}^H_k \arrow[r, "\widehat{\iota}"]                                                                                                & \mathtt{U}^G_k                                                                                                                          
\end{tikzcd}
\]
where the vertical maps are the torsors $\mathtt{M}^H_{k, k, t}$ and $\mathtt{M}^G_{k, k, t}$, and the top map is induced from the map $K^H_{\diamondsuit}(p^t)\invs{H}^1_{k, k} \to K^G_{\mathrm{Iw}}(p^t)\invs{G}^1_{k, k}$ given by $h \mapsto \widehat{\gamma}^{-1}h \widehat{\gamma}$. Note that this diagram is commutative because $\gamma \in P_G$.

Since $\mathtt{M}^G_{k, k, t}$ is a torsor for the group $\invs{M}^{\square}_{G, k, k}$ via the conjugated embedding $w_n^{-1}\invs{M}^{\square}_{G, k ,k}w_n \subset K^G_{\mathrm{Iw}}(p^t)\invs{G}^1_{k, k}$, and the projection of $\gamma$ to $M_G$ is equal to $u$, (\ref{FlagReduction}) follows. 

Since $u^{-1}\invs{M}^{\diamondsuit}_{H, k,k}u \subset u^{-1} \invs{M}^{\clubsuit}_{H, k,k} u \subset \invs{M}^{\square}_{G, k, k}$, we also obtain the reduction of structure
\[
\widehat{\iota}^* \invs{M}_{G,\mathrm{HT}, k,k, t} = \invs{M}_{H,\mathrm{HT}, k,k, t} \times^{[\invs{M}^{\clubsuit}_{H, k,k}, u]} \invs{M}^{\square}_{G, k,k}
\]
and we can twist this along $\mu$ by Lemma \ref{PropertiesOfTTorsorLemma} (and the fact $\mu$ is central, so unaffected by conjugation by $u$).
\end{proof}

\begin{remark} \label{CompatibilityInKT}
If $t' \geq t$ and $k' \geq k$, then the torsors ${^\mu \invs{M}_{G, \mathrm{HT}, k', k', t'}}$ and ${^\mu \invs{M}_{H, \mathrm{HT}, k', k', t'}}$ provide a reduction of structure for the pullbacks of ${^\mu \invs{M}_{G, \mathrm{HT}, k, k, t}}$ and ${^\mu \invs{M}_{H, \mathrm{HT}, k, k, t}}$ along the trace/inclusion maps $\invs{U}^G_{k'}(p^{t'}) \to \invs{U}^G_{k}(p^t)$ and $\invs{U}^H_{k'}(p^{t'}) \to \invs{U}^H_{k}(p^t)$ respectively (c.f. \cite[Proposition 4.6.14]{BoxerPilloni}).
\end{remark}

\begin{remark}
Let $k > 0$, and let $\invs{M}^1_{G, k}$ (resp. $\invs{M}^1_{H, k}$) denote the normal affinoid subgroup of $\invs{M}_G$ (resp. $\invs{M}_H$) consisting of elements which reduce to the identity modulo $p^k$. We set 
\[
\invs{M}^{\square}_{G, k} = \invs{M}^1_{G, k} B_{M_G}(\mbb{Z}_p), \quad \quad \invs{M}^{\clubsuit}_{H, k} = \invs{M}^1_{H, k}T^{\clubsuit}(\mbb{Z}_p), \quad \quad \invs{M}^{\clubsuit}_{H, k, t} = \invs{M}^1_{H, k} M^H_{\clubsuit}(p^t)
\]
for $k, t \geq 1$. All of these groups are \emph{open affinoid} analytic subgroups of $\invs{M}_{?}$, where $? = G, H$ according to the subscript.  

To be able to apply the results in \cite[\S 6]{BoxerPilloni}, it will be more convenient to work with the following torsors, obtained as the pushouts:
\begin{align*}
    \invs{M}_{G, \mathrm{HT}, k, t} &\defeq \invs{M}_{G, \mathrm{HT}, k,k, t} \times^{\invs{M}^{\square}_{G, k, k}} \invs{M}^{\square}_{G, k}  \\
    \invs{M}_{H, \mathrm{HT}, k, t} &\defeq \invs{M}_{H,\mathrm{HT}, k,k, t} \times^{\invs{M}^{\clubsuit}_{H, k, k, t}} \invs{M}^{\clubsuit}_{H, k, t} . 
\end{align*}
In particular, we can twist these torsors along $\mu$ and the torsors ${^\mu \invs{M}_{G, \mathrm{HT}, k, t}}$ and ${^\mu \invs{M}_{H, \mathrm{HT}, k, t}}$ are \'{e}tale torsors by Lemma \ref{PropertiesOfTTorsorLemma}. The analogous compatibility for varying $k$ and $t$ as in Remark \ref{CompatibilityInKT} still continues to hold for these torsors, and we have an analogue of Proposition \ref{MGHTorsorComparison}, namely one has a reduction of structure of \'{e}tale torsors
\[
\widehat{\iota}^*\left( {^\mu \invs{M}_{G, \mathrm{HT}, k, t}} \right) = {^\mu \invs{M}_{H, \mathrm{HT}, k, t}} \times^{[\invs{M}^{\clubsuit}_{H, k}, u]} \invs{M}^{\square}_{G, k}
\]
whenever $t > k > 0$.
\end{remark}

\subsection{Weight spaces}

For an integer $r \in \mbb{Q}_{>0}$ let $\invs{T}^1_r$ denote the subgroup of $\invs{T}$ of elements which reduce to the identity modulo $p^r$. Recall that for a Tate algebra $(A, A^+)$ over $(\mbb{Q}_p, \mbb{Z}_p)$, a character 
\[
\lambda \colon T(\mbb{Z}_p) \to (A^+)^{\times}
\]
is $r$-analytic if it extends to an analytic $A$-valued function on $T(\mbb{Z}_p) \invs{T}^1_r \subset T^{\mathrm{ad}}$.

\begin{definition}
Let $(A, A^+)$ be a Tate algebra above. We let $X^*(T; A)$ denote the space of all characters 
\[
\lambda \colon T(\mbb{Z}_p) \to (A^+)^{\times} 
\]
which are $r$-analytic, for some $r \in \mbb{Q}_{>0}$. We let $X^*(T/T_0; A) \subset X^*(T; A)$ be the subspace of all characters which are trivial on $T_0(\mbb{Z}_p)$.
\end{definition}

\begin{remark}
Note that there is a Weyl action on $X^*(T; A)$ by the usual formulae. Furthermore, even though the half sum of positive roots doesn't strictly give an element of this space, the $\star$-action of the Weyl group also still makes sense. 
\end{remark}

\begin{remark}
The functor $(A, A^+) \mapsto X^*(T/T_0; A)$ is representable by a group adic space over $\opn{Spa}(\mbb{Q}_p, \mbb{Z}_p)$, which we will denote by $\invs{W}_G$.
\end{remark}

\begin{definition}
For $i=1, \dots, n$ and $\tau \in \Psi$, let $\lambda_{i, \tau} \in X^*(T/T_0)^+$ be the character which is trivial outside the $\tau$-component, and in the $\tau$-component is given by the tuple
\[
(1, \dots, 1, 0, \dots, 0, -1, \dots, -1)
\]
where there are $i$ lots of $1$s and $-1$s.
\end{definition}

These characters give a generating set for $X^*(T/T_0; A)$ in the following sense.

\begin{lemma} \label{GeneratingSetWeightSpace}
Let $\lambda \in X^*(T/T_0; A)$ be an $r$-analytic character. Then there exist unique $r$-analytic characters $\xi_{i, \tau} \colon \mbb{Z}_p^{\times} \to (A^+)^{\times}$, for $i =1, \dots, n$ and $\tau \in \Psi$, such that
\[
\lambda = \sum_{i =1}^n \sum_{\tau \in \Psi} \xi_{i, \tau} \circ \lambda_{i, \tau}
\]
where the group structure on $X^*(T/T_0; A)$ is written additively.
\end{lemma}
\begin{proof}
Any such character $\lambda$ is a (unique) product of $r$-analytic characters $\alpha_{i, \tau} \colon \mbb{Z}_p^{\times} \to (A^+)^{\times}$ where $i=1, \dots, 2n$ and $\tau \in \Psi$, where $\alpha_{i, \tau}$ is determined by where it sends $y_{i, \tau}$. Since $\lambda$ is trivial on $T_0$, we have $\alpha_{i, \tau} = -\alpha_{2n+1-i, \tau}$ for all $i=1, \dots, 2n$ and $\tau \in \Psi$. One then defines 
\[
\xi_{i, \tau} = \left\{ \begin{array}{cc} \alpha_{i, \tau} - \alpha_{i+1, \tau} & \text{ for } i=1, \dots, n-1 \\ \alpha_{n, \tau} & \text{ for } i=n \end{array} \right.
\]
Uniqueness is a simple check.
\end{proof}

\begin{remark}
The above lemma implies that $\invs{W}_G$ is a finite disjoint union of $n[F^+:\mbb{Q}]$-dimensional open unit polydiscs.
\end{remark}

Let $S$ denote the torus $\prod_{\tau \neq \tau_0} \mbb{G}_m$, and for a Tate algebra $(A, A^+)$, let $X^*(S; A)$ denote the space of locally analytic characters $S(\mbb{Z}_p) \to (A^+)^{\times}$. A general element of $X^*(S; A)$ is a tuple $\beta = (\beta_{\tau})_{\tau \neq \tau_0}$, where $\beta_{\tau} \colon \mbb{Z}_p^{\times} \to (A^+)^{\times}$ are locally analytic. The functor $(A, A^+) \mapsto X^*(S; A)$ is representable by a $[F^+:\mbb{Q}]-1$-dimensional group adic space over $\opn{Spa}(\mbb{Q}_p, \mbb{Z}_p)$ denoted $\invs{W}_H$.

\begin{definition}
Let $X_0^*(T \times S; A) = X^*(T/T_0; A) \times X^*(S; A)$. The functor $(A, A^+) \mapsto X_0^*(T \times S; A)$ is then represented by $\invs{W} \defeq \invs{W}_G \times \invs{W}_H$.
\end{definition}

\subsection{Analytic and distribution modules}

We now define the relevant analytic and distribution modules. We introduce some notation:

\begin{notation}
For $\lambda \in X^*(T/T_0; A)$, we set $\kappa_n(\lambda) = w_n \star (-w_G^{\mathrm{max}}\lambda )$. We also define $\kappa_n(\lambda)^* = -w_{M_G}^{\mathrm{max}} \kappa_n(\lambda)$.
\end{notation}

\begin{definition} \label{DefinitionOfDistributions}
Let $\lambda \in X^*(T/T_0; A)$ be an $r_0$-analytic character, for some $r_0 \in \mbb{Z}_{>0}$. Set $S = \opn{Spa}(A, A^+)$. Then for any $r \geq r_0$, we define
\begin{align*} 
V_{G, \kappa_n(\lambda)^*}^{r\mathrm{-an}} &= \opn{anInd}_{\invs{M}^{\square}_{G, r} \cap \invs{B}_{M_G}}^{\invs{M}^{\square}_{G, r}} (w_{M_G}^{\mathrm{max}} \kappa_n(\lambda)^*) \\
 &\defeq \left\{ f \colon \left( \invs{M}^{\square}_{G, r} \right)_S \to \mbb{A}^{1, \mathrm{an}}_S : \begin{array}{c} f(mb) = (w_{M_G}^{\mathrm{max}} \kappa_n(\lambda)^*)(b^{-1}) f(m) \\ \text{ for all } b \in (\invs{M}^{\square}_{G, r} \cap \invs{B}_{M_G})_S \text{ and } m \in (\invs{M}^{\square}_{G, r})_S \end{array} \right\}
\end{align*}
as in \cite[\S 6.2.4]{BoxerPilloni}. This carries actions of $(\invs{M}^{\square}_{G, r})_S$ and $T^{M, +}$ by the formulae in \emph{loc.cit.}, where $T^{M, +} \subset T(\mbb{Q}_p)$ denotes the submonoid of elements $t \in T(\mbb{Q}_p)$ which satisfy $t B_{M_G}(\mbb{Z}_p) t^{-1} \subset B_{M_G}(\mbb{Z}_p)$. Note that $V_{G, \kappa_n(\lambda)^*}^{r\mathrm{-an}} \subset V_{G, \kappa_n(\lambda)^*}^{r'\mathrm{-an}}$ for $r' \geq r$, where the inclusion is given by restricting a function to $(\invs{M}^{\square}_{G, r'})_S$.

We write $\tilde{D}^{r\mathrm{-an}}_{G, \kappa_n(\lambda)}$ for the continuous $A$-dual of $V_{G, \kappa_n(\lambda)^*}^{r\mathrm{-an}}$, which carries actions of $(\invs{M}^{\square}_{G, r})_S$  and $T^{M, -} = (T^{M, +})^{-1}$ in the usual way. This is a Banach $A$-module but in general, it is not necessarily projective. To remedy this, one introduces the open subgroup $\invs{M}^{\square, \circ}_{G, r} = \invs{M}^{1, \circ}_{G, r} B_{M_G}(\mbb{Z}_p)$ where $\invs{M}^{1, \circ}_{G, r} \subset \invs{M}^1_{G, r}$ denotes the open subgroup of elements $m \equiv 1$ modulo $p^{r+\varepsilon}$ for some $\varepsilon > 0$. Note that this subgroup contains $\invs{M}^{\square}_{G, r+1}$. In \cite[\S 6.2.20]{BoxerPilloni}, the authors introduce a modification of the space of analytic functions $V_{G, \kappa_n(\lambda)^*}^{\circ, r\mathrm{-an}}$ using this open subgroup, and one has a $(\invs{M}^{\square, \circ}_{G, r}, T^{M, +})$-equivariant morphism $V_{G, \kappa_n(\lambda)^*}^{r\mathrm{-an}} \to V_{G, \kappa_n(\lambda)^*}^{\circ, r\mathrm{-an}}$ with dense image. One defines the space of $r$-analytic distributions $D^{r\mathrm{-an}}_{G, \kappa_n(\lambda)}$ to be the continuous $A$-dual of $V_{G, \kappa_n(\lambda)^*}^{\circ, r\mathrm{-an}}$, which is a projective Banach $A$-module. One has a $(\invs{M}^{\square, \circ}_{G, r}, T^{M, -})$-equivariant morphism $D^{r\mathrm{-an}}_{G, \kappa_n(\lambda)} \to \tilde{D}^{r\mathrm{-an}}_{G, \kappa_n(\lambda)}$ with dense image.
\end{definition}

We also introduce the following characters:

\begin{definition}
Let $(\lambda, \beta) \in X^*_0(T \times S; A)$ be an $r_0$-analytic character. Set $S = \opn{Spa}(A, A^+)$. Then for any $r \geq r_0$, we let $\sigma_n^{[\beta]}(\lambda) \colon (\invs{M}^{\clubsuit}_{H, r, 1})_S \to \mbb{G}_{m, S}^{\mathrm{an}}$ be the analytic character given by 
\[
(x; y_1, y_2, y_3; z_{1, \tau}, z_{2, \tau})_{\tau \neq \tau_0} \mapsto y_1^{-n-\xi_{n, \tau_0}} \opn{det}^{-\xi_{n, \tau_0}} \opn{det}y_3^{\xi_{n, \tau_0} +1} \prod_{\tau \neq \tau_0} \opn{det}z_{1, \tau}^{-\beta_{\tau}} \opn{det}z_{2, \tau}^{\beta_{\tau}}
\]
where $\xi_{i, \tau}$ are the characters associated with $\lambda$ as in Lemma \ref{GeneratingSetWeightSpace}.
\end{definition}

We obtain the following ``branching law in families'', which is an analytic version of Proposition \ref{ClassicalBranchingProp}. As the proof of this theorem is rather technical (and involves significantly changing the notation), we provide the proof in Appendix \ref{AppendixBranchingLaws}.

\begin{theorem} \label{PadicVectorThm}
Let $(A, A^+)$ be a Tate algebra over $(\mbb{Q}_p, \mbb{Z}_p)$ and $(\lambda, \beta) \in X^*_0(T \times S; A)$ which is $r_0$-analytic for some $r_0 \in \mbb{Z}_{>0}$. Then, for any $r \in \mbb{Z}$ such that $r \geq r_0$, there exists a non-zero vector $x_n^{[\beta]}(\lambda) \in V_{G, \kappa_n(\lambda)^*}^{r\mathrm{-an}}$ satisfying:
\begin{enumerate}
    \item The group $\invs{M}^{\clubsuit}_{H, r}$ acts on $x_n^{[\beta]}(\lambda)$ through the inverse of the character $\sigma_n^{[\beta]}(\lambda)$, via the embedding $u^{-1} \invs{M}^{\clubsuit}_{H, r} u \subset \invs{M}^{\square}_{G, r}$
    \item If $(B, B^+)$ denotes another Tate algebra with a morphism $(A, A^+) \to (B, B^+)$, and $(\lambda', \beta') \in X^*_0(T \times S; B)$ denotes the composition of $(\lambda, \beta)$ with this morphism, then the image of $x_n^{[\beta]}(\lambda)$ under the natural map
    \[
    V^{r\mathrm{-an}}_{G, \kappa_n(\lambda)^*} \to V^{r\mathrm{-an}}_{G, \kappa_n(\lambda')^*}
    \]
    is equal to $x_n^{[\beta']}(\lambda')$
    \item If $(\lambda, j) \in X^*(T/T_0)^+ \times X^*(S)$ is a pair of algebraic characters satisfying $0 \leq j_{\tau} \leq c_{n, \tau}$ for all $\tau \neq \tau_0$, then $x_n^{[\beta]}(\lambda)$ equals the image of $u^{-1} \cdot v_{\kappa_n}^{[j]}$ under the natural map
    \[
    V_{\kappa_n^*} \to V^{r\mathrm{-an}}_{G, \kappa_n(\lambda)^*} .
    \]
    Here any undefined notation is as in Proposition \ref{ClassicalBranchingProp}.
    \item The vector $x_n^{[\beta]}(\lambda)$ does not depend on the radius of analyticity (c.f. Theorem \ref{AppendixAMainThm}(4)).
\end{enumerate}
\end{theorem}
\begin{proof}
This follows from Theorem \ref{AppendixAMainThm}, noting that the character $\kappa_n(\lambda)^*$ satisfies the conditions in Lemma \ref{CoefficientRanalyticLemma} (because $\lambda$ is trivial on $T_0(\mbb{Z}_p)$), and this character specialises to a $M_G$-dominant character in $\mathcal{C}$ whenever $\lambda \in X^*(T/T_0)^+$.
\end{proof}

\begin{remark}
Note that if $(\lambda, j) \in X^*(T/T_0)^+ \times X^*(S)$ is a pair of algebraic characters as in Theorem \ref{PadicVectorThm}(3), then (after fixing an isomorphism $V_{\kappa_n} \cong V_{\kappa_n^*}^*$) we have a commutative diagram:
\[
\begin{tikzcd}
{D^{r\mathrm{-an}}_{G, \kappa_n(\lambda)}} \arrow[d] \arrow[rd] &                  \\
V_{\kappa_n} \arrow[r]                                      & {\sigma_n^{[j]}}
\end{tikzcd}
\]
where the vertical map is the dual of the map in Theorem \ref{PadicVectorThm}(3) restricted to $D^{r\mathrm{-an}}_{G, \kappa_n(\lambda)}$, the bottom map is pairing with the vector $u^{-1}\cdot v_{\kappa_n}^{[j]}$, and the diagonal map is evaluation at $x_{n}^{[\beta]}(\lambda)$. All of the maps are equivariant for the action of $\invs{M}^{\clubsuit}_{H, r+1}$ via the embedding $u^{-1} \invs{M}^{\clubsuit}_{H, r+1} u \subset \invs{M}^{\square}_{G, r+1} \subset \invs{M}^{\square, \circ}_{G, r}$.
\end{remark}

\subsection{Locally analytic cohomology} \label{LocallyAnalyticCohomologySubSec}

Let $(\underline{\lambda}, \beta) \in X_0^*(T \times S; A)$ be an $r_0$-analytic character, and let $t > k > r_0$ be integers. Let ${^\mu \invs{M}^{\circ}_{G, \mathrm{HT}, k-1, t}}$ denote the pushout of ${^\mu \invs{M}_{G, \mathrm{HT}, k, t}}$ along the inclusion $\invs{M}^{\square}_{G, k} \subset \invs{M}^{\square, \circ}_{G, k-1}$, and consider the base-extension of the torsor 
\[
\pi \times 1 \colon {^\mu \invs{M}^{\circ}_{G, \mathrm{HT}, k-1, t}} \times \opn{Spa}(A, A^+) \to \invs{U}^G_k(p^t) \times \opn{Spa}(A, A^{+}) .
\]
We define $[V^{\circ, (k-1)\mathrm{-an}}_{G, \kappa_n(\underline{\lambda})^*}]$ to be the subsheaf of $(\pi \times 1)_* \ordd_{{^\mu \invs{M}^{\circ}_{G, \mathrm{HT}, k-1, t}} \times \opn{Spa}(A, A^+)}$ of bounded sections which transform as $f(mb) = (w_{M_G}^{\mathrm{max}} \kappa_n(\lambda)^*)(b^{-1}) f(m)$ for all $b \in \invs{M}^{\square, \circ}_{G, k-1} \cap \invs{B}_{M_G}$. This defines a sheaf of topological modules over $\invs{U}^G_k(p^t)$ locally modelled on $V^{\circ, (k-1)\mathrm{-an}}_{G, \kappa_n(\underline{\lambda})^*}$ by the same proof as \cite[Proposition 6.3.3]{BoxerPilloni}. We define $[D_{G, \kappa_n(\underline{\lambda})}^{(k-1)\mathrm{-an}}]$ to be the continuous dual of $[V^{\circ, (k-1)\mathrm{-an}}_{G, \kappa_n(\underline{\lambda})^*}]$ which is a locally projective Banach sheaf locally modelled on the representation $D_{G, \kappa_n(\underline{\lambda})}^{(k-1)\mathrm{-an}}$.

\begin{remark}
The sheaf $[D_{G, \kappa_n(\underline{\lambda})}^{(k-1)\mathrm{-an}}]$ can alternatively be described as 
\[
\left( (\pi \times 1)_{*} \ordd_{{^\mu \invs{M}_{G, \opn{HT}, k, t}} \times \opn{Spa}(A, A^+)} \hatot D_{G, \kappa_n(\underline{\lambda})}^{(k-1)\mathrm{-an}} \right)^{\invs{M}^{\square}_{G, k}}
\]
where the invariants are via the (left) diagonal action and 
\[
\pi \times 1 \colon {^\mu \invs{M}_{G, \opn{HT}, k, t}} \times \opn{Spa}(A, A^+) \to \invs{U}^G_k(p^t) \times \opn{Spa}(A, A^{+}) 
\]
denotes the structural map.
\end{remark}

Let $t > m > k > r_0$ satisfy (\ref{mktstrong}). We can therefore form the cohomology
\[
R\Gamma^G_{w_n, \mathrm{an}}\left(\kappa_n(\underline{\lambda}) \right)^{-, \mathrm{fs}} \defeq R\Gamma_{\invs{I}^G_{m, k}(p^t)} \left( \invs{U}^G_k(p^t), [D_{G, \kappa_n(\underline{\lambda})}^{(k-1)\mathrm{-an}}] \right)^{-, \mathrm{fs}}
\]
where the finite-slope part is with respect to a certain power of $\invs{U}'_B(p^t)$ (by \cite[Theorem 6.4.3]{BoxerPilloni} and a similar calculation in the proof of Proposition \ref{ChangeOfSupportProp}). This definition is independent of the choice of $(m, k, t)$ by \cite[Theorem 6.4.5 and Theorem 6.4.8]{BoxerPilloni}. If one has a continuous morphism $(A, A^+) \to (\mbb{Q}_p, \mbb{Z}_p)$ such that the composition of this morphism with $\underline{\lambda}$ (denoted $\lambda$) lies in $X^*(T/T_0)^+$, then one has a natural specialisation map $R\Gamma^G_{w_n, \mathrm{an}}\left(\kappa_n(\underline{\lambda}) \right)^{-, \mathrm{fs}} \to R\Gamma^G_{w_n}(\kappa_n(\lambda))^{-, \mathrm{fs}}$ (after fixing an isomorphism $V_{\kappa_n(\lambda)} \cong V_{\kappa_n(\lambda)^*}^*$) arising from the map $D^{(k-1)\mathrm{-an}}_{G, \kappa_n(\lambda)} \to V_{\kappa_n(\lambda)}$. Furthermore, if $(A, A^+) = (\mbb{Q}_p, \mbb{Z}_p)$ then this specialisation map is an isomorphism on small slope parts (\cite[Corollary 6.8.4]{BoxerPilloni} using the improved slope bounds implied by Theorem 6.10.1 in \emph{op.cit.} because the Shimura variety is compact).

Similarly, we can also form the cohomology complexes:
\begin{align*}
R\Gamma^H_{\mathrm{id}, \mathrm{an}}\left( \invs{S}_{H, \diamondsuit}(p^t), \sigma_n^{[\beta]}(\underline{\lambda}) \right)^{(-, \dagger)} &\defeq \varprojlim_m R\Gamma_{\invs{Z}^H_m(p^t)} \left( \invs{U}^H_k(p^t), [\sigma_n^{[\beta]}(\underline{\lambda})] \right) \\
R\Gamma^H_{\mathrm{id}, \mathrm{an}}\left( \invs{S}_{H, \diamondsuit}(p^t), \sigma_n^{[\beta]}(\underline{\lambda})^{\vee} \right)^{(+, \dagger)} &\defeq \varinjlim_m R\Gamma \left( \invs{Z}^H_m(p^t), [\sigma_n^{[\beta]}(\underline{\lambda})]^{\vee} \right)
\end{align*}
for $k > r_0$ and $t \geq 1$, where the sheaves are defined using the torsor ${^\mu \invs{M}_{H, \mathrm{HT}, k, t}}$. The first definition is independent of $k$ by excision and Remark \ref{CompatibilityInKT}. As before, if $t=1$ then we omit the variety from the notation. If $(A, A^+) \to (\mbb{Q}_p, \mbb{Z}_p)$ is a continuous homomorphism and the composition of this morphism with $(\underline{\lambda}, \beta)$ (denoted $(\lambda, j)$) lies in $X^*(T/T_0)^+ \times X^*(S)$, then we have specialisation maps $R\Gamma^H_{\mathrm{id}, \mathrm{an}}( \sigma^{[\beta]}_n(\underline{\lambda}) )^{(-, \dagger)} \to R\Gamma^H_{\mathrm{id}}( \sigma^{[j]}_n(\lambda) )^{(-, \dagger)}$ and $R\Gamma^H_{\mathrm{id}, \mathrm{an}}( \sigma^{[\beta]}_n(\underline{\lambda})^\vee )^{(+, \dagger)} \to R\Gamma^H_{\mathrm{id}}( \sigma^{[j]}_n(\lambda)^\vee )^{(+, \dagger)}$.

\begin{proposition}
Let $(\underline{\lambda}, \beta) \in X_0^*(T \times S; A)$ be an $r_0$-analytic character. Then we have a well-defined $A$-linear map
\begin{equation} \label{AnalyticPullbackMorphism}
R\Gamma^G_{w_n, \mathrm{an}}\left(\kappa_n(\underline{\lambda}) \right)^{-, \mathrm{fs}} \to R\Gamma^H_{\mathrm{id}, \mathrm{an}}\left( \sigma_n^{[\beta]}(\underline{\lambda}) \right)^{(-, \dagger)}
\end{equation}
which satisfies:
\begin{enumerate}
    \item If $(A, A^+) \to (B, B^+)$ is a morphism of Tate algebras over $(\mbb{Q}_p, \mbb{Z}_p)$, and $(\underline{\lambda}', \beta') \in X_0^*(T \times S; B)$ denotes the induced character, then the morphisms in (\ref{AnalyticPullbackMorphism}) for the pairs $(\underline{\lambda}, \beta)$ and $(\underline{\lambda}', \beta')$ are compatible under base-change along the morphism $(A, A^+) \to (B, B^+)$. 
    \item If $(A, A^+) = (\mbb{Q}_p, \mbb{Z}_p)$ and $(\underline{\lambda}, \beta) = (\lambda, j)$ is algebraic as in Theorem \ref{PadicVectorThm}(3), then one has a commutative diagram:
    \[
\begin{tikzcd}
{R\Gamma^G_{w_n, \mathrm{an}}\left(\kappa_n(\lambda) \right)^{-, \mathrm{fs}}} \arrow[r, "(\ref{AnalyticPullbackMorphism})"] \arrow[d] & {R\Gamma^H_{\mathrm{id}, \mathrm{an}}\left( \sigma_n^{[j]}(\lambda) \right)^{(-, \dagger)}} \arrow[d, equals] \\
{R\Gamma^G_{w_n}\left(\kappa_n(\lambda) \right)^{-, \mathrm{fs}}} \arrow[r]                                                            & {R\Gamma^H_{\mathrm{id}}\left( \sigma_n^{[j]}(\lambda) \right)^{(-, \dagger)}}                       
\end{tikzcd}
\]
    
    where the bottom map is the one in Proposition \ref{MinusPullbackProp}.
\end{enumerate}
\end{proposition}
\begin{proof}
This is constructed in a similar way as Proposition \ref{MinusPullbackProp}, using the morphism of sheaves $\widehat{\iota}[D^{(k-1)\mathrm{-an}}_{G, \kappa_n(\underline{\lambda})}] \to [\sigma_n^{[\beta]}(\underline{\lambda})]$ arising from evaluation at the vector $x_n^{[\beta]}(\underline{\lambda})$, i.e. the pullback is constructed using a triple $(m, k, t)$ satisfying (\ref{mktstrong}) and then one traces down to level $K^H_{\diamondsuit}(p)$. Parts (1) and (2) follow from the properties of the vector $x_n(\underline{\lambda})$ in Theorem \ref{PadicVectorThm}.
\end{proof}

We have a Serre duality pairing between the complexes $R\Gamma^H_{\mathrm{id}, \mathrm{an}}(\cdots)^{(-, \dagger)}$ and $R\Gamma^H_{\mathrm{id}, \mathrm{an}}(\cdots)^{(+, \dagger)}$ which is compatible with the duality in \S \ref{PullbackAdicFunctoriality} via the specialisation maps above. Therefore we obtain a pairing:
\[
\boxed{%
\langle\langle \cdot, \cdot \rangle\rangle^{-}_{\mathrm{an}} \colon \opn{H}^{n-1}_{w_n, \mathrm{an}}\left( \kappa_n(\underline{\lambda}) \right)^{-, \mathrm{fs}} \times \opn{H}^0_{\mathrm{id}, \mathrm{an}}\left( \sigma_n^{[\beta]}(\underline{\lambda})^{\vee} \right)^{(+, \dagger)} \to A
}
\]
which is compatible under change of coefficients. We have the following compatibility with the previously defined pairings:

\begin{corollary} \label{DistAnIsAnCorollary}
Let $f \colon (A, A^+) \to (\mbb{Q}_p, \mbb{Z}_p)$ be a homomorphism over $(\mbb{Q}_p, \mbb{Z}_p)$, and suppose that the character $(\lambda, j)$, induced from composing $(\underline{\lambda}, \beta)$ with this morphism, is algebraic as in Theorem \ref{PadicVectorThm}(3). Then for any 
\begin{itemize}
    \item $\underline{\eta} \in \opn{H}^{n-1}_{w_n, \mathrm{an}}\left( \kappa_n(\underline{\lambda}) \right)^{-, \mathrm{fs}}$
    \item $\underline{\chi} \in \opn{H}^0_{\mathrm{id}, \mathrm{an}}\left( \sigma_n^{[\beta]}(\underline{\lambda})^{\vee} \right)^{(+, \dagger)}$
\end{itemize}
one has $f(\langle\langle \underline{\eta}, \underline{\chi} \rangle\rangle^{-}_{\mathrm{an}}) = \langle \eta, \chi \rangle^{-}_{\mathrm{an}}$, where $\eta$ and $\chi$ denote the specialisations of $\underline{\eta}$ and $\underline{\chi}$ respectively under the morphism $f$.
\end{corollary}

\begin{remark}
There are analogous constructions of all the various pairings in \S \ref{PullbacksOnAdicSVs}--\ref{LocallyAnalyticCohomologySection} working over a finite extension $L/\mbb{Q}_p$ and they are related by base-change of coefficients. This will be important in the construction of the $p$-adic $L$-function, because we will have to enlarge the field of definition to include the Hecke eigenvalues of the relevant automorphic representation/character. 
\end{remark}

%%%%%%%%%%%%%%%%%%%%%%%%%%%%%%%%%%%%%%%%%%%%%%%%%%%%%%%%%%%%%%%%%%%%%%%%%%%%%%%%%%%%%%%%%%%%%%%
%%%%%%%%%%%      FAMILIES OF COHOMOLOGY CLASSES    %%%%%%%%%%%%%%%%%%%%%%%%%%%%%%%%%%%%%%%%%%%%
%%%%%%%%%%%%%%%%%%%%%%%%%%%%%%%%%%%%%%%%%%%%%%%%%%%%%%%%%%%%%%%%%%%%%%%%%%%%%%%%%%%%%%%%%%%%%%%

\section{Families of cohomology classes} \label{FamiliesOfCohomologyClassesSection}

In this section we show that, under some hypotheses on the ramification of the automorphic representation $\pi$, there exists a family of cohomology classes in $\opn{H}^{n-1}_{w_n, \mathrm{an}}(\kappa_n(\lambda_A))^{-, \mathrm{fs}}$ corresponding to a family of automorphic representations passing through $\pi$. This family of cohomology classes will be one half of the input for the pairing $\langle\langle \cdot, \cdot \rangle\rangle^{-}_{\mathrm{an}}$ when constructing the $p$-adic $L$-function in \S \ref{ConstructionOfPadicLSection}. Recall that we have assumed $F$ contains an imaginary quadratic number field $E$. This will be important when speaking about automorphic base-change for unitary similitude groups.

\subsection{Families for the group \texorpdfstring{$\mbf{G}$}{G}}

Let $\pi$ be a cuspidal automorphic representation of $\mbf{G}(\mbb{A})$ such that $\pi_{\infty}$ lies in the discrete series. We impose the following assumptions:

\begin{assumption} \label{HCassumptionOnPi}
Assume that:
\begin{enumerate}
    \item The Harish-Chandra parameter of $\pi_{\infty}$ is of the form $w_n \cdot (\lambda_{\pi} + \rho)$ for some $\lambda_{\pi} \in X^*(T/T_0)^+$ (see \S \ref{DiscreteSeriesRepsSection})
    \item Any weak base-change of $\pi$ to an automorphic representation of $\opn{GL}_1(\mbb{A}_E) \times \opn{GL}_{2n}(\mbb{A}_F)$ is cuspidal\footnote{Here by weak base-change, we mean an automorphic representation of $\opn{GL}_1(\mbb{A}_E) \times \opn{GL}_{2n}(\mbb{A}_F)$ satisfying the conditions in \cite[Theorem A.1]{ShinAppendix} (the theorem of course shows that such a base-change exists).} 
    \item There exist compact open subgroups $K_p \subset \mbf{G}(\mbb{Q}_p)$ and $K^p \subset \mbf{G}(\mbb{A}^p_f)$ with $K_p$ hyperspecial, such that $K = K^pK_p$ is sufficiently small and 
    \[
    \opn{dim}_{\mbb{C}}\pi_f^K = 1 .
    \]
\end{enumerate}
\end{assumption}

\begin{remark}
Under the additional assumptions below, Assumption \ref{HCassumptionOnPi}(3) is not a severe restriction thanks to the local newform theory for general linear groups. More precisely, under Assumption \ref{CMramificationAssumption} below, the local component of $\pi$ at any ramified prime occurs as the local component of its cuspidal base-change to $\opn{GL}_1(\mbb{A}_E) \times \opn{GL}_{2n}(\mbb{A}_F)$, and is therefore generic. In particular, by \cite{JPSS81}, there exists a compact open subgroup $K = K^pK_p$ with $K_p$ hyperspecial, such that $\opn{dim}_{\mbb{C}}\pi_f^K = 1$. If $K$ is neat then Assumption \ref{HCassumptionOnPi}(3) holds, otherwise one can use a similar strategy as in \cite[Remark 3.2.1]{LZBK21} to handle more general levels. 
\end{remark}

Fix a finite set of primes $S$ containing $p$ and all primes where $K^p$ is not a good special maximal compact open subgroup as in Lemma \ref{ParahoricRestrictionLemma}. Let $\mbb{T}^{-}$ denote the Hecke algebra (over $\mbb{Q}$) given by
\[
\mbb{T}^{-} = \opn{C}^{\infty}\left( K^S \backslash \mbf{G}(\mbb{A}^S_f) / K^S \right) \otimes \mbb{Q}[T^{-}]
\]
where the convolution product for the first factor is with respect to a fixed Haar measure on $\mbf{G}$. We fix a $\mbb{C}$-algebra homomorphism $\theta_{\pi} \colon \mbb{T}^{-}_{\mbb{C}} \to \mbb{C}$ which is an eigencharacter for the action of $\mbb{T}^{-}_{\mbb{C}}$ on $\pi_f^{K^pK^G_{\mathrm{Iw}}(p)}$. By Assumption \ref{HCassumptionOnPi}(3), this homomorphism has finite-slope at $p$, so gives rise to a monoid homomorphism $\theta_{\pi, p} \colon T^{-} \to \mbb{C}^{\times}$. We let $I_{\pi}$ denote the kernel of the morphism $\theta_{\pi}$.

\begin{lemma}
There exists a number field $\Phi$ containing $F$, such that $\theta_{\pi}$ is defined over $\Phi$.
\end{lemma}
\begin{proof}
Let $\psi \boxtimes \Pi_0$ denote the weak base-change of $\pi$ to $\opn{GL}_1(\mbb{A}_E) \times \opn{GL}_{2n}(\mbb{A}_F)$. By \cite[Theorem 5.2.1]{LabesseSchwermer}, there exists $\pi_0 \subset \pi|_{\mbf{G}_0(\mbb{Q}_\ell)}$ cuspidal automorphic such that $\Pi_0$ is the weak base-change of $\pi_0$. Since $\Pi_0$ is cuspidal, we have $\opn{BC}_{\ell}(\pi_{0, \ell}) \cong \Pi_{0, \ell}$ for all rational primes $\ell$, where $\opn{BC}_{\ell}$ denotes the local (standard) base-change map (see \cite[\S C.3]{LTXZZ19}). 

This implies that the homomorphism $\theta_{\pi}$ matches with the Hecke eigensystem for $\psi \boxtimes \Pi_0$, which is regular algebraic. The result then follows from \cite[Proposition 3.4.3]{GR2014} and the references therein (note that $F$ is taken to be a totally real field in \emph{op.cit.}, but the cited result holds in general via the same proof).  
\end{proof}

The above lemma implies that we can view $\theta_{\pi}$ as a homomorphism valued in any field extension of $\Phi$. For example, if we let $L$ denote the completion of the image of $\Phi$ under the fixed isomorphism $\mbb{C} \cong \Qpb$, then $L/\mbb{Q}_p$ is a finite extension and we can view $\theta_{\pi}$ as an $L$-algebra homomorphism $\mbb{T}^{-}_L \to L$. This leads to the following small slope assumption:

\begin{assumption} \label{SSSassumption}
We assume that the monoid homomorphism $\theta_{\pi, p} \colon T^- \to L^{\times}$ is of small slope (with respect to $\kappa_n = w_n \star (-w_G^{\mathrm{max}} \lambda_{\pi})$). 
\end{assumption}

\begin{example} \label{BorelImpliesSSExample}
Let $\lambda_{\pi}^* = -w_{G}^{\mathrm{max}} \cdot \lambda_{\pi}$ (which is in fact equal to $\lambda_{\pi}$ by Assumption \ref{HCassumptionOnPi}(1)). We say that $\pi$ is Borel ordinary if $\lambda_{\pi}^*(x)^{-1}\theta_{\pi, p}(x)$ is a $p$-adic unit, where $x \in T^{--}$ is the element in Definition \ref{BorelHeckeOperator}. As seen below, $\pi$ contributes to the coherent cohomology of $S_{\mbf{G}, \mathrm{Iw}}(p)$, and the slope bounds in \cite[Conjecture 5.9.2]{BoxerPilloni} hold because the Shimura variety is compact (see Theorem 6.48 in \emph{op.cit.}). Therefore, being Borel ordinary in fact implies that the homomorphism $(-\lambda_{\pi}^*) \cdot \theta_{\pi, p}$ is valued in $\ordd_L^{\times}$.

Suppose that $\pi$ is Borel ordinary. Then we will show that $\theta_{\pi, p}$ is of small slope. For this, it is enough to calculate, for $i\neq n$, the $\tau_0$-component of $\delta_i \defeq w_i^{-1} \star \kappa_n - \lambda_{\pi}^*$ and show that there exists $x \in T^{-}$ such that $v(\delta_i(x)) > 0$. For $1 \leq i \leq 2n-1$, let $x_i \in T^{-}$ be the element which is the identity outside the $\tau_0$-component, and equal to $(1, \dots, 1, p, \dots, p)$ in the $\tau_0$-component (where there are $i$ lots of $p$). Write $\lambda_{\pi} = (0; c_{1, \tau}, \dots, c_{2n, \tau})_{\tau \in \Psi}$. We break the analysis into two cases.

Suppose that $i < n$. Then the action of $w_i^{-1}$ only affects the first $i+1$ entries of the $\tau_0$-component of the weight. In this case, we take $x = x_n$ and find that $v(\delta_i(x)) = 2c_{n, \tau_0} + 1 > 0$ because $c_{n, \tau_0} \geq 0$ (Assumption \ref{HCassumptionOnPi}(1) ).

Suppose that $i = n + \varepsilon$ for an integer $1 \leq \varepsilon \leq n-1$. Then the last $n-\varepsilon$ entries of the $\tau_0$-component of $\delta_i$ are $c_{n-\varepsilon} - c_n + \varepsilon, 0, \dots, 0$ (using the fact that $c_{j, \tau_0} = -c_{2n+1-j, \tau_0}$). We then take $x = x_{n-\varepsilon}$ and conclude that $v(\delta_i(x)) = c_{n-\varepsilon} - c_n + \varepsilon > 0$ because $\lambda_{\pi}$ is dominant.
\end{example}

Recall that we can view $X^*(T/T_0)^+$ as a subset of $\invs{W}_G(\mbb{Q}_p)$ (we will refer to this subset as the classical weights). We now introduce the notion of a family of automorphic representations and cohomology classes.

\begin{definition} \label{DefinitionOfFamilyThroughPi}
By a family $\underline{\pi}$ over an open affinoid $U = \opn{Spa}(A, A^+) \subset \invs{W}_{G, L}$ containing $\lambda_{\pi}$, passing through $\pi$, we mean an $A$-algebra homorphism 
\[
\theta_{\underline{\pi}} \colon \mbb{T}^{-}_A \to A
\]
such that for all but finitely many classical weights $\lambda \in U \cap X^*(T/T_0)^+$, there exists a cuspidal automorphic representation $\sigma$ such that the specialisation of $\theta_{\underline{\pi}}$ at $\lambda$ is an eigencharacter for the action of $\mbb{T}^{-}_L$ on $\sigma^{K^pK^G_{\mathrm{Iw}}(p)}$ (under the identification $\mbb{C} \cong \Qpb$).

Let $\eta \in \opn{H}^{n-1}\left( \invs{S}_{G, \mathrm{Iw}}(p), [V_{\kappa_n}] \right)^{-, \mathrm{ss}}$ be an eigenvector for the action of $\mbb{T}^-_L$ with eigencharacter $\theta_{\pi}$. Let $\lambda_A \colon T(\mbb{Z}_p) \to (A^+)^{\times}$ denote the universal character associated with $U$. If such a family $\underline{\pi}$ exists then, by a family $\underline{\eta}$ of cohomology classes passing through $\eta$, we mean an eigenvector $\underline{\eta} \in \opn{H}^{n-1}_{w_n, \mathrm{an}}(\kappa_n(\lambda_A))^{-, \mathrm{fs}}$ for the action of $\mbb{T}^{-}_{A}$ with eigencharacter $\theta_{\underline{\pi}}$, whose specialisation at $\lambda_{\pi}$ equals $\eta$ under the comparison isomorphism
\[
\opn{H}^{n-1}_{w_n, \mathrm{an}}\left( \kappa_n(\lambda_{\pi}) \right)^{-, \mathrm{ss}} \cong \opn{H}^{n-1}\left( \invs{S}_{G, \mathrm{Iw}}(p), [V_{\kappa_n}] \right)^{-, \mathrm{ss}} .
\]
\end{definition}

\subsection{Existence of families} \label{ExistenceOfFamilies}

In this section, we introduce some further assumptions on $\pi$ which ensure the existence of a family passing through $\pi$ as well as a family of cohomology classes. We begin with the following ramification assumption on the representation $\pi$:

\begin{assumption} \label{CMramificationAssumption}
Assume that:
\begin{enumerate}
    \item The set $S$ above contains only primes which split in $E/\mbb{Q}$, i.e. $K^S = \prod_{\ell \not\in S} K_\ell$ where $K_{\ell} \subset \mbf{G}(\mbb{Q}_{\ell})$ is a good special maximal compact open. We further assume that $K_\ell$ is hyperspecial if $\mbf{G}_{\mbb{Q}_{\ell}}$ is unramified (for $\ell \not\in S$).
    \item The eigencharacter $\theta_{\pi, p}$ appears with multiplicity one for the action on $\pi_p^{K^G_{\mathrm{Iw}}(p)}$.
\end{enumerate}
\end{assumption}

As a consequence of this assumption, we have:

\begin{lemma} \label{InjectivityOfBaseChange}
Suppose that $\pi$ satisfies Assumption \ref{CMramificationAssumption} (as well as the assumptions in the previous section). Let $\sigma$ be a cuspidal automorphic representation of $\mbf{G}(\mbb{A})$ such that $\sigma_{\infty}$ is cohomological. Suppose that $\sigma_f^K \neq 0$ and $\pi_{\ell} \cong \sigma_{\ell}$ for all $\ell \not\in S$. Then $\pi_f \cong \sigma_f$.
\end{lemma}
\begin{proof}
This is an application of Proposition \ref{GlobalMultOneProp}.
\end{proof}

We obtain the following corollary:

\begin{corollary} \label{EtalenessOfEigenvarietyCor}
Let $\pi$ be as in Lemma \ref{InjectivityOfBaseChange} and set $\kappa_n = w_n \star (-w_{G}^{\mathrm{max}}\cdot\lambda_{\pi})$. Then the localised cohomology
\[
\opn{H}^{n-1}\left( \invs{S}_{G, \mathrm{Iw}}(p), [V_{\kappa_n}] \right)_{I_{\pi}}
\]
is one-dimensional (over $L$).\footnote{We are abusing notation slightly -- by the localisation $(\cdots )_{I_{\pi}}$ we mean first base-change to $L$ and then localise at $I_{\pi}$ (the kernel of the map $\mbb{T}^{-}_L \to L$).}
\end{corollary}
\begin{proof}
Via the rigid GAGA comparison, this localised cohomology group has the same dimension as 
\[
\opn{H}^{n-1}\left( S(\mbb{C}), [V_{\kappa_n}] \right)_{I_{\pi}}
\]
where $S = S_{\mbf{G}, \mathrm{Iw}}(p)$ and we are considering its sheaf cohomology with coefficients in $[V_{\kappa_n}]$.

Let $A_{\mbf{G}} \cong \mbb{G}_m$ denote the maximal split torus inside the centre of $\mbf{G}$, and let $A_{\mbf{G}}(\mbb{R})^{\circ}$ denote the connected component of the identity in the analytic topology. Let $K_{\infty}^{\circ} \subset K_{\infty}$ denote the maximal compact subgroup, where $K_{\infty} = A_{\mbf{G}}(\mbb{R})^{\circ}K_{\infty}^{\circ}$ is as in \S \ref{DiscreteSeriesRepsSection}. Let $\ide{p}$ denote the Lie algebra of the \emph{opposite} of $P_{\mbf{G}}$, and we can write
\[
\ide{p} = \ide{p}^{\circ} \oplus \ide{a}_G
\]
where $\ide{a}_G$ is the Lie algebra of $A_G$ and $\ide{p}^{\circ} = \ide{p} \cap \ide{g}_0$, where $\ide{g}_0$ denotes the Lie algebra of $\mbf{G}_0$.

By \cite{Su19}, we have the following description 
\begin{equation} \label{PKcohomologyDescription}
\opn{H}^{n-1}\left( S(\mbb{C}), [V_{\kappa_n}] \right) = \bigoplus_{\sigma} \left( \opn{H}^{n-1}_{(\ide{p}^{\circ}, K_{\infty}^{\circ})}\left( \sigma_{\infty} \otimes V_{\kappa_n} \right) \otimes \sigma_f^{K^pK^G_{\mathrm{Iw}}(p)} \right)^{m(\sigma)}
\end{equation}
where the sum runs over all cuspidal automorphic representations $\sigma$ of $\mbf{G}(\mbb{A})$ which lie in the discrete spectrum (with multiplicity $m(\sigma)$), and are such that $A_{\mbf{G}}(\mbb{R})^{\circ}$ acts trivially on $\sigma_{\infty}$. Since $\ide{a}_G$ and $A_{\mbf{G}}(\mbb{R})^{\circ}$ act trivially on $\sigma_{\infty} \otimes V_{\kappa_n}$, we have
\[
\opn{H}^{n-1}_{(\ide{p}^{\circ}, K_{\infty}^{\circ})}\left( \sigma_{\infty} \otimes V_{\kappa_n} \right) = \opn{H}^{n-1}_{(\ide{p}, K_{\infty})}\left( \sigma_{\infty} \otimes V_{\kappa_n} \right).
\]
By the Hodge decomposition (see \cite{LP18} for example) of the singular cohomology $\opn{H}^{2n-1}\left(S(\mbb{C}), W_{\lambda_{\pi}} \right)$ with coefficients in the algebraic representation with highest weight $\lambda_{\pi}$, we see that $\sigma_{\infty}$ is cohomological if 
\[
\opn{H}^{n-1}_{(\ide{p}, K_{\infty})}\left( \sigma_{\infty} \otimes V_{\kappa_n} \right) \otimes \sigma_f^{K^pK^G_{\mathrm{Iw}}(p)} \neq 0.
\]
Furthermore, if this space is non-zero after localising at $I_{\pi}$, the conditions in Lemma \ref{InjectivityOfBaseChange} are satisfied for $\sigma$. 

Note that if $\sigma$ satisfies $\sigma_f \cong \pi_f$ then by the strong base-change results in \cite{Mok} and \cite{KMSW14} (and that $A_{\mbf{G}}(\mbb{R})^{\circ}$ acts trivially on $\sigma_{\infty}$), $\sigma_{\infty}$ must lie in the same $L$-packet for $\pi_{\infty}$. By \cite[Theorem 3.2.1]{BHR94}, if the vector space $\opn{H}^{n-1}_{(\ide{p}, K_{\infty})}\left( \sigma_{\infty} \otimes V_{\kappa_n} \right)$ is non-zero, then we must have $\sigma_{\infty} \cong \pi_{\infty}$ and $\opn{H}^{n-1}_{(\ide{p}, K_{\infty})}\left( \pi_{\infty} \otimes V_{\kappa_n} \right)$ is one-dimensional. Therefore, localising (\ref{PKcohomologyDescription}) at the ideal $I_{\pi}$, we see that
\[
\opn{H}^{n-1}\left( S(\mbb{C}), [V_{\kappa_n}] \right)_{I_{\pi}} = \left( \opn{H}^{n-1}_{(\ide{p}, K_{\infty})}\left( \pi_{\infty} \otimes V_{\kappa_n} \right) \otimes \pi_f^{K^pK^G_{\mathrm{Iw}}(p)}[\theta_{\pi, p}] \right)^{m(\pi)} .
\]
where $\pi_f^{K^pK^G_{\mathrm{Iw}}(p)}[\theta_{\pi, p}]$ denotes the (generalised) eigenspace for the character $\theta_{\pi, p}$.

By Assumption \ref{CMramificationAssumption}, we therefore see that the dimension of the cohomology group in the statement of the corollary is equal to $m(\pi)$. Since $m(\pi) > 0$ (by definition), it is enough to show that $m(\pi) \leq 1$. But there is an injective $\mbf{G}_0$-equivariant restriction map
\[
L^2_{\mathrm{disc}}(\mbf{G}) \hookrightarrow L^2_{\mathrm{disc}}(\mbf{G}_0)
\]
from the discrete spectrum of $\mbf{G}$ to that of $\mbf{G}_0$ (see \cite[Theorem 1.1.1]{LabesseSchwermer}), hence it is enough to show that the multiplicity of any cuspidal automorphic representation in $L^2_{\mathrm{disc}}(\mbf{G}_0)$ is at most $1$. But this follows from Arthur's multiplicity formula for unitary groups (see \cite{CZ21}).
\end{proof}

Recall that we have classicality isomorphisms on the small slope part
\[
R\Gamma^G_{w_n, \mathrm{an}}(\kappa_n)^{-, \mathrm{ss}} \cong R\Gamma^G_{w_n}(\kappa_n)^{-, \mathrm{ss}} \cong R\Gamma(\invs{S}_{G, \mathrm{Iw}}(p), [V_{\kappa_n}])^{-, \mathrm{ss}}.
\]
Note that the cohomology of the right-hand side vanishes outside degree $n-1$, and since $\theta_{\pi}$ is of small slope, we see that $R\Gamma^G_{w_n, \mathrm{an}}(\kappa_n)_{I_{\pi}}$ has cohomology concentrated in degree $n-1$ where it is free of rank one (over $L$).

The Tor-spectral sequence 
\[
E_2^{p, q} \colon \opn{Tor}_{-p}^A\left( \opn{H}^q_{w_n, \mathrm{an}}(\kappa_n(\lambda_A))^{-, \mathrm{fs}}, \lambda_{\pi} \right) \Rightarrow \opn{H}^{p+q}_{w_n, \mathrm{an}}\left( \kappa_n(\lambda_{\pi}) \right)^{-, \mathrm{fs}}
\]
therefore implies that there exists an affinoid $U = \opn{Spa}(A, A^+) \subset (\invs{W}_G)_L$ containing $\lambda_{\pi}$, such that 
\[
R\Gamma^G_{w_n, \mathrm{an}}(\kappa_n(\lambda_A))_{I_{\pi}}
\]
has cohomology concentrated in degree $n-1$ where it is free of rank one over the stalk of $A$ at $\lambda_{\pi}$. Here $\lambda_A \colon T(\mbb{Z}_p) \to (A^+)^{\times}$ denotes the universal character (which is trivial on $T_0(\mbb{Z}_p)$).

The construction in \cite[\S 6.9]{BoxerPilloni} gives rise to an eigenvariety $\invs{E} \to \invs{W}_G$ which is locally quasi-finite and partially proper, and parameterises finite-slope Hecke eigensystems appearing in the coherent cohomology of $\invs{S}_{G, \mathrm{Iw}}(p)$.\footnote{This is not the ``full eigenvariety'' but rather the pullback of the eigenvariety constructed in \cite[\S 6.9]{BoxerPilloni} along the closed embedding $\invs{W}_G \hookrightarrow \invs{W}_G^{\mathrm{full}}$, where $\invs{W}_G^{\mathrm{full}}$ is the weight space parameterising characters of $T(\mbb{Z}_p)$. Furthermore, including level subgroups which are good special maximal compact open but not hyperspecial does not affect the construction.} In particular, we have coherent sheaves $\widetilde{\invs{M}}^{\bullet, -, \mathrm{fs}}_{w_n}$ whose pushforward to $\invs{W}_G$ recovers the cohomology groups $\opn{H}^{\bullet}_{w_n, \mathrm{an}}(\cdots )^{-, \mathrm{fs}}$, and the ideal $I_{\pi}$ gives a point $x \in \invs{E}(L)$. Since $R\Gamma^G_{w_n, \mathrm{an}}(\kappa_n(\lambda_A))_{I_{\pi}}$ has cohomology concentrated in degree $n-1$ where it is free of rank one over the stalk of $A$ at $\lambda_{\pi}$, we can (after shrinking $U$) find an open affinoid neighbourhood $V \subset \invs{E}_L$ of $x$ such that the induced map $V \to U$ is an isomorphism. In particular, this implies:

\begin{theorem}
Shrinking $U$ if necessary
\begin{enumerate}
    \item There exists a unique family $\underline{\pi}$ over $U$ passing through $\pi$
    \item The generalised eigenspace $S^{n-1}(\underline{\pi}) \subset \opn{H}^{n-1}_{w_n, \mathrm{an}}(\kappa_n(\lambda_A))^{-, \mathrm{fs}}$ on which $\mbb{T}^{-}_A$ acts through the character $\theta_{\underline{\pi}}$, is a direct summand that is free of rank one over $A$. In particular, a basis $\underline{\eta}$ of $S^{n-1}(\underline{\pi})$ is a family of cohomology classes passing through a basis $\eta$ of $\opn{H}^{n-1}\left( \invs{S}_{G, \mathrm{Iw}}(p), [V_{\kappa_n}] \right)_{I_{\pi}}$.
\end{enumerate}
\end{theorem}
\begin{proof}
The above discussion implies that there exists a character $\theta_{\underline{\pi}}$ specialising to $\theta_{\pi}$ at $\lambda_{\pi}$ and satisfying (2), so we just need to show that $\theta_{\underline{\pi}}$ defines a unique family. But the fact that $\theta_{\underline{\pi}}$ arises from the eigenvariety $\invs{E}$ implies that for any $\lambda \in U \cap X^*(T/T_0)^+$, the specialisation of $\theta_{\underline{\pi}}$ is an eigencharacter for the action of $\mbb{T}^{-}_L$ on $\opn{H}^{n-1}_{w_n, \mathrm{an}}(\kappa_n(\lambda))^{-, \mathrm{fs}}$. Shrinking $U$ if necessary, we can ensure that it is of small slope, so (under the identification $\mbb{C} \cong \Qpb$) contributes to $\opn{H}^{n-1}(S_{\mbf{G}, \mathrm{Iw}}(\mbb{C}), [V_{\kappa_n(\lambda)}])$ \emph{with multiplicity one}. The description in (\ref{PKcohomologyDescription}) holds for this cohomology group, and therefore, letting $I$ denote the kernel of the specialisation $\theta$ of $\theta_{\underline{\pi}}$ at $\lambda$, we must have a Hecke-equivariant isomorphism
\[
\opn{H}^{n-1}(S_{\mbf{G}, \mathrm{Iw}}(\mbb{C}), [V_{\kappa_n(\lambda)}])_{I} \cong \sigma_f^{K^pK^G_{\mathrm{Iw}}(p)} [\theta_p] 
\]
for some cuspidal automorphic representation $\sigma$, since we know the dimension of the left-hand side is one.
\end{proof}

\begin{remark} \label{TheSpecialisationRemark}
We will refer to $\sigma$ in the above theorem as \emph{the} specialisation of $\theta_{\underline{\pi}}$ at $\lambda$, even though there will be several automorphic representations $\sigma'$ which have the same Hecke eigenvalues. Note that, by the Hodge decomposition, $\sigma_{\infty}$ is cohomological with respect to the algebraic representation of $\mbf{G}(\mbb{C})$ with highest weight $\lambda$.
\end{remark}

%%%%%%%%%%%%%%%%%%%%%%%%%%%%%%%%%%%%%%%%%%%%%%%%%%%%%%%%%%%%%%%%%%%%%%%%%%%%%%%%%%%%%%%%%%%%%
%%%%%%%%    FAMILIES OF AC CHARACTERS  %%%%%%%%%%%%%%%%%%%%%%%%%%%%%%%%%%%%%%%%%%%%%%%%%%%%%%
%%%%%%%%%%%%%%%%%%%%%%%%%%%%%%%%%%%%%%%%%%%%%%%%%%%%%%%%%%%%%%%%%%%%%%%%%%%%%%%%%%%%%%%%%%%%%

\section{Families of anticyclotomic characters} \label{FamiliesOfACcharsSection}

In this section we exhibit families of anticyclotomic characters in the coherent cohomology of $\invs{S}_{H, \diamondsuit}(p)$. 

\subsection{Anticyclotomic characters}

Let $\mbf{R}$ denote the unitary similitude group associated with the Hermitian space $\bigwedge^n_F W_1 \oplus \bigwedge^n_F W_2$ (with common similitude on each factor) where $W_1$ and $W_2$ are the Hermitian spaces in \S \ref{PreliminarySection}. This can be upgraded to a PEL Shimura datum via the homomorphism $h_{\mbf{R}} \defeq \opn{det} \circ h_{\mbf{H}}$ and has Hodge cocharacter $\mu_{\mbf{R}} \defeq \opn{det} \circ \mu_{\mbf{H}}$. Here $\opn{det} \colon \mbf{H} \to \mbf{R}$ denotes the homomorphism given by $(h_1, h_2) \mapsto (\opn{det}h_1, \opn{det}h_2)$. By design, one has a morphism of Shimura data $(\mbf{H}, h_{\mbf{H}}) \to (\mbf{R}, h_{\mbf{R}})$. Note that $\mu_{\mbf{R}}$ is central in $\mbf{R}_{F^{\opn{cl}}}$, so the associated parabolics and Levi are all equal to 
\[
\mbf{R}_{F^{\opn{cl}}} \cong \mbb{G}_{m, F^{\opn{cl}}} \times \prod_{\tau \in \Psi} \left( \mbb{G}_{m, F^{\opn{cl}}} \times \mbb{G}_{m, F^{\opn{cl}}} \right).
\]
Let $\opn{Res}_{F^+/\mbb{Q}} \opn{U}(1)$ be the restriction of scalars of the unitary group associated with the one-dimensional Hermitian space over $F$ (with respect to $F/F^+$). Then we have a morphism of algebraic groups:
\begin{align*}
    \invs{N} \colon \opn{Res}_{F/\mbb{Q}} \mbb{G}_m &\to \opn{Res}_{F^+/\mbb{Q}} \opn{U}(1) \\
    z &\mapsto \bar{z}/z
\end{align*}
which is open and surjective on $\mbb{A}_f$-points. On the other hand, we have a morphism
\[
\nu \colon \mbf{H} \xrightarrow{\opn{det}} \mbf{R} \to \opn{Res}_{F^+/\mbb{Q}} \opn{U}(1)
\]
where the second map is given by sending a pair $(z_1, z_2)$ to $z_2/z_1$.

\begin{notation}
Let $\ide{N}$ be the smallest ideal of $\ordd_F$ such that $\nu(U) \subset \invs{N}((\widehat{\ordd}_{F^+} + \ide{N} \widehat{\ordd}_F)^{\times})$, where $U \subset \mbf{H}(\mbb{A}_f)$ is the level of $S_{\mbf{H}, \diamondsuit}(p)$.
\end{notation}

We introduce the following space of anticyclotomic characters:

\begin{definition}
Let $\Sigma(\ide{N})$ denote the set of algebraic Hecke characters $\chi \colon \mbb{A}_F^{\times} \to \mbb{C}^{\times}$ satisfying:
\begin{enumerate}
    \item $\chi$ is anticyclotomic, i.e. its restriction to $\mbb{A}_{F^+}^{\times}$ is trivial.
    \item The infinity type of $\chi$ is $(j, -j)$ for some tuple of integers $j = (j_{\tau})_{\tau \in \Psi}$, i.e. for any $z = (z_{\tau})_{\tau \in \Psi} \in \prod_{\tau \in \Psi} F_{\tau}$ one has
    \[
    \chi(z) = \prod_{\tau \in \Psi} z_{\tau}^{-j_{\tau}} \bar{z_{\tau}}^{j_{\tau}} . 
    \]
    \item The conductor of $\chi$ divides the ideal $\ide{N}$.
\end{enumerate}
\end{definition}

\begin{remark}
Let $\chi \in \Sigma(\ide{N})$. Then, since $\chi$ is anticyclotomic, the character $\chi$ desends to a unique character
\[
\chi' \colon \left(\opn{Res}_{F^+/\mbb{Q}} \opn{U}(1)\right)(\mbb{Q}) \backslash \left(\opn{Res}_{F^+/\mbb{Q}} \opn{U}(1)\right)(\mbb{A}) \to \mbb{C}^{\times} 
\]
satisfying $\chi = \chi' \circ \invs{N}$. We consider the character $\bar{\chi} \colon \mbf{R}(\mbb{Q}) \backslash \mbf{R}(\mbb{A}) \to \mbb{C}^{\times}$ defined as $\bar{\chi}(z_1, z_2) = \chi'(z_2/z_1)$.
\end{remark}

Any character $\chi \in \Sigma(\ide{N})$ has an associated $p$-adic algebraic Hecke character, denoted $\chi_p \colon \mbb{A}_F^{\times} \to \Qpb^{\times}$, by defining 
\[
\chi_p(x) = \iota_p(\chi_f(x)) \prod_{\tau \in \Psi} x_{\ide{p}_{\tau}}^{-j_{\tau}} x_{\overline{\ide{p}}_{\tau}}^{j_{\tau}} 
\]
where $\iota_p \colon \mbb{C} \cong \Qpb$ denotes the fixed isomorphism in \S \ref{NotationSection}, and $\ide{p}_{\tau}$ is the prime ideal corresponding to $\tau$ with respect to this isomorphism. We are interested in $p$-adically interpolating algebraic $p$-adic characters of the form
\[
\chi_{0,p} \prod_{\tau \in \Psi} \chi_{\tau, p}^{m_\tau} 
\]
where $\chi_0 \in \Sigma(\ide{N})$ is an anticyclotomic Dirichlet character, $\chi_{\tau} \in \Sigma(\ide{N})$ is a fixed anticyclotomic character of infinity type $(1_{\tau}, -1_{\tau})$ ($1_{\tau}$ is the tuple which is non-zero only in the $\tau$-component, where it is equal to $1$) and $m_{\tau}$ are integers. Furthermore, we want to interpret such a family as a coherent cohomology class. 

The strategy we will use for producing such a family follows three steps:
\begin{enumerate}
    \item We will first construct a family of cohomology classes interpolating these characters in the cohomology of a Shimura set associated with the group $\mbf{R}$. 
    \item Using the results in Appendix \ref{AppendixComparisonsFamilies}, we will pull back this construction to the Shimura variety $\invs{S}_{H, \diamondsuit}(p)$ via the morphism $\opn{det} \colon \mbf{H} \to \mbf{R}$.
    \item Finally, we will construct the family and describe the interpolation property.
\end{enumerate}

\subsection{Step 1: Classes for the Shimura set} \label{Step1AC}

Let $C \subset \mbf{R}(\mbb{A}_f)$ be a sufficiently small compact open subgroup, and let $\chi \in \Sigma(\ide{N})$ be an anticyclotomic character of infinity type $(j, -j)$ such that $\bar{\chi}$ is trivial on $C$. Let $\Delta \defeq S_{\mbf{R}, C}$ denote the associated Shimura set (over $F^{\mathrm{cl}}$), which satisfies 
\[
\Delta(\mbb{C}) = \mbf{R}(\mbb{Q}) \backslash \mbf{R}(\mbb{A}_f) / C .
\]
The goal of this section is to associate to $\bar{\chi}$ a class in the coherent cohomology of $\Delta$, and explain how one can raise this class to $p$-adic powers. 

Let $R_{dR} \to \Delta$ denote the standard principal $\mbf{R}_{F^{\mathrm{cl}}}$-bundle, which satisfies
\[
R_{dR}(\mbb{C}) = \mbf{R}(\mbb{Q}) \backslash \mbf{R}(\mbb{C}) \times \mbf{R}(\mbb{A}_f) / C
\]
(via the embedding $F^{\mathrm{cl}} \subset \mbb{C}$). This bundle has a trivialisation in the following way. Fix a set of representatives $\{s_1, \dots, s_r \} \subset \mbf{R}(\mbb{A}_f)$ for each point of $\Delta(\mbb{C})$, then we have an identification of torsors
\begin{equation} \label{TorsorIdentificationOverC}
\Delta(\mbb{C}) \times \mbf{R}(\mbb{C}) = R_{dR}(\mbb{C})
\end{equation}
by sending $([s_i], \gamma)$ to $[\gamma, s_i]$. One can show that, for any number field $\Phi/F^{\mathrm{cl}}$, this identification descends to an identification $\Delta_{\Phi} \times \mbf{R}_{\Phi} = R_{dR, \Phi}$.\footnote{One should think of such a choice of representatives as a choice of canonical model for $\Delta(\mbb{C})$. Of course, canonical models are unique up to unique isomorphism, but for this identification of torsors, it is helpful to fix such a choice.}

Recall that we have a fixed prime $\ide{p}$ of $F$ lying above $p$ (corresponding to the fixed embedding $\tau_0$). We fix a choice of prime $\ide{P}$ of $\Phi$ lying above $\ide{p}$, and by passing to completions, we obtain a finite extension $L \defeq \Phi_{\ide{P}}$ of $\mbb{Q}_p$. Let $\Delta^{\mathrm{an}}_L$ denote the adic space associated with $\Delta_L$, and let $\invs{R}_{\mathrm{HT}, L}^{\mathrm{an}} \to \Delta^{\mathrm{an}}_L$ denote the $R_L^{\mathrm{an}}$-torsor parameterising frames of (the pro-\'{e}tale sheaf) $\invs{V}_{\et} \otimes_{\hat{\mbb{Q}}_p} \hat{\ordd}_{\Delta^{\mathrm{an}}_L}$ (respecting certain tensors), where $R = \mbf{R}_{\mbb{Q}_p}$ and $\invs{V}_{\et}$ is the $p$-adic local system associated with a faithful representation $V$ of $R$ (see \cite[\S 2.3]{CS17}). 

Since $\mu_{\mbf{R}}$ is central in $\mbf{R}_{F^{\opn{cl}}}$, one has an isomorphism of torsors between the analytification of $R_{dR, L}$ and ${^\mu \invs{R}_{\mathrm{HT}, L}^{\mathrm{an}}}$ (the twist of $\invs{R}_{\mathrm{HT}, L}^{\mathrm{an}}$ along $\mu_{\mbf{R}}$).

\begin{notation}
Consider the open affinoid subgroup
\[
\invs{R}_{k, L} = \ordd_L^{\times} (1 + \invs{B}_k) \times \prod_{\tau \in \Psi} \left( \ordd_L^{\times} (1 + \invs{B}_k) \times \ordd_L^{\times} (1 + \invs{B}_k) \right) \subset R_L^{\opn{an}}
\]
where $\invs{B}_k$ is the ``closed disc'' (over $L$) in \S \ref{TubesInFlagVar}. We denote a general element of this subgroup by $(x_0, x_{1, \tau}, x_{2, \tau})_{\tau \in \Psi}$.
\end{notation}

\begin{corollary}
The above identification induces an identification
\[
\Delta^{\mathrm{an}}_L \times R^{\mathrm{an}}_L = {^\mu \invs{R}_{\mathrm{HT}, L}^{\mathrm{an}}} .
\]
\end{corollary}

It is evident from this identification that one obtains the following reduction of structure 
\[
\invs{R}_{\mathrm{HT}, L, k} \defeq \Delta^{\mathrm{an}}_L \times \invs{R}_{k, L} \hookrightarrow \Delta^{\mathrm{an}}_L \times R^{\mathrm{an}}_L = \invs{R}_{\mathrm{HT}, L}^{\mathrm{an}}
\]
for any $k \geq 1$. We can (and do) choose the set of representatives $\{s_1, \dots, s_r\}$ such that $s_i \in \mbf{R}(\mbb{A}_f^p)$.\footnote{This is possible because a finite Galois extension can be generated by Frobeniuses outside any finite set of primes.} Then we associate with $\bar{\chi}$ the global section 
\[
R_{dR}(\mbb{C}) \to \mbb{C}
\]
given by sending $[x, y] \mapsto \xi^{[j]}(x) \bar{\chi}(y)$, where $x \in \mbf{R}(\mbb{C})$ and $y \in \mbf{R}(\mbb{A}_f)$, and 
\begin{align*} 
\xi^{[j]} \colon \mbf{R}(\mbb{C}) \cong \mbb{C}^{\times} \times \prod_{\tau \in \Psi} \left(\mbb{C}^{\times} \times \mbb{C}^{\times} \right) &\to \mbb{C}^{\times} \\
 (x_0, x_{1,\tau}, x_{2, \tau})_{\tau \in \Psi} &\mapsto \prod_{\tau \in \Psi} \left( \frac{x_{2,\tau}}{x_{1, \tau}} \right)^{j_{\tau}} .
\end{align*}    
This global section is well-defined precisely because $\chi$ is an algebraic Hecke character of infinity-type $(j, -j)$, and transforms under the action of $\mbf{R}(\mbb{C})$ by the character $\xi^{[j]}$, so descends to a cohomology class
\[
[\chi]_B \in \opn{H}^0\left( \Delta(\mbb{C}), [\xi^{[j]}] \right) .
\]
Via the identification in (\ref{TorsorIdentificationOverC}), the class $[\chi]_B$ coincides with the product of the global section of $\Delta(\mbb{C})$ taking $s_i$ to $\bar{\chi}(s_i)$, and the global section $\mbf{R}(\mbb{C}) \xrightarrow{\xi^{[j]}} \mbb{C}^{\times} \subset \mbb{C}$. Since $\bar{\chi}(s_i)$ are elements of some number field, we can find a large enough $\Phi$ such that $[\chi]_B$ descends to a global section in $\opn{H}^0(\Delta_{\Phi}, [\xi^{[j]}])$. Via the rigid GAGA comparison, we therefore obtain a global section $[\chi]_{\mathrm{HT}} \in \opn{H}^0(\Delta_L^{\mathrm{an}}, [\xi^{[j]}])$ characterised by the global section
\begin{align*} 
\Delta^{\mathrm{an}}_L \times R^{\mathrm{an}}_L &\to \mbb{A}^{1, \mathrm{an}} \\
([s_i], t) &\mapsto \bar{\chi}(s_i) \xi^{[j], \mathrm{an}}(t)
\end{align*}
where we are viewing $\bar{\chi}(s_i)$ as an element of $L^{\times}$ via the identification $\mbb{C} \cong \Qpb$. 

\begin{lemma} \label{CharacterDescription}
For any integer $k \geq 1$, the global section $[\chi]_{\mathrm{HT}}$ is described by the morphism
\begin{align*} 
\Delta^{\mathrm{an}}_L \times \invs{R}_{k, L} &\to \mbb{A}^{1, \mathrm{an}} \\
([s_i], (x_0, x_{1,\tau}, x_{2, \tau})_{\tau \in \Psi}) &\mapsto \bar{\chi}(s_i) \prod_{\tau \in \Psi} \left( \frac{x_{2,\tau}}{x_{1, \tau}}\right)^{j_{\tau}} 
\end{align*}
which is valued in $\ordd_L^{\times} (1 + \invs{B}_k )$.
\end{lemma}
\begin{proof}
This follows immediately from the fact that $\bar{\chi}(s_i) \in \ordd_L^{\times}$ (under the identification $\mbb{C} \cong \Qpb$). Indeed, because the representatives $s_i$ have been chosen to have no component at $p$, $\bar{\chi}(s_i)$ is in the image of the (continuous) Galois character $\Gal(F^{\mathrm{ab}}/F) \to L^{\times}$ associated with $\chi_p$ (via class field theory). But Galois groups are compact, so this is valued in $\ordd_L^{\times}$. 
\end{proof}

The description in Lemma \ref{CharacterDescription} allows us to raise this cohomology class to $p$-adic powers, in the following way. Let $(A, A^+)$ be a Tate algebra over $(L, \ordd_L)$ and let $\beta \colon \ordd_L^{\times} \to (A^+)^{\times}$ be a $k$-analytic character, i.e. it extends to a pairing
\[
\ordd_L^{\times}(1+\invs{B}_k) \times_{\opn{Spa}(L, \ordd_L)} \opn{Spa}(A, A^+) \to \mbb{G}^{\mathrm{an}}_m  .
\]
Then via the torsor $\invs{R}_{\mathrm{HT}, L, k}$, one obtains an $A$-Banach sheaf $[\beta \circ \xi^{[j]}]$ and a cohomology class
\[
[\chi]_{\mathrm{HT}}^{\beta} \in \opn{H}^0(\Delta^{\mathrm{an}}_L, [\beta \circ \xi^{[j]}])
\]
described by the morphism
\begin{align*}
    \Delta^{\mathrm{an}}_L \times \invs{R}_{k, L} &\to \mbb{A}^{1, \mathrm{an}} \times \opn{Spa}(A, A^+) \\
([s_i], (x_0, x_{1,\tau}, x_{2, \tau})_{\tau \in \Psi}) &\mapsto \beta \left( \bar{\chi}(s_i)  \prod_{\tau \in \Psi} \left( \frac{x_{2,\tau}}{x_{1, \tau}}\right)^{j_{\tau}} \right)
\end{align*}
which is well-defined by Lemma \ref{CharacterDescription}. This description is independent of the radius of analyticity $k$. 

\begin{remark} \label{ComplexComparisonRemark}
If we take $(A, A^+) = (L, \ordd_L)$ and $\beta(-) = (-)^k$ for some integer $k$, then $[\chi]_{\mathrm{HT}}^{\beta}$ is equal to the $k$-fold cup product of $[\chi]_{\mathrm{HT}}$ (which makes sense for negative integers because $[\chi]_{\mathrm{HT}}$ is an invertible section). In particular, under the rigid GAGA comparison (and the identification $\mbb{C} \cong \Qpb$)
\[
\opn{H}^0\left( \Delta^{\mathrm{an}}_L, [\beta \circ \xi^{[j]}] \right) = \opn{H}^0\left( \Delta^{\mathrm{an}}_{\mbb{Q}_p}, [\beta \circ \xi^{[j]}] \right) \otimes_{\mbb{Q}_p} L \hookrightarrow \opn{H}^0\left( \Delta(\mbb{C}), [\beta \circ \xi^{[j]}] \right) 
\]
the class $[\chi]_{\mathrm{HT}}^{\beta}$ is mapped to $[\chi^k]_B$.
\end{remark}

\subsection{Step 2: Pullback to the Shimura variety for \texorpdfstring{$\mbf{H}$}{H}}

Recall that we have a morphism $(\mbf{H}, X_{\mbf{H}}) \to (\mbf{R}, X_{\mbf{R}})$ of Shimura data induced from the homomorphism $\opn{det} \colon \mbf{H} \twoheadrightarrow \mbf{R}$. Let $U = U^p K^H_{\diamondsuit}(p)$ and let $C = \nu(U)$. By shrinking $U^p$ is necessary, we may assume that $C$ is neat. We therefore obtain a morphism
\[
S_{\mbf{H}, \diamondsuit}(p) \to S_{\mbf{R}, C} := \Delta 
\]
which we will also denote by $\opn{det}$. The fibres of this morphism (after base-changing to a sufficiently large field extension) are disjoint unions of connected components of $S_{\mbf{H}, \diamondsuit}(p)$.

Let $H_{dR} \to S_{\mbf{H}, \diamondsuit}(p)$ denote the standard principle $\mbf{H}_{F^{\mathrm{cl}}}$-bundle as in \cite[\S III.3]{MilneCanonicalMixed}, which satisfies 
\[
H_{dR}(\mbb{C}) = \mbf{H}(\mbb{Q}) \backslash X_{\mbf{H}} \times \mbf{H}(\mbb{C}) \times \mbf{H}(\mbb{A}_f) / K .
\]
One has a natural morphism $H_{dR}(\mbb{C}) \to R_{dR}(\mbb{C})$ induced from the morphism $\opn{det}$ and, as explained in \S III.4 of \emph{op.cit.}, this descends to a morphism on the canonical models\footnote{Since $(\mbf{H}, X_{\mbf{H}})$ does not satisfy axiom (SD3) in \cite[Definition B.16]{ACES}, one has to use the additional property that this Shimura--Deligne datum embeds into a Siegel datum to ensure the existence of a canonical model for $H_{dR}$.} of these standard principle bundles; i.e. we obtain a morphism (of principle bundles) $H_{dR} \to R_{dR}$. One can check on complex points that this induces an isomorphism $H_{dR} \times^{\mbf{H}_{F^{\mathrm{cl}}}} \mbf{R}_{F^{\mathrm{cl}}} \cong \opn{det}^* R_{dR}$, where the pushout is via the morphism $\opn{det} \colon \mbf{H}_{F^{\mathrm{cl}}} \to \mbf{R}_{F^{\mathrm{cl}}}$.

On the other hand, the bundle $H_{dR}$ can be expressed as the pushout $P_{H, dR} \times^{\mbf{P}_H} \mbf{H}_{F^{\mathrm{cl}}}$, and since the morphism $\nu \colon \mbf{P}_H \to \mbf{R}_{F^{\mathrm{cl}}}$ factors through the projection $\mbf{P}_{H} \twoheadrightarrow \mbf{M}_H$, one obtains an isomorphism 
\[
M_{H, dR} \times^{\mbf{M}_H} \mbf{R}_{F^{\mathrm{cl}}} \cong H_{dR} \times^{\mbf{H}_{F^{\mathrm{cl}}}} \mbf{R}_{F^{\mathrm{cl}}} \cong \opn{det}^* R_{dR} .
\]
Passing to the associated adic spaces and using the de Rham--$p$-adic comparison, one obtains an isomorphism (of $R^{\mathrm{an}}$-torsors)
\begin{equation} \label{AnalyticIso}
{^\mu \invs{M}_{H, \mathrm{HT}}^{\mathrm{an}}} \times^{M_{H}^{\mathrm{an}}} R^{\mathrm{an}} \cong \opn{det}^* \left( {^\mu \invs{R}_{\mathrm{HT}}^{\mathrm{an}}} \right) .
\end{equation}
It will be helpful to reinterpret this isomorphism in terms of flag varieties. We have a commutative diagram:
\[
\begin{tikzcd}
{\invs{S}_{H, U^p}} \arrow[d] \arrow[r, "{\pi_{H, \mathrm{HT}}}"] & \mathtt{FL}^H \arrow[d] \\
{\invs{S}_{R, C^p}} \arrow[r, "{\pi_{R,\mathrm{HT}}}"]            & \mathtt{FL}^R          
\end{tikzcd}
\]
where the vertical arrows are induced from the homomorphism $\opn{det}$.

Let $\mathtt{R}^{\mathrm{an}}$ and $\mathtt{M}^{H,\mathrm{an}}$ denote the torsors $R^{\mathrm{an}} \to \mathtt{FL}^R$ and $H^{\mathrm{an}}/N_H^{\mathrm{an}} \to \mathtt{FL}^H$ respectively (where both structural maps are given by $x \mapsto x^{-1}$ to ensure that they are right torsors). Note that the torsor $\mathtt{R}^{\mathrm{an}}$ is trivial, so $\opn{det}^* \mathtt{R}^{\mathrm{an}}$ is identified with $\mathtt{FL}^G \times R^{\mathrm{an}}$ and we have a canonical isomorphism
\[
\mathtt{M}^{H, \mathrm{an}} \times^{M_H^{\mathrm{an}}} R^{\mathrm{an}} \cong \opn{det}^* \mathtt{R}^{\mathrm{an}} .
\]
Since pull-back commutes with colimits (so in particular pushouts) and this is compatible with the $K^H_{\diamondsuit}(p)$-equivariant structure, this induces an isomorphism
\[
    \invs{M}_{H, \mathrm{HT}}^{\mathrm{an}} \times^{M_{H}^{\mathrm{an}}} R^{\mathrm{an}} = \pi_{H,\mathrm{HT}}^*( \mathtt{M}^{H, \mathrm{an}} \times^{M_H^{\mathrm{an}}} R^{\mathrm{an}} )/K^H_{\diamondsuit}(p) \cong \pi_{H, \mathrm{HT}}^*(\opn{det}^* \mathtt{R}^{\mathrm{an}})/K^H_{\diamondsuit}(p) = \opn{det}^* \invs{R}_{\mathrm{HT}}^{\mathrm{an}} .
\]
We can twist this isomorphism along $\mu \colon \mbb{Z}_p^{\times} \to M_H^{\opn{an}} \xrightarrow{\opn{det}} R^{\opn{an}}$ (induced from $\mu_{\mbf{R}} = \opn{det} \circ \mu_{\mbf{H}}$) to obtain an isomorphism
\begin{equation} \label{FlagVarIso}
{^\mu \invs{M}_{H, \mathrm{HT}}^{\mathrm{an}}} \times^{M_{H}^{\mathrm{an}}} R^{\mathrm{an}} \cong \opn{det}^* \left( {^\mu \invs{R}_{\mathrm{HT}}^{\mathrm{an}}} \right) .
\end{equation}

\begin{proposition}
The isomorphisms (\ref{AnalyticIso}) and (\ref{FlagVarIso}) coincide.
\end{proposition}
\begin{proof}
With notation as in Appendix \ref{AppendixComparisonsFamilies}, the isomorphism (\ref{AnalyticIso}) (resp. (\ref{FlagVarIso}))  is induced from the natural transformation $\eta_{\opn{dR}}$ (resp. $\eta_{\et}$). The result now follows from Corollary \ref{FunctOfMTorsorsCor}.
\end{proof}

We obtain the following corollary:

\begin{corollary} \label{ReductionCommutativeDiagram}
Over $\invs{U}^H_k(p)_L$ one has a commutative diagram
\[
\begin{tikzcd}
{{^\mu \invs{M}_{H, \mathrm{HT}, L}^{\mathrm{an}}} \times^{M_{H,L}^{\mathrm{an}}} R^{\mathrm{an}}_L} \arrow[r, "(\ref{AnalyticIso})"] & \opn{det}^* \left({^\mu \invs{R}_{\mathrm{HT}, L}^{\mathrm{an}}}\right)  \\
{{^\mu \invs{M}_{H, \mathrm{HT}, k, 1, L}} \times^{\invs{M}_{H, k, 1, L}^{\clubsuit}} \invs{R}_{k, L}} \arrow[r, "\sim"] \arrow[u]           & {\opn{det}^* \left( {^ \mu \invs{R}_{\mathrm{HT}, k, L}} \right)} \arrow[u]
\end{tikzcd}
\]
for any finite extension $L/\mbb{Q}_p$, where the left-hand map is induced from the reduction of structure in \S \ref{FurtherReductionOfStructureSection}.
\end{corollary}
\begin{proof}
To simplify notation, we will establish the case $L=\mbb{Q}_p$ only, as the general case follows the exact same argument.

Note that the left-hand vertical map is induced from the morphism $\invs{M}_{H, \mathrm{HT}, k, k, 1} \to \invs{M}_{H, \mathrm{HT}}^{\mathrm{an}}$ and pushing out along $\invs{M}^{\clubsuit}_{H, k, k, 1} \to \invs{R}_{k}$ factors through the affinoid group $\invs{M}_{H, k, 1}^{\clubsuit}$, so the left hand vertical map does indeed make sense.

Using the fact that the morphism (\ref{AnalyticIso}) coincides with (\ref{FlagVarIso}) and untwisting along $\mu \colon \mbb{Z}_p^{\times} \to \invs{M}_{H, k, 1}^{\clubsuit} \xrightarrow{\opn{det}} \invs{R}_k$, we can work on the level of flag varieties. In this setting we have a commutative diagram (because the morphism $\invs{M}^{\diamondsuit}_{H, k, k, 1} \to \invs{R}_k$ extends to a morphism $\mathtt{M}^{H}_{k, k, 1} \to \invs{R}_k$)
\[
\begin{tikzcd}
{\mathtt{M}^{H, \mathrm{an}}|_{\mathtt{U}^H_k} \times^{M_H^{\mathrm{an}}} R^{\mathrm{an}}} \arrow[r, "\sim"]  & \mathtt{U}^H_k \times R^{\mathrm{an}}      \\
{\mathtt{M}^H_{k, k, 1} \times^{\invs{M}^{\diamondsuit}_{H, k, k, 1}} \invs{R}_k} \arrow[r, "\sim"] \arrow[u] & \mathtt{U}^H_k \times \invs{R}_k \arrow[u]
\end{tikzcd}
\]
which gives the desired result. 
\end{proof}

\subsection{Step 3: Construction of the family}

Fix a collection $\{ \chi_{\tau} : \tau \in \Psi \} \subset \Sigma(\ide{N})$ of anticyclotomic characters, where $\chi_{\tau}$ has infinity type $(1_{\tau}, -1_{\tau})$ and let $\chi_0 \in \Sigma(\ide{N})$ be a fixed anticyclotomic Dirichlet character. Let $L'/\mbb{Q}_p$ be a sufficiently large finite extension containing the fields of definition of $\chi_{\tau, p}$, and let $L/L'$ be finite extension containing the field of definition of $\chi_{0,p}$.

\begin{theorem} \label{FamilyOfACCharMainThm}
Let $(A, A^+)$ be a Tate algebra over $(L, \ordd_L)$ and let $(\beta_{\tau})_{\tau \in \Psi}$ be a collection of locally analytic characters $\ordd_{L'}^{\times} \to (A^+)^{\times}$. Let $\xi^{[\beta]} \colon \invs{R}_{k, L'} \to \mbb{G}_m^{\mathrm{an}}$ denote the character given by sending $(x_0, x_{1, \tau}, x_{2, \tau})_{\tau \in \Psi}$ to $\prod_{\tau} \beta_{\tau}(x_{2,\tau}/x_{1, \tau})$, for any sufficiently large $k$. Then there exists a class
\[
\underline{\chi} \in \opn{H}^0_{\opn{id}, \mathrm{an}}( \xi^{[\beta]} \circ \opn{det} )^{(+, \dagger)} \defeq \varinjlim_m \opn{H}^0\left( \invs{Z}^H_m(p), [\xi^{[\beta]} \circ \opn{det}] \right)
\]
such that 
\begin{enumerate}
    \item If $(A, A^+) = (L, \ordd_L)$ and $\beta_{\tau}$ are integers, then $\underline{\chi}$ extends to a class in $\opn{H}^0\left(\invs{S}_{H, \diamondsuit}(p)_L, [\xi^{[\beta]} \circ \opn{det} ] \right)$ whose image under the map (induced from rigid GAGA and the identification $\mbb{C} \cong \Qpb$)
\[
\opn{H}^0\left(\invs{S}_{H, \diamondsuit}(p)_L, [\xi^{[\beta]} \circ \opn{det} ] \right) = \opn{H}^0\left(\invs{S}_{H, \diamondsuit}(p), [\xi^{[\beta]} \circ \opn{det} ] \right) \otimes_{\mbb{Q}_p} L \hookrightarrow \opn{H}^0\left(\invs{S}_{H, \diamondsuit}(p)(\mbb{C}), [\xi^{[\beta]} \circ \opn{det} ] \right)
\]
is equal to $\opn{det}^* ([\chi_0]_B\cdot \prod_{\tau \in \Psi} [\chi_{\tau}^{\beta_{\tau}}]_B )$. In other words, for classical weights this family specialises to the cohomology class representing the automorphic form
\begin{align*} 
\mbf{H}(\mbb{Q}) \backslash \mbf{H}(\mbb{A}) &\to \mbb{C} \\
(h_1, h_2) &\mapsto \bar{\chi}_0(\opn{det}(h_1, h_2)) \cdot \prod_{\tau \in \Psi} \bar{\chi}_{\tau}(\opn{det}(h_1, h_2))^{\beta_{\tau}} .
\end{align*}
\item For varying $(A, A^+)$, the constructions of $\underline{\chi}$ are compatible.
\end{enumerate}
\end{theorem}
\begin{proof}
Recall the definitions of $[\chi_0]_{\mathrm{HT}}$ and $[\chi_{\tau}]_{\mathrm{HT}}^{\beta_{\tau}}$ from \S \ref{Step1AC} (where we view $[\chi_{\tau}]_{\mathrm{HT}}^{\beta_{\tau}}$ as a class defined over $L$). We define 
\[
\underline{\chi} = \opn{det}^* [\chi_0]_{\mathrm{HT}} \cdot \prod_{\tau \in \Psi} \opn{det}^* [\chi_{\tau}]_{\mathrm{HT}}^{\beta_{\tau}} .
\]
The interpolation property follows from Corollary \ref{ReductionCommutativeDiagram} and Remark \ref{ComplexComparisonRemark}, and it is clear from the definition of $[\cdots]_{\mathrm{HT}}$ that this construction is compatible under base-change.
\end{proof}

%%%%%%%%%%%%%%%%%%%%%%%%%%%%%%%%%%%%%%%%%%%%%%%%%%%%%%%%%%%%%%%%%%%%%%%%%%%%%%%%%%%%%%%%%
%%%%%%    CONSTRUCTION OF THE P-ADIC L-FUNCTION    %%%%%%%%%%%%%%%%%%%%%%%%%%%%%%%%%%%%%%
%%%%%%%%%%%%%%%%%%%%%%%%%%%%%%%%%%%%%%%%%%%%%%%%%%%%%%%%%%%%%%%%%%%%%%%%%%%%%%%%%%%%%%%%%

\section{Construction of the \texorpdfstring{$p$-adic $L$-function}{p-adic L-function}} \label{ConstructionOfPadicLSection}

In this final section, we construct the $p$-adic $L$-function associated with a family of cohomology classes $\underline{\eta}$ and a family $\underline{\chi}$ of anticyclotomic characters. We will end by discussing its relation to unitary Friedberg--Jacquet periods.

\subsection{Definition of the \texorpdfstring{$p$-adic $L$-function}{p-adic L-function}}

Let $\pi$ be a cuspidal automorphic representation of $\mbf{G}(\mbb{A})$ satisfying Assumptions \ref{HCassumptionOnPi}, \ref{SSSassumption} and \ref{CMramificationAssumption}. Then the construction in \S \ref{ExistenceOfFamilies} implies that there exists a unique family $\theta_{\underline{\pi}}$ and family of cohomology classes $\underline{\eta} \in S^{n-1}(\underline{\pi})$ passing through $\pi$, defined over a sufficiently small affinoid $U = \opn{Spa}(A, A^+) \subset \invs{W}_{G, L}$. 

For the family of anticyclotomic characters, we make the following assumption:

\begin{assumption}
The class number of $F$ is not divisible by $p$.
\end{assumption}

By this assumption, for every $\tau \in \Psi$, we can fix an anticyclotomic character $\chi_{\tau} \in \Sigma(\mathfrak{N})$ of infinity type $(1_{\tau}, -1_{\tau})$, such that associated $p$-adic Hecke character is valued in $\mbb{Q}_p^{\times}$ (see the discussion in \cite[\S 4.2]{Collins}, for example). Fix an anticyclotomic Dirichlet character $\chi_0 \in \Sigma(\ide{N})$, and let $V = \opn{Spa}(B, B^+) \subset \invs{W}_{H, L}$ be an open affinoid subspace with universal character $\lambda_B \colon \left(\mbb{Z}_p^{\times}\right)^{[F^+:\mbb{Q}] - 1} \to (B^+)^{\times}$. We can naturally view $\lambda_A$ and $\beta \defeq \lambda_B$ as characters valued in $A \hatot B$. Then the results in \S \ref{FamiliesOfACcharsSection} imply that there exists a family $\underline{\chi} \in \opn{H}^0_{\mathrm{id}, \mathrm{an}}(\sigma_n^{[\beta]}(\lambda_A)^{\vee})^{(+, \dagger)}$ which interpolates (the coherent cohomology classes associated with) the anticyclotomic characters
\[
\chi_{(\lambda, j)} \defeq \chi_0 \cdot \chi_{\tau_0}^{-(c_{n, \tau_0} + 1)} \cdot \prod_{\tau \neq \tau_0} \chi_{\tau}^{-j_{\tau}} 
\]
where $(\lambda, j) \in X^*(T/T_0)^+ \times X^*(S)^+ \cap U \times V$ with $\lambda = (0; c_{1, \tau}, \dots, c_{2n, \tau})_{\tau \in \Psi}$ and $j = (j_{\tau})_{\tau \neq \tau_0}$ satisfying $0 \leq j_{\tau} \leq c_{n, \tau}$.

\begin{definition}
With the set-up as above, we define
\[
\mathscr{L}_p(\underline{\eta}, \underline{\chi}) \defeq \langle\langle \underline{\eta}, \underline{\chi} \rangle\rangle_{\mathrm{an}}^{-} \quad \in \ordd(U \times V)
\]
where the right-hand side is as in \S \ref{LocallyAnalyticCohomologySubSec}.
\end{definition}

\begin{remark}
Since the pairing $\langle\langle \cdot, \cdot \rangle\rangle_{\mathrm{an}}^{-}$ is compatible with change of coefficients, the $p$-adic analytic functions $\mathscr{L}_p(\underline{\eta}, \underline{\chi})$ glue as $V$ varies. Therefore, we can (and do) view 
\[
\mathscr{L}_p(\underline{\eta}, \underline{\chi} ) \in \ordd(U \times \invs{W}_{H,L}) 
\]
which makes sense because the families $\underline{\chi}$ glue for varying $V$, by Theorem \ref{FamilyOfACCharMainThm}(2) (note that we can choose an open affinoid cover of $\invs{W}_H$ such that the universal characters for each open are locally analytic -- see \cite[Lemma 4.1.5]{lz-coleman}). 
\end{remark}

\subsection{The interpolation property}

Keeping with the same set-up as in the previous section, we introduce the following ``region of interpolation'':

\begin{definition}
Let $\Sigma^{\mathrm{int}}$ denote the subset of $X^*(T/T_0)^+ \times X^*(S)^+ \cap (U \times \invs{W}_H)(L)$ of all pairs $(\lambda, j)$ with $\lambda = (0; c_{1, \tau}, \dots, c_{2n, \tau})_{\tau \in \Psi}$ and $j = (j_{\tau})_{\tau \neq \tau_0}$ satisfying $0 \leq j_{\tau} \leq c_{n, \tau}$.
\end{definition}

For $(\lambda, j) \in \Sigma^{\mathrm{int}}$ let 
\[
\eta_{\lambda} \in \opn{H}^{n-1}_{w_n, \mathrm{an}}\left( \kappa_n(\lambda) \right)^{-, \mathrm{ss}} \cong \opn{H}^{n-1}\left( \invs{S}_{G, \mathrm{Iw}}(p), [V_{\kappa_n(\lambda)}] \right)^{-, \mathrm{ss}}
\]
denote the specialisation of $\underline{\eta}$ at $(\lambda, j)$, which we can view as an element of $\opn{H}^{n-1}\left( S_{\mbf{G}, \mathrm{Iw}}(p)(\mbb{C}), [V_{\kappa_n(\lambda)}] \right)$ via rigid GAGA and the identification $\iota_p \colon \mbb{C} \cong \Qpb$. Let $\mathscr{L}_p(\eta_{\lambda}, \chi_{(\lambda, j)} )$ denote the specialisation of $\mathscr{L}_p(\underline{\eta}, \underline{\chi})$ under the map $\ordd(U \times \invs{W}_H) \to L$ induced from $(\lambda, j)$.     

We obtain the following interpolation property for $\mathscr{L}_p(\underline{\eta}, \underline{\chi} )$.

\begin{proposition} \label{PadicLequalsComplex}
After possibly shrinking $U$ around $\lambda_{\pi}$, for any $(\lambda, j) \in \Sigma^{\mathrm{int}}$ one has
\[
\iota_p^{-1} \mathscr{L}_p(\eta_{\lambda}, \chi_{(\lambda, j)} ) = \langle \eta_{\lambda}, \nu^*[\chi_{(\lambda, j)}]_B \rangle_{\mathrm{alg}}
\]
where $\iota_p \colon \mbb{C} \cong \Qpb$ denotes the fixed isomorphism, and the pairing in the right-hand side has been base-changed to $\mbb{C}$ (via the embedding $F^{\mathrm{cl}} \hookrightarrow \mbb{C}$).
\end{proposition}
\begin{proof}
If we let 
\[
\nu^*[\chi_{(\lambda, j)}]_{\mathrm{HT}} \in \opn{H}^0_{\mathrm{id}, \mathrm{an}}(\sigma^{[j]}(\lambda)^{\vee} )^{(+, \dagger)} = \opn{H}^0_{\mathrm{id}}(\sigma^{[j]}(\lambda)^{\vee} )^{(+, \dagger)}
\]
denote the specialisation of $\underline{\chi}$, then the results in \S \ref{FamiliesOfACcharsSection} imply that $\nu^*[\chi_{(\lambda, j)}]_{\mathrm{HT}}$ is in the image of the restriction map 
\[
\opn{H}^0\left( \invs{S}_{H, \diamondsuit}(p), [\sigma_n^{[j]}(\lambda)]^{\vee} \right) \to \opn{H}^0_{\mathrm{id}}(\sigma^{[j]}(\lambda)^{\vee} )^{(+, \dagger)}
\]
and its image under the rigid GAGA comparison is equal to $\nu^*[\chi_{(\lambda, j)}]_B$. The result then follows from Corollary \ref{DistAnIsAnCorollary}, Theorem \ref{RestrictionToStrataThm} and Proposition \ref{GAGApairingProp}.
\end{proof}

\begin{remark}
The equality in Proposition \ref{PadicLequalsComplex} depends on a choice of isomorphism $V_{\kappa_n(\lambda)^*}^* \cong V_{\kappa_n(\lambda)}$ over $F^{\mathrm{cl}}$.
\end{remark}

Let $[\mbf{H}] = \mbf{H}(\mbb{Q}) A_{\mbf{G}, \mbf{H}}(\mbb{A}) \backslash \mbf{H}(\mbb{A})$, where $A_{\mbf{G}}$ denotes the maximal split subtorus of the centre of $\mbf{G}$ and $A_{\mbf{G}, \mbf{H}} = A_{\mbf{G}} \cap \mbf{H}$ (which in fact equals $A_{\mbf{G}}$). By choosing a Haar measure for $\mbf{H}(\mbb{Q}) A_{\mbf{G}, \mbf{H}}(\mbb{A})$ and using a fixed Haar measure for $\mbf{H}(\mbb{A})$, one obtains a measure on the quotient $[\mbf{H}]$ which we will denote by $\bar{d}h$. We also let $[\mbf{H}]' = \mbf{H}(\mbb{Q}) A_{\mbf{G}, \mbf{H}}(\mbb{R})^{\circ} \backslash \mbf{H}(\mbb{A})$ and, similar to above, we have an induced measure $\overline{d}'h$. We choose these measures so they are compatible under the quotient map $[\mbf{H}]' \to [\mbf{H}]$. We also assume that the volume of $U_{\infty}^{\circ} U$ with respect to the Haar measure on $\mbf{H}(\mbb{A})$ is contained in $(F^{\mathrm{cl}})^{\times}$, where $U^{\circ}_{\infty}$ is the maximal compact subgroup of $U_{\infty} = K_{\infty} \cap \mbf{H}(\mbb{R})$.

\begin{corollary} \label{RelationToAutoPeriods}
Let $(\lambda, j) \in \Sigma^{\mathrm{int}}$ and $\sigma$ be the cuspidal automorphic representation of $\mbf{G}(\mbb{A})$ associated with $\eta_{\lambda}$ (see \S \ref{ExistenceOfFamilies}). Then there exists $G \in \sigma$ such that
\begin{equation} \label{InterpolationFormula}
\iota_p^{-1} \mathscr{L}_p(\eta_{\lambda}, \chi_{(\lambda, j)} ) \sim_{F^{\mathrm{cl},\times}} (2 \pi i)^{-(n-1)} \int_{[\mbf{H}]'} G(h) \cdot \chi_{(\lambda, j)} (\nu(h)) \bar{d}'h
\end{equation}
where $\sim_{F^{\mathrm{cl},\times}}$ means up to a non-zero constant in $F^{\mathrm{cl},\times}$ which only depends on $\lambda$ and the choice of Haar measures as above.

Furthermore, if the central character of $\pi$ restricted to $A_{\mbf{G}, \mbf{H}}(\mbb{A})$ is trivial, then we have the relation
\[
\iota_p^{-1} \mathscr{L}_p(\eta_{\lambda}, \chi_{(\lambda, j)} ) \sim_{F^{\mathrm{cl},\times}} (2 \pi i)^{-(n-1)} \int_{[\mbf{H}]} G(h) \cdot \chi_{(\lambda, j)} (\nu(h)) \bar{d}h
\]
after possibly shrinking $U$ around $\lambda_{\pi}$.
\end{corollary}
\begin{proof}
By Proposition \ref{PadicLequalsComplex}, it is equivalent to showing that $\langle \eta_{\lambda}, \nu^*[\chi_{(\lambda, j)}]_B \rangle_{\mathrm{alg}}$ equals the right-hand side of (\ref{InterpolationFormula}). We will freely use the notation from the proof of Corollary \ref{EtalenessOfEigenvarietyCor}. We first note that we have an morphism 
\[
\opn{Hom}_{K_{\infty}}(\nu_{n-1}, \sigma_{\infty}) \to \opn{Hom}_{K_{\infty}}(\bigwedge^{n-1} \left(\ide{p}/\ide{m}\right), \sigma_{\infty} \otimes V_{\kappa_n(\lambda)} )
\]
where notation is as in \S \ref{DiscreteSeriesRepsSection}, given by precomposing with the map of $M_{\mbf{G}}$-representations
\begin{equation} \label{BlattnerProjection}
\bigwedge^{n-1} \left(\ide{p}/\ide{m}\right) \otimes V_{\kappa_n(\lambda)^*} \to \nu_{n-1}
\end{equation}
(which is uniquely determined up to $\mbb{C}^{\times}$) and using a fixed isomorphism 
\begin{equation} \label{DualIso}
V_{\kappa_n(\lambda)^*}^* \cong V_{\kappa_n(\lambda)}.
\end{equation}
This induces an isomorphism $\opn{Hom}_{K_{\infty}}(\nu_{n-1}, \sigma_{\infty}) \cong \opn{H}^{n-1}_{(\ide{p}, K_{\infty})}(\sigma_{\infty} \otimes V_{\kappa_n(\lambda)})$, and hence we obtain an injective map
\[
\opn{Hom}_{K_{\infty}}(\nu_{n-1}, \sigma_{\infty}) \otimes \sigma_f^{K^pK^G_{\mathrm{Iw}}(p)} \hookrightarrow \opn{H}^{n-1}\left(S_{\mbf{G}, \mathrm{Iw}}(p)(\mbb{C}), [V_{\kappa_n(\lambda)}] \right)
\]
whose image is identified with the localisation of the right-hand side at the kernel of the specialisation of $\theta_{\underline{\pi}}$ at $\lambda$.

The representation $\bigwedge^{n-1} \left(\ide{p}/\ide{m}\right)$ is definable over $F^{\mathrm{cl}}$ so we choose the map (\ref{BlattnerProjection}) to be defined over $F^{\mathrm{cl}}$. We also choose the same isomorphism (\ref{DualIso}) as in Proposition \ref{PadicLequalsComplex}, which is defined over $F^{\mathrm{cl}}$. Recall from Proposition \ref{ClassicalBranchingProp} that we have a (unique up to scaling) vector $v_{\kappa_n(\lambda)}^{[j]} \in V_{\kappa_n(\lambda)^*}$ on which $M_{\mbf{H}}$ acts through the character $\sigma_n^{[j]}(\lambda)^{-1}$. Let $z$ be the image of $w \otimes v_{\kappa_n(\lambda)}^{[j]}$ under the map (\ref{BlattnerProjection}), where $w$ is a choice of highest weight vector of $\bigwedge^{n-1} \left(\ide{p}/\ide{m}\right)$ defined over $F^{\mathrm{cl}}$. This vector $z$ is non-zero because $\sigma_n^{[j]}(\lambda)^{\vee}$ appears as a direct factor with multiplicity one in both the codomain and domain of (\ref{BlattnerProjection}). 

Via the above injective map, the class $\eta_{\lambda}$ corresponds to a homomorphism $G_{\eta_{\lambda}} \otimes \varphi_f$, where $\varphi_f \in \sigma^{K^pK^G_{\mathrm{Iw}}(p)}$. We take $G$ to be $G \defeq \widehat{\gamma} \cdot \left( G_{\eta_{\lambda}}(z) \otimes \varphi_f \right) \in \sigma$ (where $\widehat{\gamma}$ is viewed as an element of $\mbf{G}(\mbb{Q}_p) \subset \mbf{G}(\mbb{A})$). If we let $\ide{p}_H$ (resp. $\ide{m}_H$) denote the Lie algebra of the opposite of $P_{\mbf{H}}$ (resp. $M_{\mbf{H}}$), then $\bigwedge^{n-1}\ide{p}_H/\ide{m}_H$ is identified with the line spanned by the vector $w$. By \cite{Su19}, we have an isomorphism
\[
\opn{H}^{n-1}\left( S_{\mbf{H}, \diamondsuit}(p)(\mbb{C}), [\sigma_n^{[j]}(\lambda)] \right) \cong \opn{H}^{n-1}_{(\ide{p}_H, U_{\infty})}( C^{\infty}([\mbf{H}]'/U)^{U_{\infty}\mathrm{-fin}} \otimes \sigma_n^{[j]}(\lambda) )
\]
where $U_{\infty} = K_{\infty} \cap \mbf{H}(\mbb{R})$ and $U \subset \mbf{H}(\mbb{A}_f)$ is the level of the Shimura variety $S_{\mbf{H}, \diamondsuit}(p)$. Under this identification, the class $\widehat{\iota}^* \eta_{\lambda}$ is represented by the homomorphism
\begin{align*}
    \bigwedge^{n-1} \ide{p}_H/\ide{m}_H &\to C^{\infty}([\mbf{H}]'/U)^{U_{\infty}\mathrm{-fin}} \otimes \sigma_n^{[j]}(\lambda) \\
     w &\mapsto G|_{\mbf{H}} .
\end{align*}
The result now follows from \cite[Proposition 3.8]{HarrisPartial}.

For the last part, note that the central character of $\pi$ restricted to $A_{\mbf{G}}(\mbb{A})$ is necessarily a Dirichlet character (because $\pi$ contributes to the coherent cohomology of $S_{\mbf{G}, \mathrm{Iw}}(p)$ and the centre acts trivially on $V_{\kappa_n}$) and is therefore determined by the image of Hecke operators $[K^S a K^S]$ under the map $\theta_{\pi}$, for $a \in A_{\mbf{G}}(\mbb{A}_f^S)$. The image of these operators under $\theta_{\pi}$ form a discrete subgroup, so we can shrink $U$ if necessary so that the images of these operators under $\theta_{\underline{\pi}}$ are constant (note that one normally normalises the Hecke operators by the weight, but because our weights are trivial on $T_0$, this normalisation is trivial). Therefore our assumption implies that the central character of $\sigma$ is trivial on $A_{\mbf{G}}(\mbb{A})$, so we can descend to $[\mbf{H}]$.
\end{proof}

\begin{remark} \label{UFJRemark}
If we define $[\mbf{H}_0] = \mbf{H}_0(\mbb{Q}) \backslash \mbf{H}_0(\mbb{A})$, then $[\mbf{H}]$ is the disjoint union of finitely many translates of $[\mbf{H}_0]$. Therefore the integral (over $[\mbf{H}]$) in Corollary \ref{RelationToAutoPeriods} is non-zero if and only if
\[
\int_{[\mbf{H}_0]} G(h) \cdot \chi_{(\lambda, j)}( \nu(h)) \bar{d} h
\]
is non-zero. This latter integral is a so-called \emph{unitary Friedberg--Jacquet period}.
\end{remark}

%%%%%%%%%%%%%%%%%%%%%%%%%%%%%%%%%%%%%%%%%%%%%%%%%%%%%%%%%%%%%%%%%%%%%%%

\appendix

%%%%%%%%%%%%%%%%%%%%%%%%%%%%%%%%%%%%%%%%%%%%%%%%%%%%%%%%%%%%%%%%%%%%%%%%
%%%%%    APPENDIX A  %%%%%%%%%%%%%%%%%%%%%%%%%%%%%%%%%%%%%%%%%%%%%%%%%%%
%%%%%%%%%%%%%%%%%%%%%%%%%%%%%%%%%%%%%%%%%%%%%%%%%%%%%%%%%%%%%%%%%%%%%%%%

\section{Branching laws} \label{AppendixBranchingLaws}

The goal of this appendix is to prove Theorem \ref{PadicVectorThm}. The idea is to $p$-adically interpolate the branching law appearing in Proposition \ref{ClassicalBranchingProp}. Since the groups $M_G$ and $M_H$ are products of general linear groups indexed by the CM type $\Psi$ (and an additional ``similitude factor''), it will be more convenient to analyse the branching law for each factor.

Unfortunately this means that we will have to use conflicting notation when performing this case-by-case analysis; therefore, we warn the reader that the notation in \S\ref{AppendixPrelim}--\ref{AppendixGL2n} is different from the rest of the article. We have however endeavoured to keep the notation uniform throughout these four subsections (e.g. the element $u$ and torus $T^{\diamondsuit}$ play the same role in the analysis, but change for each group). We hope that this change doesn't cause any confusion.

\subsection{A preliminary lemma} \label{AppendixPrelim}

For a split unramified reductive group $G$ over $\mbb{Z}_p$, let $B_G \subset G$ denote a Borel subgroup and $\overline{B}_G$ its opposite with respect to a fixed maximal torus $T \subset B_G$. Let $U_G \subset B_G$ and $\overline{U}_G \subset \overline{B}_G$ denote the unipotent radicals.

Let $\invs{G}$ denote the adic generic fibre of the completion of $G$ along its special fibre, and let $G^{\mathrm{an}}$ denote the analytification of $G_{\mbb{Q}_p}$ (so we have $\invs{G} \subset G^{\mathrm{an}}$). We use similar notation for $U_G$, $B_G$, etc. For an integer $r \geq 1$, we let $\invs{G}^1_r$ denote the subgroup of $\invs{G}$ of elements which reduce to the identity modulo $p^r$. Similarly, for $\invs{H} = \invs{U}_G$, $\overline{\invs{U}}_G$, $\invs{B}_G$, $\overline{\invs{B}}_G$, let $\invs{H}^1_r$ denote the elements in $\invs{H}$ which reduce to the identity modulo $p^r$.

Recall the notation $\invs{B}_r^{\circ} \subset \overline{\invs{B}}^{\circ}_r \subset \invs{B}_r \subset \overline{\invs{B}}_r$ for the four different ``flavours of disc'' in \S \ref{TubesInFlagVar}.

\begin{lemma} \label{AppendixGLdLemma}
Let $d, r \geq 1$ and $Y$ a $(d \times d)$-matrix with entries in $\invs{B}_r^{\circ}$. Let $\xi$ denote the antidiagonal $(d \times d)$-matrix with $1$s along the antidiagonal. Then there exist elements $R \in \invs{U}^1_{\opn{GL}_d, r}$ and $S \in \invs{B}^1_{\opn{GL}_d, r}$ such that
\[
\xi + Y = R \cdot \xi \cdot S .
\]
\end{lemma}
\begin{proof}
The element $1+Y\xi^{-1}$ defines an element of the group $\invs{GL}_{d, r}^1$. One has an Iwahori decomposition
\[
\invs{GL}_{d, r}^1 = \invs{U}^1_{\opn{GL}_d, r} \cdot \overline{\invs{B}}^1_{\opn{GL}_d, r}
\]
so there exist elements $R \in \invs{U}^1_{\opn{GL}_d, r}$ and $S' \in \overline{\invs{B}}^1_{\opn{GL}_d, r}$ such that $1+Y\xi^{-1} = RS'$. We then take $S = \xi^{-1} S' \xi$.
\end{proof}

\subsection{The group \texorpdfstring{$\opn{GL}_{2n-1}$}{GL(2n-1)}}

We first establish the following lemma:

\begin{lemma} \label{AppendixGLn-1Lemma}
Let $\xi$ be the $(n \times n-1)$-matrix whose first row is zero and the bottom $(n-1 \times n-1)$-matrix is the antidiagonal matrix with $1$s along the antidiagonal. Let $Y$ be any $(n \times n-1)$-matrix with entries in $\invs{B}_r^{\circ}$. Then there exists $R \in \invs{U}^1_{\opn{GL}_n, r}$ and $S \in \invs{B}^1_{\opn{GL}_{n-1}, r}$ such that
\[
\xi + Y = R \cdot \xi \cdot S .
\]
\end{lemma}
\begin{proof}
We denote the top row of $Y$ by $y$ and the bottom $(n-1 \times n-1)$-matrix by $Y'$. Let $R' \in \invs{U}^1_{\opn{GL}_n, r}$ and $S \in \invs{B}^1_{\opn{GL}_{n-1}, r}$ be as in Lemma \ref{AppendixGLdLemma} such that
\[
\xi' + Y' = R' \cdot \xi' \cdot S
\]
where $\xi'$ is the $(n-1 \times n-1)$ antidiagonal matrix with non-zero entries equal to $1$. Then we take 
\[
R = \tbyt{1}{r}{}{R'} \in \invs{U}^1_{\opn{GL}_n, r}
\]
where $r = yS^{-1}(\xi')^{-1}$.
\end{proof}

Let $G = \opn{GL}_{2n-1}$ and $H =  \opn{GL}_{n-1} \times \opn{GL}_n$ over $\mbb{Z}_p$. We consider $H$ as a subgroup of $G$ via the block diagonal embedding (where the top left block is of size $\opn{GL}_{n-1}$). Fix the standard Borel $B_G$ and torus $T$ in $G$. Elements of the torus $T$ are given by tuples $(y_1, \dots, y_{2n-1})$ (corresponding to the entries of the diagonal matrix) and we let $T^{\diamondsuit} \subset T$ denote the subtorus of elements satisfying $y_{i} = y_{2n-i}$ for all $i=1, \dots, 2n-1$. For an integer $r \geq 1$, we set $\invs{G}^{\square}_r = \invs{G}^1_r \cdot B_G(\mbb{Z}_p)$ and $\invs{H}^{\diamondsuit}_r = \invs{H}^1_r \cdot T^{\diamondsuit}(\mbb{Z}_p)$.

Let $u \in G(\mbb{Z}_p)$ denote the block matrix
\[
u = \tbyt{1}{}{\xi}{1}
\]
where the top right block is of size $(n-1 \times n-1)$ and $\xi$ is as in Lemma \ref{AppendixGLn-1Lemma}. 

\begin{proposition} \label{BoxDecompositionProp}
One has the following equality
\[
\invs{G}^{\square}_r = \left( u^{-1} \invs{H}^{\diamondsuit}_r u \right) \cdot \left( \invs{G}^{\square}_r \cap \invs{B}_G \right) . 
\]
\end{proposition}
\begin{proof}
By multiplying by elements of $\left(\invs{G}_r^{1} \cap \invs{B}_G \right) B_G(\mbb{Z}_p) $ on the right, we are reduced to proving the statement
\[
\overline{\invs{U}}^1_{G, r} \subset \left( u^{-1} \invs{H}^{\diamondsuit}_r u \right) \cdot \left( \invs{G}^{\square}_r \cap \invs{B}_G \right) 
\]
because one has an Iwahori decomposition $\invs{G}^1_r = \overline{\invs{U}}^1_{G, r} \cdot (\invs{G}^1_r \cap \invs{B}_G)$. Let $x \in \overline{\invs{U}}^1_{G, r}$ be a general element written as a block matrix
\[
x = \tbyt{x_1}{}{x_2}{x_3}
\]
where the top left (resp. bottom right) block has size $(n-1 \times n-1)$ (resp. $n \times n$). Then 
\[
h \defeq \tbyt{x_1}{}{}{x_3}
\]
defines an element of $\invs{H}^1_r$. Let $\overline{N}$ denote the unipotent radical of the standard opposite parabolic of $G$ with Levi $H$. Then we have 
\[
(u^{-1} h^{-1} u) \cdot x \in \overline{\invs{N}}_{r}^1   
\]
where $\overline{\invs{N}}_{r}^1$ denote the subgroup of $\overline{\invs{N}}$ of elements which reduce to the identity modulo $p^{r}$. Hence we are reduced to proving $\overline{\invs{N}}_{r}^1 \subset \left( u^{-1} \invs{H}^{\diamondsuit}_r u \right) \cdot \left( \invs{G}^{\square}_r \cap \invs{B}_G \right)$. But if 
\[
{\tbyt{1}{}{Y}{1} \in \overline{\invs{N}}_{r}^1}
\]
is a general element, then we have 
\[
\tbyt{1}{}{Y}{1} = u^{-1} \tbyt{S^{-1}}{}{}{R} u \tbyt{S}{}{}{R^{-1}}
\]
where $R, S$ are as in Lemma \ref{AppendixGLn-1Lemma}.
\end{proof}

\begin{remark} \label{BoxDecompositionRemark}
The proof of Proposition \ref{BoxDecompositionProp} in fact shows that $\invs{G}^{\square}_r = \left( u^{-1} \invs{H}^{1}_r u \right) \cdot \left( \invs{G}^{\square}_r \cap \invs{B}_G \right)$.
\end{remark}

\subsection{The group \texorpdfstring{$\opn{GL}_1 \times \opn{GL}_{2n-1}$}{GL(1) x GL(2n-1)}}

We now let $G = \opn{GL}_1 \times \opn{GL}_{2n-1}$ and $H = \opn{GL}_1 \times \opn{GL}_{n-1} \times \opn{GL}_n$ embedded block diagonally. Define $\invs{G}^{\square}_r$ and $\invs{H}^{\diamondsuit}_r$ analogously as in the previous section, where now $T^{\diamondsuit}$ is the subtorus of elements $(y_1, \dots, y_{2n})$ with $y_1 = y_{n+1}$ and $y_{i} = y_{2n+2-i}$ for all $i=2, \dots, 2n$.

We take $u \in G(\mbb{Z}_p)$ to be the element which is $1$ in the $\opn{GL}_1$-component, and equal to the element $u$ in the previous section in the $\opn{GL}_{2n-1}$-component. Then we obtain the following decomposition:

\begin{proposition}
Let $r \geq 1$. Then we have 
\[
\invs{G}^{\square}_r = \left( u^{-1} \invs{H}^{\diamondsuit}_r u \right) \cdot \left( \invs{G}^{\square}_r \cap \invs{B}_G \right) .
\]
\end{proposition}
\begin{proof}
This follows from Proposition \ref{BoxDecompositionProp} and Remark \ref{BoxDecompositionRemark}. 
\end{proof}

\subsection{The group \texorpdfstring{$\opn{GL}_{2n}$}{GL(2n)}} \label{AppendixGL2n}

We now let $G = \opn{GL}_{2n}$ and $H = \opn{GL}_n \times \opn{GL}_n$ embedded block diagonally. We define $\invs{G}^{\square}_r$ and $\invs{H}^{\diamondsuit}_r$ analogously as in the previous section, but now $T^{\diamondsuit}$ is the subtorus given by elements $(y_1, \dots, y_{2n})$ satisfying $y_i = y_{2n+1-i}$ for all $i=1, \dots, 2n$.

We let $u \in G(\mbb{Z}_p)$ denote the block matrix
\[
u = \tbyt{1}{}{\xi}{1}
\]
where all blocks are of size $(n \times n)$, and $\xi$ is the antidiagonal matrix with non-zero entries equal to $1$.

\begin{proposition} \label{GL2nBoxDecomp}
Let $r \geq 1$. Then we have 
\[
\invs{G}^{\square}_r = \left( u^{-1} \invs{H}^{\diamondsuit}_r u \right) \cdot \left( \invs{G}^{\square}_r \cap \invs{B}_G \right) .
\]
\end{proposition}
\begin{proof}
By reasoning as in the proof of Proposition \ref{BoxDecompositionProp}, it is enough to show
\[
\invs{\overline{N}}_{r}^1 \subset \left( u^{-1} \invs{H}^{\diamondsuit}_r u \right) \cdot \left( \invs{G}^{\square}_r \cap \invs{B}_G \right)
\]
where $\overline{N}$ denotes the unipotent radical of the standard opposite parabolic of $G$ with Levi $H$. But this follows from the same proof in Proposition \ref{BoxDecompositionProp} using Lemma \ref{AppendixGLdLemma} (with $d=2n$).
\end{proof}

\subsection{Proof of Theorem \ref{PadicVectorThm}}

We now return to the setting of \S \ref{LocallyAnalyticCohomologySection} (and return to using the notation introduced in the main body of the article). By combining the previous sections, we immediately find that:

\begin{proposition} \label{BoxDecompMsquare}
Let $r \geq 1$. Then one has equalities 
\[
\invs{M}^{\square}_{G, r} = \left( u^{-1} \invs{M}^{\diamondsuit}_{H, r} u \right) \cdot \left( \invs{M}^{\square}_{G, r} \cap \invs{B}_{M_G} \right), \quad \quad \invs{M}^{\square}_{G, r} = \left( u^{-1} \invs{M}^{\clubsuit}_{H, r} u \right) \cdot \left( \invs{M}^{\square}_{G, r} \cap \invs{B}_{M_G} \right) .
\]
\end{proposition}
\begin{proof}
For the first equality, this follows by breaking up the groups into the factors indexed by $\tau \in \Psi$. The factor corresponding to $\tau_0$ follows from Proposition \ref{BoxDecompositionProp}, and the factors for $\tau \neq \tau_0$ follow from Proposition \ref{GL2nBoxDecomp}. There is nothing to check for the extra $\opn{GL}_1$-factors in $M_G$ and $M_H$. The second equality follows from $u^{-1} \invs{M}^{\diamondsuit}_{H, r} u \subset u^{-1} \invs{M}^{\clubsuit}_{H, r} u \subset \invs{M}^{\square}_{G, r}$.
\end{proof}

We now introduce the relevant algebraic weights for representations of $M_G$. Recall any algebraic character of the torus $T$ can be represented by a tuple 
\[
\kappa = (\kappa_0; \kappa_{1, \tau}, \dots, \kappa_{2n, \tau})_{\tau \in \Psi}
\]
where $\kappa_0$ and $\kappa_{i, \tau}$ are integers. By the $\tau$-factor or $\tau$-component of $\kappa$, we mean the tuple $(\kappa_{1, \tau}, \dots, \kappa_{2n, \tau})$, and by the $\opn{GL}_1$-factor, we mean the integer $\kappa_0$. It will be helpful to use this terminology when defining certain characters below.

\begin{definition}
Let $\kappa$ be an algebraic character of $T$ as above. We say:
\begin{enumerate}
    \item $\kappa$ is $M_G$-dominant if 
    \[
    \kappa_{2, \tau_0} \geq \cdots \geq \kappa_{2n, \tau_0}
    \]
    and 
    \[
    \kappa_{1, \tau} \geq \cdots \geq \kappa_{2n, \tau} 
    \]
    for all $\tau \in \Psi - \{\tau_0\}$.
    \item $\kappa$ is pure of weight $w \in \mbb{Z}$ if 
    \[
    \kappa_{i, \tau_0} + \kappa_{2n+2-i, \tau_0} = w
    \]
    for all $i = 2, \dots, n$, and $\kappa_{i, \tau} + \kappa_{2n+1-i, \tau} = 0$ for all $i=1, \dots, 2n$ and $\tau \neq \tau_0$.
\end{enumerate}
\end{definition}

The set of characters which are pure (of some weight $w \in \mbb{Z}$) form a group, and we let $\invs{C}$ denote the submonoid of $M_G$-dominant characters which are pure of weight $w \leq 0$ satisfying $\kappa_{n+1, \tau_0} \leq w$. We will always write the group law for $\invs{C}$ additively. We consider the following special elements of $\invs{C}$:
\begin{itemize}
    \item $\mu_0 = (1; 0, \dots, 0)_{\tau \in \Psi}$
    \item $\mu_w = (\mu_{w, 0}, \mu_{w, 1, \tau}, \dots, \mu_{w, 2n, \tau})_{\tau \in \Psi}$, where $\mu_{w, 0} = \mu_{w, 1, \tau_0} = \mu_{w, i, \tau} = 0$ for all $i=1, \dots, 2n$ and $\tau \neq \tau_0$, and we have 
    \[
    \mu_{w, 2, \tau_0} = \cdots = \mu_{w, n, \tau_0} = 0 \quad \quad \mu_{w, n+1, \tau_0} = \cdots = \mu_{w, 2n, \tau_0} = -1 
    \]
    \item $\mu_{1, \tau_0}$ which is the identity in the $\opn{GL}_1$-factor and $\tau \neq \tau_0$ factors, and in the $\tau_0$-factor is given by
    \[
    (1, 0, \dots, 0) .
    \]
    \item For $i=2, \dots, n$, we let $\mu_{i, \tau_0}$ be the character which is the identity in the $\opn{GL}_1$-factor and $\tau \neq \tau_0$ factors, and in the $\tau_0$-factor is given by
    \[
    (0, 1, \dots, 1, 0, \dots, 0, -1, \dots, -1)
    \]
    where there are $i-1$ lots of $1$s and $-1$s.
    \item We let $\mu_{n+1, \tau_0}$ be the character which is the identity outside the $\tau_0$-factor, and the $\tau_0$-factor is given by
    \[
    (0, 1, \dots, 1, -1, \dots, -1)
    \]
    where there are $n-1$ lots of $1$ and $n$ lots of $-1$
    \item For $i=1, \dots, n$ and $\tau \neq \tau_0$, we let $\mu_{i, \tau}$ denote the character which is the identity outside the $\tau$-factor, and at the $\tau$-factor is 
    \[
    (1, \dots, 1, 0, \dots, 0, -1, \dots, -1)
    \]
    where there are $i$ lots of $1$s and $-1$s .
\end{itemize}
This collection of characters forms a generating set for $\invs{C}$ in the following sense: for any $\kappa \in \invs{C}$, there exist unique integers $a_0, a_{1, \tau_0} \in \mbb{Z}$ and $a_w, a_{i, \tau} \in \mbb{Z}_{\geq 0}$ for $(i, \tau) \neq (1, \tau_0)$, such that
\[
\kappa = a_0 \mu_0 + a_w \mu_w + a_{n+1, \tau_0} \mu_{n+1, \tau_0} + \sum_{i=1}^n \sum_{\tau \in \Psi} a_{i, \tau} \mu_{i, \tau} .
\]
Explicitly, the integers are given by
\begin{itemize}
    \item $a_0 = \kappa_0$
    \item $a_w = -(\kappa_{2, \tau_0} + \kappa_{2n, \tau_0})$
    \item $a_{1, \tau_0} = \kappa_{1, \tau_0}$
    \item For $i=2, \dots, n+1$, one has
    \[
    a_{i, \tau_0} = \left\{ \begin{array}{cc} \kappa_{i, \tau_0} - \kappa_{i+1, \tau_0} & \text{ if } i \leq n-1 \\ \kappa_{n+1, \tau_0} - \kappa_{n+2, \tau_0} & \text{ if } i=n \\ (\kappa_{n, \tau_0} + \kappa_{n+2, \tau_0}) - \kappa_{n+1, \tau_0} & \text{ if } i=n+1 \end{array} \right.
    \]
    \item For $i=1, \dots, n$ and $\tau \neq \tau_0$, one has
    \[
    a_{i, \tau} = \left\{ \begin{array}{cc} \kappa_{i, \tau} - \kappa_{i+1, \tau} & \text{ if } i \leq n-1 \\ \kappa_{n, \tau} & \text{ if } i=n \end{array} \right.
    \]
\end{itemize}

Let $\invs{D} = \prod_{\tau \neq \tau_0} \mbb{Z}_{\geq 0}$ equipped with the monoid structure given by component-wise addition. We will denote elements of $\invs{D}$ by tuples $j = (j_{\tau})_{\tau \neq \tau_0}$. We let $\invs{E} \subset \invs{C} \times \invs{D}$ be the collection of pairs $(\kappa, j)$ which satisfy $j_{\tau} \leq \kappa_{n, \tau}$ for all $\tau \neq \tau_0$. This forms a submonoid of $\invs{C} \times \invs{D}$. Then $\invs{E}$ has a generating set given by the pairs $(\mu_0, 0)$, $(\mu_w, 0)$, $(\mu_{i, \tau}, 0)$, and $(\mu_{n, \tau}, 1_{\tau} )$ for $\tau \neq \tau_0$, where $1_{\tau} \in \invs{D}$ is the tuple which is zero outside $\tau$ and has $1$ in the $\tau$-component. More precisely, for any $(\kappa, j) \in \invs{E}$, there exist unique integers $a_0, a_{1, \tau_0} \in \mbb{Z}$, $a_w, a_{i, \tau} \in \mbb{Z}_{\geq 0}$ for $(i, \tau) \neq (1, \tau_0)$, and $b_{\tau} \in \mbb{Z}_{\geq 0}$ for $\tau \neq \tau_0$ such that
\[
(\kappa, j) = a_0 (\mu_0, 0) + a_w (\mu_w, 0) + a_{n+1, \tau_0} (\mu_{n+1, \tau_0}, 0) + \sum_{i=1}^n \sum_{\tau \in \Psi} a_{i, \tau} (\mu_{i, \tau}, 0) + \sum_{\tau \neq \tau_0}b_{\tau} (\mu_{n, \tau}, 1_{\tau}) .
\]
Explicitly, the integers are given by
\begin{itemize}
    \item $a_0, a_w, a_{1, \tau_0}, \dots, a_{n+1, \tau_0}$ and $a_{1, \tau}, \dots, a_{n-1, \tau}$ are given by the formulae above
    \item For $\tau \neq \tau_0$, one has $a_{n, \tau} = \kappa_{n, \tau} - j_{\tau}$
    \item $b_{\tau} = j_{\tau}$ .
\end{itemize}

\begin{definition}
For any $(\kappa, j) \in \invs{E}$, we let $\sigma_{\kappa}^{[j]}$ denote the character of $M_H$ given by sending a general element $(x; y_1, y_2, y_3; z_{1, \tau}, z_{2, \tau})_{\tau \neq \tau_0}$ to 
\[
x^{-\kappa_0}y_{1}^{-\kappa_{1, \tau_0}} \opn{det}y_2^{\kappa_{n+1, \tau_0} - w} \opn{det}y_3^{-\kappa_{n+1, \tau_0}} \prod_{\tau \neq \tau_0} \opn{det}z_{1, \tau}^{-j_{\tau}} \opn{det}z_{2, \tau}^{j_{\tau}} 
\]
where $w = \kappa_{2, \tau_0} + \kappa_{2n, \tau_0}$ denotes the weight of $\kappa$.
\end{definition}

For any $\kappa \in \invs{C}$, let $V_{\kappa}$ denote the irreducible algebraic representation of $M_G$ with highest weight $\kappa$, which can be viewed as the space of algebraic functions $f \colon M_G \to \mbb{A}^1$ satisfying 
\[
f(m b) = (w_{M_G}^{\mathrm{max}}\kappa)(b^{-1})f(m)
\]
for all $b \in B_{M_G}$. The action of $M_G$ on $f$ is then given by $m \cdot f(n) = f(m^{-1} n)$. We have the following classical branching law:

\begin{theorem} \label{ClassicalAppendixBranchingThm}
Let $(\kappa, j) \in \invs{E}$. Then there exists a unique vector $x_{\kappa}^{[j]} \in V_{\kappa}$ such that:
\begin{enumerate}
    \item $x_{\kappa}^{[j]}$ is an eigenvector for the action of $u^{-1} M_H u$ with eigencharacter given by the inverse of $\sigma_{\kappa}^{[j]}$
    \item $x_{\kappa}^{[j]}(1) = 1$, where we are viewing $x_{\kappa}^{[j]} \colon M_G \to \mbb{A}^1$ as an algebraic function
    \item The vectors $x_{\mu_0}^{[0]}$ and $x_{\mu_{1, \tau_0}}^{[0]}$ are invertible in $\ordd(M_G)$, and we have
    \[
    x_{\kappa}^{[j]} = (x_{\mu_0}^{[0]})^{a_0} \cdot (x_{\mu_w}^{[0]})^{a_w} \cdot (x_{\mu_{n+1, \tau_0}}^{[0]})^{a_{n+1, \tau_0}} \cdot \prod_{\substack{i=1, \dots, n \\ \tau \in \Psi}} (x_{\mu_{i, \tau}}^{[0]})^{a_{i, \tau}} \cdot \prod_{\tau \neq \tau_0} (x_{\mu_{n, \tau}}^{[1_{\tau}]})^{b_{\tau}}
    \]
    where the product takes place in $\ordd(M_G)$ and the exponents are the integers above.
\end{enumerate}
\end{theorem}
\begin{proof}
By applying \cite[Theorem 2.1]{Knapp} for each general linear factor of $M_G$,\footnote{Knapp works in the setting of compact unitary groups, but the proof works verbatim for general linear groups.} there exists a unique up to scaling (non-zero) vector $x_{\kappa}^{[j]} \in V_{\kappa}$ satisfying property (1). Since $u^{-1}M_H u B_{M_G}$ is Zariski open in $M_G$ (Lemma \ref{OpenOrbitLemma}), the vector is non-vanishing on this cell, so we can normalise $x_{\kappa}^{[j]}$ as in (2) to determine the vector uniquely. The vectors $x_{\mu_0}^{[0]}$ and $x_{\mu_{1, \tau_0}}^{[0]}$ are invertible in $\ordd(M_G)$ because the corresponding representations $V_{\mu_0}$ and $V_{\mu_{1, \tau_0}}$ are one-dimensional. Property (3) then follows immediately from uniqueness, the identity
\[
\sigma_{\kappa}^{[j]} = (\sigma_{\mu_0})^{a_0} \cdot (\sigma_{\mu_w}^{[0]})^{a_w} \cdot (\sigma_{\mu_{n+1, \tau_0}}^{[0]})^{a_{n+1}, \tau_0} \cdot (\sigma_{\mu_{1, \tau_0}}^{[0]})^{a_{1, \tau_0}} \cdot \prod_{\tau \neq \tau_0} (\sigma_{\mu_{n, \tau}}^{[1_{\tau}]})^{b_{\tau}}
\]
and the fact that $\sigma_{\mu_{i, \tau}}^{[0]}$ is the trivial character for $(i, \tau) \neq (1, \tau_0), (n+1, \tau_0)$.
\end{proof}

\begin{remark}
Note that we introduced some asymmetry here -- we could have equally worked with the monoid $\invs{D} = \prod_{\tau \neq \tau_0} \mbb{Z}_{\leq 0}$ (or even more generally, products of $\mbb{Z}_{\geq 0}$ and $\mbb{Z}_{\leq 0}$) and the monoid $\invs{E}$ defined by the equations $-j_{\tau} \leq \kappa_{n, \tau}$. 
\end{remark}

To prove Theorem \ref{PadicVectorThm}, we will use a $p$-adic version of the product formula in Theorem \ref{ClassicalAppendixBranchingThm}(3). 

\begin{lemma} \label{CoefficientRanalyticLemma}
Let $(A, A^+)$ be a Tate algebra over $(\mbb{Q}_p, \mbb{Z}_p)$, and suppose that $\kappa \colon T(\mbb{Z}_p) \to (A^+)^{\times}$ is an $r$-analytic character, for some $r \in \mbb{Q}_{>0}$, which satisfies
\[
\kappa_{i, \tau_0}+\kappa_{2n+2-i, \tau_0} = \kappa_{j, \tau_0} + \kappa_{2n+2-j, \tau_0}
\]
for all $i, j = 2, \dots, n$, and $\kappa_{i, \tau} + \kappa_{2n+1-i, \tau} = 0$ for all $i=1, \dots, n$ and $\tau \neq \tau_0$. Let $\beta = (\beta_{\tau}) \colon \prod_{\tau \neq \tau_0} \mbb{Z}_p^{\times} \to (A^+)^{\times}$ be an $r$-analytic character. Then there exist unique $r$-analytic characters 
\begin{itemize}
    \item $\xi_0, \xi_w, \xi_{i, \tau} \colon \mbb{Z}_p^{\times} \to (A^+)^{\times}$ for $i=1, \dots, n$ and $\tau \in \Psi$ 
    \item $\xi_{n+1, \tau_0} \colon \mbb{Z}_p^{\times} \to (A^+)^{\times}$ 
    \item $\Xi_{\tau} \colon \mbb{Z}_p^{\times} \to (A^+)^{\times}$ for $\tau \in \Psi - \{\tau_0 \}$
\end{itemize}
such that 
\begin{align*}
(\kappa, \beta) = \xi_0 \circ (\mu_0, 0) &+ \xi_w \circ (\mu_w, 0) + \xi_{n+1, \tau_0} \circ (\mu_{n+1, \tau_0}, 0) \\ &+ \sum_{i=1}^n \sum_{\tau \in \Psi} \xi_{i, \tau} \circ (\mu_{i, \tau}, 0) + \sum_{\tau \neq \tau_0}\Xi_{\tau} \circ (\mu_{n, \tau}, 1_{\tau})
\end{align*}
where the group law is written additively. 
\end{lemma}
\begin{proof}
We define the $r$-analytic characters via the same formulae as above, i.e. $\xi_0 = \kappa_0$, $\xi_w = -(\kappa_{2, \tau_0} + \kappa_{2n, \tau_0})$, etc. It is clear that these are uniquely determined. 
\end{proof}

We will also need the following lemma:

\begin{lemma}
Let $r \in \mbb{Z}_{>0}$. Then for any $(\kappa, j) \in \invs{E}$, one has
\[
x_{\kappa}^{[j]}\left( \invs{M}^{\square}_{G, r} \right) \subset \mbb{Z}_p^{\times} (1+\invs{B}_r) 
\]
where we are viewing $x_{\kappa}^{[j]}$ as an analytic function $M_G^{\mathrm{an}} \to \mbb{A}^{1, \mathrm{an}}$.
\end{lemma}
\begin{proof}
By Proposition \ref{BoxDecompMsquare}, we have $\invs{M}^{\square}_{G, r} = \left(u^{-1} \invs{M}^{\clubsuit}_{H, r} u \right) (\invs{M}^{\square}_{G, r} \cap \invs{B}_{M_G} )$, therefore the transformation properties for $x_{\kappa}^{[j]}$ imply that
\[
x_{\kappa}^{[j]}(m) = \sigma_{\kappa}^{[j]}(m_1) \cdot (w_{M_G}^{\mathrm{max}}\kappa)(m_2^{-1})
\]
for any $m \in \invs{M}^{\square}_{G, r}$ satisfying $m=u^{-1}m_1u \cdot m_2$ for $m_1 \in \invs{M}^{\clubsuit}_{H, r}$ and $m_2 \in \invs{M}^{\square}_{G, r} \cap \invs{B}_{M_G}$. But $\sigma^{[j]}_{\kappa}$ and $w_{M_G}^{\mathrm{max}}\kappa$ are algebraic characters, so their analytifications map $\invs{M}^{\clubsuit}_{H, r}$ and $\invs{M}^{\square}_{G, r} \cap \invs{B}_{M_G}$ into $\mbb{Z}_p^{\times}(1+\invs{B}_r)$, as required.
\end{proof}

\begin{remark}
The previous lemma implies that for any $r$-analytic character $\xi \colon \mbb{Z}_p^{\times} \to (A^+)^{\times}$, the composition $\xi \circ x_{\kappa}^{[j]}$ defines an analytic function
\[
\xi \circ x_{\kappa}^{[j]} \colon \invs{M}_{G, r}^{\square} \times \opn{Spa}(A, A^+) \to \mbb{G}_{m, A}^{\mathrm{an}} \subset \mbb{A}^{1, \mathrm{an}}_{A} .
\]
\end{remark}

We now introduce the $p$-adic vectors. Recall that for an $r$-analytic weight $\kappa \colon T(\mbb{Z}_p) \to (A^+)^{\times}$, we let $V_{\kappa}^{r\mathrm{-an}}$ denote the $r$-analytic induction as in Definition \ref{DefinitionOfDistributions}. 

\begin{definition}
Let $r \in \mbb{Z}_{>0}$ and let $(\kappa, \beta)$ be a pair of $r$-analytic characters as in Lemma \ref{CoefficientRanalyticLemma}. Then we define
\[
x_{\kappa}^{[\beta]} \defeq (x_{\mu_0}^{[0]})^{\xi_0} \cdot (x_{\mu_w}^{[0]})^{\xi_w} \cdot (x_{\mu_{n+1, \tau_0}}^{[0]})^{\xi_{n+1, \tau_0}} \cdot \prod_{\substack{i=1, \dots, n \\ \tau \in \Psi}} (x_{\mu_{i, \tau}}^{[0]})^{\xi_{i, \tau}} \cdot \prod_{\tau \neq \tau_0} (x_{\mu_{n, \tau}}^{[1_{\tau}]})^{\Xi_{\tau}}
\]
where the product takes place in $\ordd(\invs{M}^{\square}_{G, r}) \hatot A$ and the analytic characters $\xi_{...}$ and $\Xi_{...}$ are as in Lemma \ref{CoefficientRanalyticLemma}. Here we have written $(-)^{\xi}$ as a shorthand for $\xi \circ (-)$. This defines an element of $V^{r\mathrm{-an}}_{\kappa}$.

We let $\sigma_{\kappa}^{[\beta]}$ denote the character 
\begin{align*}
    \sigma_{\kappa}^{[\beta]} \colon \invs{M}^{\clubsuit}_{H, r} \times \opn{Spa}(A, A^+) &\to \mbb{G}_{m, A}^{\mathrm{an}} \\
    (x; y_1, y_2, y_3, z_{1, \tau}, z_{2, \tau}) &\mapsto \kappa_0(x^{-1}) \kappa_{1, \tau_0}(y_1^{-1}) (\kappa_{n+1, \tau_0}\kappa_{2, \tau_0}^{-1}\kappa_{2n, \tau_0}^{-1})(\opn{det}y_2)\kappa_{n+1, \tau_0}^{-1}(\opn{det}y_3) \cdot \\
    & \quad \quad \cdot \prod_{\tau \neq \tau_0} \beta_{\tau}(\opn{det}z_{1, \tau}^{-1}\opn{det}z_{2, \tau})
\end{align*}
which makes sense because $\kappa$ and $\beta$ are $r$-analytic.
\end{definition}

Finally, we obtain the following theorem:

\begin{theorem} \label{AppendixAMainThm}
Let $r \in \mbb{Z}_{>0}$ and let $(\kappa, \beta)$ be a pair of $r$-analytic characters as in Lemma \ref{CoefficientRanalyticLemma}. Then
\begin{enumerate}
    \item $x_{\kappa}^{[\beta]}$ is a (non-zero) eigenvector for the action of $u^{-1}\invs{M}^{\clubsuit}_{H, r} u$ with eigencharacter given by the inverse of $\sigma_{\kappa}^{[\beta]}$
    \item If $(B, B^+)$ is another Tate algebra with a morphism $(A, A^+) \to (B, B^+)$, and $(\kappa', \beta')$ denotes the composition of $(\kappa, \beta)$ with this morphism, then the image of $x_{\kappa}^{[\beta]}$ under the natural map
    \[
    V_{\kappa}^{r\mathrm{-an}} \to V_{\kappa'}^{r\mathrm{-an}}
    \]
    is equal to $x_{\kappa'}^{[\beta']}$
    \item If $(\kappa, \beta)$ arises from a pair of algebraic characters $(\kappa, j) \in \invs{E}$, then $x_{\kappa}^{[\beta]}$ is equal to the image of $x_{\kappa}^{[j]}$ under the the natural map
    \[
    V_{\kappa} \to V_{\kappa}^{r\mathrm{-an}}
    \]
    given by restricting (the analytification of) a function $M_G \to \mbb{A}^1$ to $\invs{M}^{\square}_{G, r}$. 
    \item The vector $x_{\kappa}^{[\beta]}$ does not depend on the radius of analyticity, i.e. if $r' \geq r$ is another integer, then the constructions for $x_{\kappa}^{[\beta]}$ coincide under the map
    \[
    V_{\kappa}^{r\mathrm{-an}} \to V_{\kappa}^{r'\mathrm{-an}}
    \]
    given by restriction to $\invs{M}^{\square}_{G, r'}$.
\end{enumerate}
\end{theorem}
\begin{proof}
Part (1) follows from the fact we have a similar product formula for $\sigma_{\kappa}^{[\beta]}$ as in the proof of Theorem \ref{ClassicalAppendixBranchingThm}, replacing the coefficients $a_{...}$ and $b_{...}$ by $\xi_{...}$ and $\Xi_{...}$. 

The remaining properties are clear from construction, using the fact that the characters $\xi_{...}$ and $\Xi_{...}$ are unique and Theorem \ref{ClassicalAppendixBranchingThm}(3). 
\end{proof}

%%%%%%%%%%%%%%%%%%%%%%%%%%%%%%%%%%%%%%%%%%%%%%%%%%%%%%%%%%%%%%%%%%%%%%%%%%%%%%%%%%%%%%%%%%%%%%%
%%%%%%%%    APPENDIX B    %%%%%%%%%%%%%%%%%%%%%%%%%%%%%%%%%%%%%%%%%%%%%%%%%%%%%%%%%%%%%%%%%%%%%
%%%%%%%%%%%%%%%%%%%%%%%%%%%%%%%%%%%%%%%%%%%%%%%%%%%%%%%%%%%%%%%%%%%%%%%%%%%%%%%%%%%%%%%%%%%%%%%

\section{Comparisons in families} \label{AppendixComparisonsFamilies}

In this appendix, we describe the key ingredient needed to compare the coherent cohomology classes associated with algebraic Hecke characters and (algebraic) $p$-adic Hecke characters. We restrict ourselves to the case of PEL Shimura data which give rise to compact Shimura varieties -- more general versions of the functorial properties we describe can be found in \cite{DLLZ18}.

\subsection{Canonical constructions}

In this section we let $(\mathscr{G}, X_{\mathscr{G}})$ be a PEL-type Shimura--Deligne datum satisfying (SD5) as in \cite[\S B.3]{ACES}. Suppose that the associated Shimura variety admits a canonical model over the reflex field, which we will denote by $F$. We fix a rational prime $p > 2$ which is unramified in $F$ and for which $\mathscr{G}_{\mbb{Q}_p}$ is unramified. We fix a prime $\ide{p}$ of $F$ lying above $p$.

Let $K \subset \mathscr{G}(\mbb{A}_f)$ be a neat compact open subgroup. Then the Shimura variety $S_{\mathscr{G}, K}$ parameterises abelian varieties $A$ with PEL structure (corresponding to the PEL-data defining $(\mathscr{G}, X_{\mathscr{G}})$), such that the first relative homology of $A$ is modelled on the defining representation for $\mathscr{G}$. Let $S = S_{\mathscr{G}, U}$ and denote the universal abelian variety over $S$ by $A$. We are interested in the local systems/locally free sheaves obtained from the relative homology of $A$.

\begin{assumption} \label{AppendixCompactAssumption}
We assume that the Shimura variety $S$ is compact.
\end{assumption}

Recall that there exist ``canonical constructions'' $\xi_B$ (resp. $\xi_{dR}$, resp. $\xi_{\et}$) which are tensor functors on the category of algebraic representations of $\mathscr{G}$ valued in the category of variations of Hodge structures over $S(\mbb{C})$ (resp. locally free sheaves on $S$ with an integrable connection, resp. $p$-adic local systems on $S$). More precisely, if $V$ is an algebraic representation of $\mathscr{G}$, then
\begin{enumerate}
    \item The variation of Hodge structure $\xi_B(V)$ is constructed from the left $\mathscr{G}(\mbb{Q})$-torsor
    \[
    X_{\mathscr{G}} \times \mathscr{G}(\mbb{A}_f) /K \to \mathscr{G}(\mbb{Q}) \backslash X_{\mathscr{G}} \times \mathscr{G}(\mbb{A}_f) /K = S(\mbb{C})
    \]
    and the $\mathscr{G}(\mbb{Q})$-representation $V$ (see \cite[\S 2.3]{CS17} for example).
    \item The locally free sheaf $\xi_{dR}(V)$ arises from the $\mathscr{G}_F$-torsor 
    \[
    \mathscr{G}_{dR} \to S
    \]
    (the standard principal bundle) and the algebraic representation $V_F$ of $\mathscr{G}_F$ (see \cite[\S III.3]{MilneCanonicalMixed}).
    \item The $p$-adic local system $\xi_{\et}(V)$ can be constructed by choosing a $\mathscr{G}(\mbb{Z}_p)$-stable lattice $T \subset V_{\mbb{Q}_p}$ and using the pro-system of torsors 
    \[
    S_{\mathscr{G}, K'} \to S_{\mathscr{G}, K}
    \]
    for $K' \subset K$ (see \cite[\S 4]{ACES} and the references therein for example). One can also interpret this in terms of the perfectoid Shimura variety (see \S \ref{AppendixPerfectoidShimuraVarieties} below).
\end{enumerate}

\begin{notation}
Let $V$ be an algebraic representation of $\mathscr{G}$. We write $\invs{V}_B$, $\invs{V}_{dR}$ and $\invs{V}_{\et}$ for $\xi_B(V)$, $\xi_{dR}(V)$ and $\xi_{\et}(V)$ respectively. 
\end{notation}

\begin{remark}
The above functors are normalised so that $\invs{W}_{?}$ equals the first relative homology of $A/S$ with respect to the relevant cohomology theory, for $? = B, dR, \et$, where $W$ denotes the defining representation of $\mathscr{G}$.
\end{remark}

We have several comparisons between these sheaves/local systems. 

\begin{enumerate}
    \item (Betti--$p$-adic, \cite[Expose xi]{SGA4}) Since $S$ is smooth, one has a morphism of sites $\beta \colon S_{\opn{cl}} \to S_{\et}$ from the site of \'{e}tale coverings of $S(\mbb{C})$ to the \'{e}tale site of $S$. Then for any algebraic representation $V$ of $\mathscr{G}$, one has
    \[
    \invs{V}_{\et} \cong \beta_* \left( \invs{V}_B \otimes_{\mbb{Q}} \mbb{Q}_p \right) .
    \]
    Indeed, one has a similar map of sites for $A$, whose pushforward is exact and commutes with pushforward along $A \to S$ (in the analytic and \'{e}tale topologies).
    \item (Betti--de Rham) For an algebraic representation $V$ of $\mathscr{G}$, one has a comparison isomorphism
    \[
    \invs{V}_{dR} \otimes_{\ordd_S} \ordd_{S(\mbb{C})} \cong \invs{V}_B \otimes_{\mbb{Q}} \ordd_{S(\mbb{C})} .
    \]
    \item (de Rham--$p$-adic, \cite[\S 2.2]{CS17}) Let $L/F_{\ide{p}}$ be a finite extension and let $A^{\mathrm{an}} \to S^{\mathrm{an}}$ be the morphism of adic spaces associated with $A_L \to S_L$. Then for any algebraic representation $V$ of $\mathscr{G}$, one has an isomorphism
    \[
    \invs{V}_{dR, L}^{\mathrm{an}} \otimes_{\ordd_{S^{\mathrm{an}}}} \ordd \mathbb{B}_{dR, S^{\mathrm{an}}} \cong \invs{V}_{\et, L}^{\mathrm{an}} \otimes_{\hat{\mbb{Q}}_p} \ordd \mathbb{B}_{dR, S^{\mathrm{an}}} 
    \]
    of sheaves on the pro-\'{e}tale site of $S^{\mathrm{an}}$ compatible with filtrations and connections. Here $(-)^{\mathrm{an}}$ means pull-back to the associated adic space.
\end{enumerate}

More precisely, one has the above comparisons for $\invs{W}_?$ and the work of Ancona \cite{Ancona2015} and Torzewski \cite{Torzewski2019} shows that all of these ``canonical constructions'' factor through a functor valued in relative Chow motives over $S$, so the comparisons can be extended to all algebraic representations. In particular, since the comparisons above are functorial with respect to algebraic operations (e.g. correspondences on $A$), the above comparisons are also functorial in the algebraic representation $V$.

\subsection{Functoriality}

Let $(\mathscr{G}_1, X_{1})$ and $(\mathscr{G}_2, X_2)$ be two PEL-type Shimura--Deligne data (with a common reflex field $F$) as in the previous subsection, including Assumption \ref{AppendixCompactAssumption}. Suppose that we have a homomorphism $f \colon \mathscr{G}_1 \to \mathscr{G}_2$ inducing a morphism of Shimura data (and arising from a morphism of PEL data). Let $K_1 \subset \mathscr{G}_1(\mathbb{A}_f)$ be a neat compact open subgroup and $K_2 \subset \mathscr{G}_2(\mbb{A}_f)$ a neat compact open subgroup containing $f(K_1)$. Let $S_i = S_{\mathscr{G}_i, K_i}$ for $i=1, 2$.

The morphism $f$ induces a map of torsors 
\[
\begin{tikzcd}
X_1 \times \mathscr{G}_1(\mathbb{A}_f) / K_1  \arrow[d] \arrow[r] & X_2 \times \mathscr{G}_2(\mathbb{A}_f)/K_2 \arrow[d] \\
S_1(\mathbb{C}) \arrow[r]                                         & S_2(\mathbb{C})                                     
\end{tikzcd}
\]
and hence a natural isomorphism $\eta_B \colon \xi_{1,B} \circ f^* \to f^* \circ \xi_{2,B}$, where we have use the notation $\xi_{i, B}$ to emphasise which Shimura variety and group the construction refers to.

Similarly, the morphism $f$ induces morphisms of (finite \'{e}tale) torsors $S_{\mathscr{G}_1, K_1'} \to S_{\mathscr{G}_2, K_2'}$ over $S_1 \to S_2$, for any $K_1' \subset K_1$ and $f(K_1') \subset K_2' \subset K_2$. These are compatible with varying $K_1'$ and $K_2'$, so induce a natural isomorphism $\eta_{\et} \colon \xi_{1, \et} \circ f^* \to f^* \circ \xi_{2, \et}$. 

\begin{lemma}
The Betti--$p$-adic comparison identifies the natural isomorphisms $\eta_{B} \otimes_{\mbb{Q}} \mbb{Q}_p$ and $\eta_{\et}$. 
\end{lemma}
\begin{proof}
Let $V$ be an algebraic representation of $\mathscr{G}_1$. Then it is well-known that one can construct $\xi_B(V)$ either by considering $V$ as a left $\mathscr{G}_1(\mbb{Q})$-module (as above) or by viewing $V$ as a right $K_1$-module (with no left $\mathscr{G}_1(\mbb{Q})$-action) and setting
\[
\xi_B(V) = \mathscr{G}_1(\mbb{Q}) \backslash X_1 \times \mathscr{G}_1(\mbb{A}_f) \times V /K_1 .
\]
In particular, choosing a $\mathscr{G}_1(\mbb{Z}_p)$-stable lattice $T \subset V_{\mbb{Q}_p}$, one easily sees that the two constructions $\xi_B(V) \otimes \mbb{Q}_p$ and $\xi_{\et}(V)$ are identified under the Betti--$p$-adic comparison. Similar calculations apply for the group $\mathscr{G}_2$.
\end{proof}

We also obtain a natural isomorphism involving the functor $\xi_{dR}$ as follows. Since $f$ induces a morphism of Shimura data, by theory of canonical models for standard principal bundles (see \cite[\S III.4]{MilneCanonicalMixed}), one obtains a morphism of torsors $\mathscr{G}_{1, dR} \to \mathscr{G}_{2, dR}$ which induces the desired natural isomorphism $\eta_{dR} \colon \xi_{1, dR} \circ f^* \to f^* \circ \xi_{2, dR}$. Pulling this back to $\mbb{C}$, this morphism of torsors is identified with the morphism
\[
\mathscr{G}_{1, dR}(\mbb{C}) = \mathscr{G}_1(\mbb{Q}) \backslash X_{1} \times \mathscr{G}_1(\mbb{C}) \times \mathscr{G}_1(\mbb{A}_f) / K_1 \to \mathscr{G}_2(\mbb{Q}) \backslash X_{2} \times \mathscr{G}_2(\mbb{C}) \times \mathscr{G}_2(\mbb{A}_f) / K_2 = \mathscr{G}_{2, dR}(\mbb{C}) 
\]
sending $[x, g, g']$ to $[f(x), f(g), f(g')]$. But $\mathscr{G}_{i, dR}(\mbb{C})$ is the pushout of the torsor $X_i \times \mathscr{G}_{i}(\mbb{A}_f)/K_i$ along the map $\mathscr{G}_{i}(\mbb{Q}) \to \mathscr{G}_i(\mbb{C})$, and it is clear that this morphism of torsors is induced from the one above. In other words, the Betti--de Rham comparison identifies $\eta_{dR} \otimes_{\ordd_S} \ordd_{S(\mbb{C})}$ and $\eta_B \otimes_{\mbb{Q}} \ordd_{S(\mbb{C})}$.

\begin{proposition} \label{dRpadicComparisonProp}
The de Rham--$p$-adic comparison identifies the natural isomorphisms $\eta_{dR}^{\mathrm{an}} \otimes_{\ordd_{S^{\mathrm{an}}}} \ordd \mathbb{B}_{dR, S^{\mathrm{an}}}$ and $\eta_{\et}^{\mathrm{an}} \otimes_{\hat{\mbb{Q}}_p} \ordd \mathbb{B}_{dR, S^{\mathrm{an}}} $.
\end{proposition}
\begin{proof}
Essentially this follows because $\eta_B$ is induced from an (absolute) Hodge cycle for a certain abelian variety, which is known to be de Rham (\cite{Blasius94}). 

Let $W_2$ denote the defining representation of $\mathscr{G}_2$. Since we already know $\eta_B$, $\eta_{dR}$, $\eta_{\et}$ are natural isomorphisms of additive tensor functors (respecting this structure and duals), and every representation $\mathscr{G}_2$ is a direct summand of tensor products of $W_2$ and $W_2^*$, it enough to check that $\eta^{\mathrm{an}}_{dR}(W_2) \otimes_{\ordd_{S^{\mathrm{an}}}} \ordd \mathbb{B}_{dR, S^{\mathrm{an}}} = \eta_{\et}^{\mathrm{an}}(W_2) \otimes_{\hat{\mbb{Q}}_p} \ordd \mathbb{B}_{dR, S^{\mathrm{an}}}$. Fix a presentation 
\[
W_2 = e \left( \bigoplus_{i=1}^k W_1^{\otimes a_i} \otimes (W_1^*)^{\otimes b_i} \right)
\]
for some positive integers $a_i, b_i$ and idempotent $e$. Since $\xi_{1,B}$, $\xi_{1, \opn{dR}}$ and $\xi_{1, \et}$ factor through a functor valued in relative Chow motives, we obtain idempotents $e_B$, $e_{\opn{dR}}$, $e_{\et}$ in the respective target categories which are all compatible under the comparison isomorphisms. 

Let $A_1$ and $A_2$ denote the universal abelian varieties over $S_1$ and $S_2$, and let $f^*A_2$ denote the pullback to $S_1$. For $? = B, \opn{dR}, \et$, the isomorphisms $\eta_{?}(W_2)$ are described by isomorphisms
\[
\xi_{1, ?}(W_2) \cong e_? \left( \bigoplus_{i=1}^k \invs{W}_{1, ?}^{\otimes a_i} \otimes (\invs{W}_{1, ?}^{\vee})^{\otimes b_i} \right) \xrightarrow{\sim} \invs{H}^{?}_1 \left( f^* A_2 / S_1 \right) \cong f^* \invs{H}^?_1 \left( A_2/S_2 \right)
\]
where $\invs{H}^?_1(\cdots)$ denotes first relative homology of the appropriate cohomology theory and the last isomorphism is proper base-change.

We just need to check the middle isomorphism is compatible under the de Rham--\'{e}tale comparison isomorphism. It is enough to check this at points of $S_1$ which are defined over number fields (c.f. the proof of \cite[Proposition 2.3.9]{CS17})) -- let $\theta_?$ denote the middle isomorphism specialised at such a point. By above, we know that $\theta_B$ and $\theta_{\opn{dR}}$ are compatible under the Betti--de Rham comparison, and that $\theta_B$ and $\theta_{\et}$ are compatible under the Betti--\'{e}tale comparison. The result now follows from the fact that $\theta_B$ can be represented as a Hodge class (by using the polarisation and K\"{u}nneth formula) for an abelian variety constructed from copies of $A_1$ and $f^*A_2$. Indeed, by \cite{DeligneHodgeCycles} it is an absolute Hodge class whose de Rham realisation is defined over the field of definition of the point (by the paragraph preceding the proposition). By \cite{Blasius94}, this Hodge class is de Rham, which precisely means that $\theta_{\opn{dR}}$ and $\theta_{\et}$ are compatible under the de Rham--\'{e}tale comparison, as required.
\end{proof}

\begin{remark} \label{FramesOfPullbackRem}
One can show that the pushout $\mathscr{G}_{1, \opn{dR}} \times^{\mathscr{G}_1} \mathscr{G}_2$ is identified with frames of $\invs{H}^{\opn{dR}}_1\left( f^*A_2 / S_1 \right)$ preserving a collection of Hodge tensors coming from a choice of $\mathscr{G}_1$-equivariant embedding $W_2^{\otimes} \subset W_1^{\otimes}$ and the isomorphism $\theta_{\opn{dR}}$ above. The isomorphism $\mathscr{G}_{1, \opn{dR}} \times^{\mathscr{G}_1} \mathscr{G}_2 \to f^*\mathscr{G}_{2, \opn{dR}}$ is then induced from the proper base-change isomorphism $\invs{H}^{\opn{dR}}_1\left( f^*A_2 / S_1 \right) \cong f^* \invs{H}^{\opn{dR}}_1\left(A_2/S_2 \right)$, which matches the Hodge tensors. A similar description also holds for the \'{e}tale and Betti constructions.
\end{remark}

Let $L/F$ be a finite extension and let $\mu_i \colon \mbb{G}_{m, L} \to \mathscr{G}_{i, L}$ be a choice of Hodge cocharacter for the Shimura datum $(\mathscr{G}_i, X_i)$, for $i=1, 2$. We assume that $\mu_2 = f \circ \mu_1$. Fix a prime $\ide{P}$ of $L$ lying above $\ide{p}$, and we base-change the Shimura varieties $S_1$ and $S_2$ to $L_{\ide{P}}$ (but omit this from the notation). For $i=1, 2$ and over $L_{\ide{P}}$, we have two parabolics $\mathscr{P}_i^{\opn{std}}$ and $\mathscr{P}_i$, with common Levi $\mathscr{M}_i$, associated with $\mu_i$. We have pro\'{e}tale torsors over $S_i$ given by:
\begin{align*}
    \invs{P}_{i, \opn{dR}}^{\opn{an}}(U) &\defeq \left\{ \hat{\ordd}_{S_i} \otimes W_i|_U \xrightarrow{\sim} \invs{W}_{i, \opn{dR}}^{\opn{an}} \otimes_{\ordd_{S_i}} \hat{\ordd}_{S_i}|_U : \begin{array}{c} \text{ preserving Hodge filtration } \\ \text{ and Hodge tensors } \end{array} \right\} \\
    \invs{P}_{i, \opn{HT}}^{\opn{an}}(U) &\defeq \left\{ \hat{\ordd}_{S_i} \otimes W_i|_U \xrightarrow{\sim} \invs{W}_{i, \et}^{\opn{an}} \otimes_{\hat{\mbb{Q}}_p} \hat{\ordd}_{S_i}|_U : \begin{array}{c} \text{ preserving Hodge--Tate filtration } \\ \text{ and Hodge tensors } \end{array} \right\}
\end{align*}
where $W_i$ is the defining representation of $\mathscr{G}_i$. These are $\mathscr{P}^{\opn{std}}_{i}$ and $\mathscr{P}_i$ torsors respectively. We denote by $\invs{M}_{i, \opn{dR}}^{\opn{an}}$ and $\invs{M}_{i, \opn{HT}}^{\opn{an}}$ their pushouts to $\mathscr{M}_i$. Then the results of \cite{CS17} imply that $\invs{M}_{i, \opn{dR}}^{\opn{an}} \cong {^\mu \invs{M}_{i, \opn{HT}}^{\opn{an}}}$, where the twist is along $\mu_i$ as in \S \ref{TwistingTorsorsSubSec}. Note that $f(\mathscr{P}_1^{\opn{std}}) \subset \mathscr{P}_2^{\opn{std}}$, $f(\mathscr{P}_1) \subset \mathscr{P}_2$ and $f(\mathscr{M}_1) \subset \mathscr{M}_2$ by the assumption that $\mu_2 = f \circ \mu_1$.

\begin{corollary} \label{FunctOfMTorsorsCor}
We have a commutative diagram of torsors:
\[
\begin{tikzcd}
{\invs{M}_{1, \opn{dR}}^{\opn{an}} \times^{\mathscr{M}_1} \mathscr{M}_2} \arrow[r] \arrow[d] & {f^* \invs{M}^{\opn{an}}_{2, \opn{dR}}} \arrow[d]        \\
{{^\mu \invs{M}_{1, \opn{HT}}^{\opn{an}}} \times^{\mathscr{M}_1} \mathscr{M}_2} \arrow[r]         & {f^* \left( {^\mu \invs{M}^{\opn{an}}_{2, \opn{HT}}} \right)}
\end{tikzcd}
\]
where the horizontal arrows are induced from the natural transformations $\eta_{\opn{dR}}^{\opn{an}}$ and $\eta_{\et}^{\opn{an}}$, and the vertical arrows are induced from the isomorphism of de Rham and twisted Hodge--Tate torsors above.
\end{corollary}
\begin{proof}
Let $\pi \colon A_2 \to S_2$ denote the universal abelian variety and $f^{-1}\pi \colon f^*A_2 \to S_1$ its pullback under $f$. To simplify notation, set $\invs{E} \defeq \invs{H}^1_{\opn{dR}}(A_2/S_2)$, $\invs{E}' \defeq \invs{H}^1_{\opn{dR}}(f^*A_2/S_1)$, $\mbb{L} \defeq R^1\pi_* \hat{\mbb{Z}}_{p, A_2}$ and $\mbb{L}' \defeq R^1(f^{-1}\pi)_* \hat{\mbb{Z}}_{p, f^*A_2}$. By Proposition \ref{dRpadicComparisonProp} and Remark \ref{FramesOfPullbackRem}, we know that the following diagram commutes:
\[
\begin{tikzcd}
{f^*\left(\invs{E} \otimes_{\ordd_{S_2}} \ordd\mbb{B}_{\opn{dR}, S_2} \right)} \arrow[d] \arrow[r, equal] & {f^* \invs{E} \otimes_{\ordd_{S_1}} \ordd \mbb{B}_{\opn{dR}, S_1}} \arrow[r, "\sim"]    & {\invs{E}' \otimes_{\ordd_{S_1}} \ordd\mbb{B}_{\opn{dR}, S_1}} \arrow[d] \\
{f^*\left( \mbb{L} \otimes_{\hat{\mbb{Z}}_p} \ordd\mbb{B}_{\opn{dR}, S_2} \right)} \arrow[r, equal]       & {f^* \mbb{L} \otimes_{\hat{\mbb{Z}}_p} \ordd \mbb{B}_{\opn{dR}, S_1}} \arrow[r, "\sim"] & {\mbb{L}' \otimes_{\hat{\mbb{Z}}_p} \ordd\mbb{B}_{\opn{dR}, S_1}}       
\end{tikzcd}
\]
where the horizontal arrows are the proper base-change isomorphisms and the vertical arrows (which are isomorphisms) arise from the comparisons of relative $p$-adic Hodge theory. The module $\invs{E} \otimes_{\ordd_{S_2}} \ordd\mbb{B}_{\opn{dR}, S_2}^+$ is an $\ordd \mbb{B}^+_{\opn{dR}}$-module with an integrable connection, so satisfies
\[
\invs{E} \otimes_{\ordd_{S_2}} \ordd\mbb{B}_{\opn{dR}, S_2}^+ = \mbb{M}_0 \otimes_{\mbb{B}^+_{\opn{dR, S_2}}} \ordd \mbb{B}_{\opn{dR, S_2}}
\]
where $\mbb{M}_0 = \left(\invs{E} \otimes_{\ordd_{S_2}} \ordd\mbb{B}_{\opn{dR}, S_2}^+\right)^{\nabla = 0}$ (see \cite[Theorem 7.2]{Scholze_2013}). Hence
\[
f^* \left( \invs{E} \otimes_{\ordd_{S_2}} \ordd\mbb{B}_{\opn{dR}, S_2}^+ \right)^{\nabla = 0} = \left( f^* \invs{E} \otimes_{\ordd_{S_1}} \ordd \mbb{B}_{\opn{dR}, S_1} \right)^{\nabla = 0} = f^*\mbb{M}_0 .
\]
Set $\mbb{M}_0' = \left(\invs{E}' \otimes_{\ordd_{S_1}} \ordd\mbb{B}_{\opn{dR}, S_1}\right)^{\nabla = 0}$, $\mbb{M} = \mbb{L} \otimes_{\hat{\mbb{Z}}_p} \ordd \mbb{B}^+_{\opn{dR}, S_2}$ and $\mbb{M}' = \mbb{L}' \otimes_{\hat{\mbb{Z}}_p} \ordd \mbb{B}^+_{\opn{dR}, S_1}$. Since the base-change maps are compatible with structures, they induce isomorphisms
\begin{eqnarray}
f^* \mbb{M}_0 & \xrightarrow{\sim} & \mbb{M}_0' \label{BoldM0Iso} \\
f^* \mbb{M} & \xrightarrow{\sim} & \mbb{M}' \label{BoldMIso}
\end{eqnarray}
and we have a commutative diagram:
\[
\begin{tikzcd}
{f^*\mbb{M}_0 \otimes_{\mbb{B}^+_{\opn{dR}, S_1}} \mbb{B}_{\opn{dR}, S_1}} \arrow[d] \arrow[r, "(\ref{BoldM0Iso}) \otimes 1"] & {\mbb{M}_0' \otimes_{\mbb{B}^+_{\opn{dR}, S_1}} \mbb{B}_{\opn{dR}, S_1}} \arrow[d] \\
{f^*\mbb{M} \otimes_{\mbb{B}^+_{\opn{dR}, S_1}} \mbb{B}_{\opn{dR}, S_1}} \arrow[r] \arrow[r, "(\ref{BoldMIso}) \otimes 1"]    & {\mbb{M}' \otimes_{\mbb{B}^+_{\opn{dR}, S_1}} \mbb{B}_{\opn{dR}, S_1}}            
\end{tikzcd}
\]
where the vertical arrows are as in \cite[Proposition 2.2.3]{CS17}. In particular, considering the relative Hodge filtration as in \emph{op.cit.} and passing to gradeds, we have 
\[
\begin{tikzcd}
f^* \opn{gr}^j\invs{E} \arrow[d] \arrow[r]                                                      & \opn{gr}^j \invs{E}' \arrow[d]                                     \\
f^*\left( \opn{gr}_j (\mbb{L} \otimes_{\hat{\mbb{Z}}_p} \hat{\ordd}_{S_2})(j) \right) \arrow[r] & \opn{gr}_j (\mbb{L}' \otimes_{\hat{\mbb{Z}}_p} \hat{\ordd}_{S_1} )(j)
\end{tikzcd}
\]
for all $j \geq 0$. Note that the pullback $f^*$ preserves the relevant filtrations because each graded piece is locally free. This last commutative diagram (or more precisely its dual version) describes the compatibility we desire in the statement of the corollary. Indeed, the isomorphism $\theta_{\opn{dR}}$ in the proof of Proposition \ref{dRpadicComparisonProp} preserves Hodge filtrations, and by a similar argument above, one can show $\theta_{\et}$ preserves relative Hodge--Tate filtrations. Therefore the pushouts $\invs{P}^{\opn{an}}_{1, \opn{dR}} \times^{\mathscr{P}_1^{\opn{std}}} \mathscr{P}_2^{\opn{std}}$ and $\invs{P}^{\opn{an}}_{1, \opn{HT}} \times^{\mathscr{P}_1} \mathscr{P}_2$ can be described as frames of $\invs{H}^{\opn{dR}}_1\left(f^*A_2/S_1 \right)$ and $\invs{H}^{\et}_1\left(f^*A_2/S_1 \right)$ respectively, preserving Hodge tensors and filtrations.
\end{proof}

\subsection{Perfectoid Shimura varieties} \label{AppendixPerfectoidShimuraVarieties}

Continuing with the set-up as in the previous subsection, assume that $K_i$ is of the form $K_i^p \times K_{i, p} \subset \mathscr{G}_i(\mbb{A}_f^p) \times \mathscr{G}_i(\mbb{Q}_p)$ for $i=1, 2$. Let $\invs{S}_1$ and $\invs{S}_2$ denote the adic spaces over $F_{\ide{p}}$ associated with $S_1$ and $S_2$, and let $\invs{S}_{i, \infty}$ denote the perfectoid Shimura variety (of tame level $K_i^p$), as constructed in \cite{ScholzeTorsion}. Then \cite[Theorem 1.10]{CS17}, implies that we have a commutative diagram
\[
\begin{tikzcd}
{\invs{S}_{1, \infty}} \arrow[d] \arrow[r, "{\pi_{\mathrm{HT}, 1}}"] & \mathtt{FL}_1 \arrow[d] \\
{\invs{S}_{2, \infty}} \arrow[r, "{\pi_{\mathrm{HT}, 2}}"]           & \mathtt{FL}_2          
\end{tikzcd}
\]
where $\mathtt{FL}_i$ denotes the adic flag variety associated with the Shimura datum $(\mathscr{G}_i, X_i)$ and $\pi_{\mathrm{HT}, i}$ is the corresponding Hodge--Tate period morphism. Both of the vertical maps are induced from $f$. 

For $i=1, 2$ consider the torsor $\mathtt{FL}_i \times \mathscr{G}_i(\mbb{Q}_p)$ with the right action of $\mathscr{G}_i(\mbb{Q}_p)$ given by $(x, g) \cdot g' = (xg', (g')^{-1}g)$. We then obtain a torsor
\[
\pi_{\mathrm{HT}}^*(\mathtt{FL}_i \times \mathscr{G}_i(\mbb{Q}_p)) /K_{i, p} = \invs{S}_{i, \infty} \times^{K_{i, p}} \mathscr{G}_i(\mbb{Q}_p)
\]
over $\invs{S}_i$. By the description of $\xi_{\et}$ as above, and the fact that $\invs{S}_{i, \infty}$ is essentially the limit $\varprojlim_{K_{i, p}'} \invs{S}_{\mathscr{G}_i, K_i^pK_{i, p}'}$, this torsor encodes $\xi_{\et}^{\mathrm{an}}$. 

We have a natural map of torsors $\mathtt{FL}_1 \times \mathscr{G}_1(\mbb{Q}_p) \to \mathtt{FL}_2 \times \mathscr{G}_2(\mbb{Q}_p)$ induced from $f$, which is compatible with the equivariant structure. Pulling back along the Hodge--Tate period morphism and descending, we obtain a natural transformation $\eta' \colon \xi_{1, \et} \circ f^* \to f^* \circ \xi_{2, \et}$.

\begin{lemma} \label{EtaleAndFlagCoincide}
The natural transformations $\eta_{\et}^{\mathrm{an}}$ and $\eta'$ coincide.
\end{lemma}
\begin{proof}
This follows from the above commutative diagram and (on the level of topological spaces) the map $\invs{S}_{1, \infty} \to \invs{S}_{2, \infty}$ is the inverse limit of (the analytification of) maps $S_{\mathscr{G}_1, K'_1} \to S_{\mathscr{G}_2, K'_2}$. 
\end{proof}

%%%%%%%%%%%%%%%%%%%%%%%%%%%%%%%%%%%%%%%%%%%%%%%%%%%%%%%%%%%%%%%%%%%%%%%%%%%%%%%%%%%%%%%%%%%%%
%%%%%%      APPENDIX C      %%%%%%%%%%%%%%%%%%%%%%%%%%%%%%%%%%%%%%%%%%%%%%%%%%%%%%%%%%%%%%%%%
%%%%%%%%%%%%%%%%%%%%%%%%%%%%%%%%%%%%%%%%%%%%%%%%%%%%%%%%%%%%%%%%%%%%%%%%%%%%%%%%%%%%%%%%%%%%%

\section{Unitary base change}

In this appendix, we describe how the results on endsocopic classification of unitary groups in \cite{Mok} and \cite{KMSW14} imply a certain strong multiplicity one theorem for automorphic representations of $\mbf{G}(\mbb{A})$. Note that these cited papers are conditional on the stabilisation of the trace formula for unitary groups. Throughout, we let $\mbf{G}$ and $\mbf{G}_0$ be as in \S \ref{PreliminarySection}, and we write $U$ for the unitary group over $F^+$ associated with $W$ (so $\mbf{G}_0 = \opn{Res}_{F^+/\mbb{Q}}U$). As usual, we assume that $F$ contains an imaginary quadratic number field $E$.

\begin{lemma} \label{ParahoricRestrictionLemma}
Let $\ell$ be any (finite) rational prime. Then there exists a good special maximal compact open subgroup $K \subset \mbf{G}(\mbb{Q}_\ell)$ (as in \cite[\S 2.1]{Minguez}) such that the intersection $K \cap \mbf{G}_0(\mbb{Q}_\ell) \subset \mbf{G}_0(\mbb{Q}_\ell)$ is a good special maximal compact open subgroup. Furthermore, if $\mbf{G}_{\mbb{Q}_\ell}$ is unramified, we can arrange it so that both $K$ and $K \cap \mbf{G}_0(\mbb{Q}_\ell)$ are hyperspecial.
\end{lemma}

\begin{proof}
This follows from the fact that:
\begin{itemize}
    \item $\mbf{G}_0$ and $\mbf{G}$ have the same adjoint group.
    \item The induced map from the Kottwitz group of $\mbf{G}_0$ to that of $\mbf{G}$ is injective. More precisely, the Kottwitz group of the former is $\mbb{Z}/2\mbb{Z}$, of the latter is $\mbb{Z} \times  \mbb{Z}/2\mbb{Z}$ and the induced map is inclusion into the second factor.
\end{itemize}
One then applies the description of all parahoric subgroups as in \cite{HainesRapoport}.
\end{proof}

Let $\ell$ be a finite rational prime. Then the results of \cite{Mok}, \cite{KMSW14} imply that there exists a local (standard) base change map $\opn{BC}_\ell$ from irreducible admissible representations of $\mbf{G}_0(\mbb{Q}_\ell) \cong \prod_{v | \ell} U(F^+_v)$ to irreducible admissible representations of $\mbf{G}_0(\mbb{Q}_\ell \otimes_{\mbb{Q}} E) \cong U(\mbb{Q}_\ell \otimes_{\mbb{Q}} F) \cong \prod_{w | \ell} \opn{GL}_{2n}(F_w)$ (this can be defined unconditionally if all primes above $\ell$ split in $F/F^+$, or if the group and representation are both unramified).

\begin{lemma} \label{LocalMultOneLemma}
Let $\ell$ be an odd rational prime. Let $K \subset \mbf{G}(\mbb{Q}_\ell)$ be a good special maximal compact open subgroup as in Lemma \ref{ParahoricRestrictionLemma}, and let $\pi$ and $\sigma$ be irreducible admissible unitary representations of $\mbf{G}(\mbb{Q}_\ell)$. Suppose that:
\begin{itemize}
    \item There exist irreducible admissible unitary representations $\pi_0$ and $\sigma_0$ of $\mbf{G}_0(\mbb{Q}_\ell)$ such that $\pi_0 \subset \pi|_{\mbf{G}_0(\mbb{Q}_\ell)}$, $\sigma_0 \subset \sigma|_{\mbf{G}_0(\mbb{Q}_\ell)}$ and $\opn{BC}_\ell(\pi_0) \cong \opn{BC}_\ell(\sigma_0)$
    \item Both $\pi^K \neq 0$ and $\sigma^K \neq 0$.
\end{itemize}
Then $\pi \cong \sigma$ or $\pi \cong \sigma \otimes \omega$, where $\omega$ is the quadratic character associated with $E \otimes_{\mbb{Q}} \mbb{Q}_\ell / \mbb{Q}_{\ell}$. In particular, if $\ell$ ramifies or splits in $E/\mbb{Q}$, then $\pi \cong \sigma$.
\end{lemma}

\begin{proof}
Suppose for the moment that $\pi$ and $\sigma$ share an irreducible constituent under the action of $\mbf{G}_0(\mbb{Q}_\ell)$. Then \cite[Proposition 4.1.3]{LabesseSchwermer} implies that $\pi \cong \sigma \otimes \chi$ for some character of the quotient $\mbf{G}_0(\mbb{Q}_\ell)\mbf{Z}_{\mbf{G}}(\mbb{Q}_\ell) \backslash \mbf{G}(\mbb{Q}_\ell)$. But this quotient is contained in $N((E \otimes \mbb{Q}_\ell)^{\times}) \backslash \mbb{Q}_\ell^{\times}$ via the similitude character, where $N$ denotes the norm map, hence $\chi$ is either the trivial character or the quadratic character $\omega$. If $\ell$ is split then $\omega = 1$, otherwise if $\ell$ is ramified, then $\omega$ is ramified. But since $\pi$ and $\sigma$ are $K$-spherical, they cannot be isomorphic via a ramified twist. The latter is true because the image of the $\mbb{Q}_\ell$-points of a Levi of a minimal parabolic in $\mbf{G}_{\mbb{Q}_{\ell}}$ under the similitude map contains $\mbb{Z}_{\ell}^{\times}$ (by the structure of \emph{even-dimensional} unitary groups in \cite[Example 3.2]{Minguez} and that any non-trivial quadratic form in two or more variables represents every element of $\mbb{F}_{\ell}^{\times}$), and the fact that the intersection of the Levi with a good maximal special subgroup is the unique maximal compact open subgroup (see \S 2.1 in \emph{op.cit.}).  

We now show that $\pi$ and $\sigma$ share an irreducible constituent. Since $\pi^K \neq 0$, there exists an irreducible constituent $\pi_0' \subset \pi|_{\mbf{G}_0(\mbb{Q}_\ell)}$ which has non-trivial invariants under $K_0 \defeq K \cap \mbf{G}_0(\mbb{Q}_\ell)$. Since $K_0$ is a good special maximal compact open subgroup, $\pi_0'$ has a set of associated Satake parameters which is determined from the Satake parameters for $\pi$. Hence $\pi_0$ and $\pi_0'$ have the same set of Satake parameters (but are spherical for different choices of special subgroups). This implies that $\opn{BC}_\ell(\pi_0) \cong \opn{BC}_\ell(\pi_0')$. By a similar argument for $\sigma$, we may replace $\pi_0$ and $\sigma_0$ by $\pi_0'$ and $\sigma_0'$ respectively. Now we note that the base-change map on Langlands/Arthur parameters is injective (\cite[\S 2.2]{Mok}) hence $\pi_0'$ and $\sigma_0'$ have the same Satake parameters, as required.
\end{proof}

Now we discuss a global application of this lemma. Let $S$ be a finite set of rational primes which split in $E/\mbb{Q}$. Let $K = K^S \times K_S \subset \mbf{G}(\mbb{A}_f^S) \times \mbf{G}(\mbb{A}_S)$ be a compact open subgroup such that $K^S = \prod_{\ell \not\in S} K_\ell$ with each $K_{\ell} \subset \mbf{G}(\mbb{Q}_\ell)$ a good special maximal compact open subgroup, which is hyperspecial if $\mbf{G}_{\mbb{Q}_{\ell}}$ is unramified. Let $T$ denote a cofinite set of rational primes containing $2$ and all primes which are inert in $E/\mbb{Q}$.

\begin{proposition} \label{GlobalMultOneProp}
Let $\pi$ and $\sigma$ be cuspidal automorphic representations of $\mbf{G}(\mbb{A})$ such that $\pi_{\infty}$ and $\sigma_{\infty}$ are cohomological and $\pi_f^K \neq 0$ and $\sigma_f^K \neq 0$. Suppose that $\pi_\ell \cong \sigma_\ell$ for all $\ell \in T - (S \cap T)$. Also, suppose that the weak base-change of $\pi$ to $\opn{GL}_1(\mbb{A}_E) \times \opn{GL}_{2n}(\mbb{A}_F)$ is cuspidal. Then $\pi_f \cong \sigma_f$.
\end{proposition}

\begin{proof}
Since $\pi_{\infty}$ and $\sigma_{\infty}$ are cohomological, they admit weak base-changes by \cite{ShinAppendix}. These weak base-changes must be isomorphic by our assumptions and strong multiplicity one, and hence also cuspidal by assumption. By \cite[Theorem 5.2.1]{LabesseSchwermer}, we can find cuspidal automorphic representations $\pi_0$ and $\sigma_0$ of $\mbf{G}_0(\mbb{A})$ such that $\pi_0 \subset \pi|_{\mbf{G}_0(\mbb{A})}$ and $\sigma_0 \subset \sigma|_{\mbf{G}_0(\mbb{A})}$. By compatibility of base-change for unitary and unitary similitude groups, the weak-base changes of $\pi_0$ and $\sigma_0$ are isomorphic (and cuspidal). Call the common representation $\Pi_0$. By the theory of endoscopy (see \cite[Proposition C.3.1(2)]{LTXZZ19}), we actually have the stronger compatibility $\opn{BC}_{\ell}(\pi_{0, \ell}) \cong \Pi_{0, \ell} \cong \opn{BC}_{\ell}(\sigma_{0, \ell})$ for all rational primes $\ell$. Then
\begin{enumerate}
    \item If $\ell \in S$, then the weak base-changes of $\pi$ and $\sigma$ are locally isomorphic at $\ell$ (\cite[Theorem A.1(2)]{ShinAppendix}). Since local base-change is injective at these primes, we have $\pi_{\ell} \cong \sigma_{\ell}$.
    \item If $\ell \not\in S \cup \{ 2\}$ and ramifies or splits in $E/\mbb{Q}$, then we have $\pi_{\ell} \cong \sigma_{\ell}$ by Lemma \ref{LocalMultOneLemma}.
    \item If $\ell \not\in S$ and is inert in $E/\mbb{Q}$, or $\ell = 2$, then $\ell \in T - (S \cap T)$ and $\pi_{\ell} \cong \sigma_{\ell}$ by assumption.
\end{enumerate}
This completes the proof.
\end{proof}

%%%%%%%%%%%%%%%%%%%%%%%%%%%%%%%%%%%%%%%%%%%%%%%%%%%%%%%%%%%%%%%%%%%%%%%%%%%%%%%%%%%%
%%%%%%    BIBLIOGRAPHY     %%%%%%%%%%%%%%%%%%%%%%%%%%%%%%%%%%%%%%%%%%%%%%%%%%%%%%%%%
%%%%%%%%%%%%%%%%%%%%%%%%%%%%%%%%%%%%%%%%%%%%%%%%%%%%%%%%%%%%%%%%%%%%%%%%%%%%%%%%%%%%

\newcommand{\etalchar}[1]{$^{#1}$}
\renewcommand{\MR}[1]{}
\providecommand{\bysame}{\leavevmode\hbox to3em{\hrulefill}\thinspace}
\providecommand{\MR}{\relax\ifhmode\unskip\space\fi MR }
\providecommand{\MRhref}[2]{%
  \href{http://www.ams.org/mathscinet-getitem?mr=#1}{#2}
}
\providecommand{\href}[2]{#2}

%%%%%%%%%%%%%%%%%%%%%%%%%%%%%%%%%%%%%%%%%%%%%%%%%%%%%%%%%%%%%%%%%%%%%%%%%%%%%%%%%%
%%%%%    AFFILIATION      %%%%%%%%%%%%%%%%%%%%%%%%%%%%%%%%%%%%%%%%%%%%%%%%%%%%%%%%
%%%%%%%%%%%%%%%%%%%%%%%%%%%%%%%%%%%%%%%%%%%%%%%%%%%%%%%%%%%%%%%%%%%%%%%%%%%%%%%%%%

\Addresses

\end{document}